%% file: projection-free-OCO.tex
\let\oldvec\vec
\documentclass[11pt]{article}
\let\vec\oldvec

\RequirePackage{libertine}
\usepackage[scaled=0.9]{helvet}
\RequirePackage[varqu]{zi4} 
\RequirePackage[libertine, timesmathacc]{newtxmath}
\usepackage{wrapfig}

\input{arxiv_style.tex}

\input{zak.tex}
\input{macros.tex}

\newcommand{\citep}[1]{\cite{#1}}
\usepackage[toc,page,header]{appendix}

\makeatletter

\newcommand*\wthelper[2]{%
	\hbox{\dimen@\accentfontxheight#1%
		\accentfontxheight#11.2\dimen@
		$\m@th#1\widehat{#2}$%
		\accentfontxheight#1\dimen@
	}%
}

\newcommand*\accentfontxheight[1]{%
	\fontdimen5\ifx#1\displaystyle
	\textfont
	\else\ifx#1\textstyle
	\textfont
	\else\ifx#1\scriptstyle
	\scriptfont
	\else
	\scriptscriptfont
	\fi\fi\fi3
}

\let\wtilde\undefined

\makeatletter
\newcommand*\wtilde[1]{\mathpalette\wthelpers{#1}}

\newcommand*\wthelpers[2]{%
	\hbox{\dimen@\accentfontxheight#1%
		\accentfontxheight#11.2\dimen@
		$\m@th#1\widetilde{#2}$%
		\accentfontxheight#1\dimen@
	}%
}

\makeatother

\title{Efficient Projection-Free Online Convex Optimization with Membership Oracle}
\author{{\bf Zakaria Mhammedi}\\ Massachusetts Institute of Technology \\ \texttt{mhammedi@mit.edu}}
 \date{}

\begin{document}

	\maketitle

\begin{abstract}
In constrained convex optimization, existing methods based on the ellipsoid or cutting plane method do not scale well with the dimension of the ambient space. Alternative approaches such as Projected Gradient Descent only provide a computational benefit for simple convex sets such as Euclidean balls, where Euclidean projections can be performed efficiently. For other sets, the cost of the projections can be too high. To circumvent these issues, alternative methods based on the famous Frank-Wolfe algorithm have been studied and used. Such methods use a Linear Optimization Oracle at each iteration instead of Euclidean projections; the former can often be performed efficiently. Such methods have also been extended to the online and stochastic optimization settings. However, the Frank-Wolfe algorithm and its variants do not achieve the optimal performance, in terms of regret or rate, for general convex sets. What is more, the Linear Optimization Oracle they use can still be computationally expensive in some cases. In this paper, we move away from Frank-Wolfe style algorithms and present a new reduction that turns any algorithm $\A$ defined on a Euclidean ball (where projections are cheap) to an algorithm on a constrained set $\K$ contained within the ball, without sacrificing the performance of the original algorithm $\A$ by much. Our reduction requires $O(T \ln T)$ calls to a Membership Oracle on $\K$ after $T$ rounds, and no linear optimization on $\K$ is needed. Using our reduction, we recover optimal regret bounds [resp.~rates], in terms of the number of iterations, in online [resp.~stochastic] convex optimization. Our guarantees are also useful in the offline convex optimization setting when the dimension of the ambient space is large. 
\end{abstract}
	
	\clearpage
	\tableofcontents
	\clearpage
	\section{Introduction}
	In this paper, we are interested in designing efficient algorithms for constrained convexity optimization in the online, offline, and stochastic settings. Popular algorithms for optimizing a convex objective $f$ defined on a bounded convex set $\K\subset\reals^d$ include, for example, the ellipsoid or cutting plane methods \citep{grotschel1993, grotschel2012, bubeck2015}. Though such algorithms enjoy linear convergence rates, where the optimality gap decreases exponentially fast with the number of iterations, their per-iteration computational complexity depends super-linearly in the dimension $d$. In particular, if $\Cost({\cM}_{\K})$ is the computational cost of testing if some point $\x\in \reals^d$ belongs to $\K$, then state-of-the-art cutting plane-based methods have a computational complexity of order $O(d^3 + d^2\Cost({\cM}_{\K})) \log(1/\epsilon)$ to find an $\epsilon$-suboptimal point \cite{lee2015, lee2018efficient}. Thus, when the dimension $d$ is very large, such methods can become impractical.
	
	 For some convex sets such as a Euclidean ball, alternative methods, such as the Projected Gradient Descent (PGD) algorithm, dispense of the ``expensive'' dimension dependence in their computational complexity at the cost of a worse dependence in $1/\epsilon$ (e.g.~$1/\epsilon^2$ instead of $\log(1/\epsilon)$) in their convergence rate. Despite the latter cost, such methods are still favorable in practice when the dimension $d$ is large. We refer the interested reader to \cite{bubeck2015} for a nice comparison between the convergence rates of different optimization algorithms. In some settings, evaluating the objective function $f$ is itself expensive. For example, this is the case in many machine learning applications where the objective involves a sum over a large number of data points. A popular approach around this issue is to use (projected) Stochastic Gradient Descent (SGD), where it suffices to evaluate a stochastic version of the objective function at each iteration, which can often be done efficiently. 
	 SGD enjoys similar computational benefits as PGD. However, when the domain $\K$ of interest is not a Euclidean ball, the main bottleneck of such methods is often the projections that need be performed each time the iterates of the algorithm step out of $\K$. This has fueled a long line of research around the design of, so-called, projection-free algorithms that swap expensive projection steps for linear optimization over the domain $\K$, which can often be done much more efficiently. The reader is referred to the vast literature on projection-free optimization that include, for instance, \cite{frank1956,hazan2008sparse,jaggi2013,garber2016faster,kerdreux2020accelerating, Bomze2021}. 
	
The first projection-free algorithm (a.k.a.~the Frank-Wolfe algorithm) was introduced by \cite{frank1956} who applied it to smooth optimization on polyhedral sets. The Frank-Wolfe algorithm swaps expensive projections on the convex set $\K$ for linear optimization on $\K$, which for many sets of interest can be performed efficiently (see Table \ref{tab:comp}). A similar algorithm was later introduced for the Online Convex Optimization (OCO) setting with linear losses \cite{kalai2005}. This is a setting where instead of optimizing a fixed function, the goal is to minimize a quantity called \emph{regret}. Formally, at each round $t$ in OCO, a learner (algorithm) outputs an iterate $\x_t$ in some convex set $\K$, then observes a convex loss function $\ell_t :\K\rightarrow \reals$ that may be chosen by an adversary based on $\x_t$ and the history up to round $t$. The goal is to minimize the regret, which is defined as the difference between the cumulative loss of the learner and that of any comparator $\x\in \K$: 
	\begin{align}
 \text{Regret}_T(\x)\coloneqq \sum_{t=1}^T \ell_t(\x_t)- \sum_{t=1}^T \ell_t(\x) .\nn
	\end{align}
Offline [resp.~Stochastic] optimization are special cases of OCO, where $\ell_t$ is fixed [resp.~stochastic with a fixed mean] for all $t$. Algorithms designed for OCO (e.g. Online Gradient Descent \cite{zinkevich2003}) often lead to optimal convergence rates when used in the offline and stochastic settings via online-to-batch conversion techniques \cite{cesa2004, shalev2011, cutkosky2019}. For examples of problems that can be modeled via OCO we refer the reader to the nice OCO introduction by \cite{hazan2016introduction}. Unlike in the offline setting where there are algorithms such as those based on the cutting plane method that converge exponentially fast with the number of iterations (though a single iteration might be computationally expensive), there are no equivalent options for OCO since the losses are different at each round and the goal is to minimize the regret. This means that designing efficient projection-free algorithms for OCO is even more crucial than in the offline setting.

In this paper, we continue the long line of research in the design of efficient projection-free algorithms for online and stochastic optimization (see e.g.~\cite{kalai2005,hazan2012, hazana16, hazan2020}). We depart from the Frank-Wolfe-style approach by presenting a new algorithm that requires a Membership Oracle $\cM_\K$ for $\K$ instead of a Linear Optimization Oracle. When given $\x\in \reals^d$, the Membership Oracle returns $\mathbb{I}\{\x\in \K\}$. We present an algorithm that requires a total of $O(T\ln T)$ calls to a (possibly approximate) Membership Oracle and guarantees a regret bound of $O(\sqrt{T})$ for general convex functions and convex sets. 
This regret is optimal in the number of rounds $T$. For the general OCO setting (without additional assumptions on $\K$ or the losses), no existing Frank-Wolfe-style algorithm guarantees such a dependence in $T$ in its regret bound using the same number of calls to a Linear Optimization Oracle for $\K$. 

Given that Membership Oracles are typically used to implement Linear Optimization Oracles in the first place \cite{grotschel2012, bubeck2015,lee2018efficient}, it may appear that our approach is strictly better than those based on linear optimization. Though this may be the case from the regret bound's perspective, it is not always the case when it comes to the computational complexity of the Oracles. This is because there are convex sets $\K$ on which linear optimization is cheaper than testing for membership, and vice versa. We summarize the Oracle complexities for different sets in Table \ref{tab:comp}. It is known that the complexity of a Membership Oracle $\cM_{\K}$ is essentially equivalent to the complexity of linear optimization on $\K^\circ$, the polar set of $\K$ (definition in the next section). In fact, linear optimization on $\K^\circ$ (which is what we do as part of our new algorithm) can be performed up to error $\epsilon$ using $O(\log (1/\epsilon))$ calls to $\cM_{\K}$. Conversely, $\cM_{\K}$ can be implemented using a single call to a Linear Optimization Oracle on $\K^\circ$. Since $(\K^\circ)^\circ=\K$ for a closed convex set $\K$, this also means that linear optimization on $\K$ can be viewed as testing membership/performing separation on $\K^\circ$. 
Therefore, when it comes to the computational complexity of the Oracles, Frank-Wolfe-style algorithms and ours may be viewed as complementary.
\begin{table}[h]
\centering
\begin{tabular}{lllll}
\hline
\bf{Domain $\K$} & \multicolumn{2}{l}{Operation Required by } &\multicolumn{2}{l}{Computational Complexity of} \\
&LOO $\cO_{\K}(\x;\cdot )$  & {\color{blue}MO $\cM_{\K}(\x;\cdot)$}  & LOO $\cO_{\K}(\cdot;\delta)$ & {\color{blue} MO $\cM_{\K}(\cdot;\delta)$ }\\
\hline 
\hline
\hspace{-0.2cm}  $\ell_p$-ball in $\reals^d$	&  $ \|\x\|_{q}$ &  {\color{blue}$\|\x\|_{p}$}  &   $O(d)$ & {\color{blue}$O(d)$} \\ 
\hline
\hspace{-0.2cm}  Simplex $\Delta_d \in\reals^d$ & $ \|\x\|_{\infty}$ & {\color{blue} $ \inner{\mathbf{1}}{\x} $}   & $O(d)$ &
 {\color{blue}$O(d)$}.\\
\hline
\hspace{-0.2cm}  Trace-norm-ball in $\reals^{m\times n}$ & $\|\x\|_{\op}$ & {\color{blue} $\|\x\|_{\text{tr}}$} & $O (\text{nnz}(\x)/\sqrt{\delta})$  & {\color{blue} \Cost(SVD)} \\
\hline
\hspace{-0.2cm}    Op-norm-ball in $\reals^{m\times n}$ & $\|\x\|_{\text{tr}}$ &  {\color{blue}$\|\x\|_{{\op}}$}&  \Cost(SVD) & {\color{blue}$O (\text{nnz}(\x)/\sqrt{\delta})$} \\
\hline
\hspace{-0.4cm}  \begin{tabular}{l} Conv-hull of Permutation \\ \quad Matrices in $\reals^{n\times n}$ \end{tabular}&  &  {\color{blue}  \hspace{-0.3cm} \begin{tabular}{l} $\e_i^\top \x \mathbf{1}=1$;  \\  $\e_i^\top \x^\top \mathbf{1}=1$,\end{tabular} \hspace{-0.3cm}$\forall i$} & $O(n^3)$  & {\color{blue}$O(n^2)$} \\ 
\hline
\hspace{-0.4cm}  \begin{tabular}{l} Convex-hull of Rotation \\ \quad Matrices in $\reals^{n\times n}$ \end{tabular}&  &   & \Cost(SVD)& {\color{blue} \Cost(SVD)} \\
\hline
\hspace{-0.4cm}   \begin{tabular}{l} PSD matrices in $\reals^{n\times n}$\\ \quad with unit trace
 \end{tabular} &  $\lambda_{\max} (\x)$ &  {\color{blue} \hspace{-0.3cm} \begin{tabular}{l} $\lambda_{\min}(\x)$\\  $\text{tr}(\x)$
\end{tabular}}  & $O(\text{nnz}(\x) /\sqrt{\delta})$ & {\color{blue}$O(\text{nnx}(\x)/\sqrt{\delta})$} \\
\hline
\hspace{-0.4cm}  \begin{tabular}{l} PSD matrices in $\reals^{n\times n}$\\ \quad with diagonals $\leq$ 1
  \end{tabular} &   &   \hspace{-0.3cm}  {\color{blue} \begin{tabular}{l} $\lambda_{\min}(\x)$\\  $\max_{i\in[n]}(x_{ii})$
\end{tabular} } & $O(\text{nnz}(\x) \sqrt{n^{3} /\delta^{5}})$& {\color{blue}$O(\text{nnx}(\x)/\sqrt{\delta})$} \\
\hline
\hspace{-0.2cm}   \hspace{-0.4cm}  \begin{tabular}{l} The flow polytope with \\ \quad (\#nodes,\ \#edges)=($d$,$m$)
\end{tabular} &   &  & $\wtilde O(d+m)$& {\color{blue}$O(d+m)$} \\
\hline
\hspace{-0.2cm}   \hspace{-0.4cm}  \begin{tabular}{l} The Matroid polytope for \\ \quad  Matroid $M$; \#elem.=$d$
\end{tabular} &   &  & $\wtilde O(d \Cost(\cI_M))$& {\color{blue}\hspace{-0.3cm}   \begin{tabular}{l}  $O(d^2 \Cost(\cI_M)$\\  \quad   $+d^3)$
\end{tabular} } \\
 \hline
 \hline
	\end{tabular}
	\caption{Computational complexity of performing linear optimization [resp.~testing membership] for different sets of interest. $\cO_{\K}$ [resp.~$\cM_{\K}$] denotes a Linear Optimization Oracle (LOO) [resp.~Membership Oracle (MO)]. $\delta>0$ represents the allowed Oracle error (see Section \ref{sec:prelim}). We hide any logarithmic dependence in $1/\delta$. $\Cost$(SVD) [resp.~$\Cost(\cI_{M})$] represents the computational cost of performing SVD [resp.~testing if a set is independent in $M$ (see Section \ref{sec:applications})]. For $\x\in \reals^{n\times m}$, $\text{nnz}(\x)$ represents the number of non-zeros of $\x$. The details for the computational cost of the Linear Optimization Oracles listed can be found in \cite{hazan2012,jaggi2013}. The details for the Membership Oracle complexities can be found in Section \ref{sec:applications}.}
	\label{tab:comp}
\end{table}
\paragraph{Contributions.} For OCO with general convex losses, we provide an algorithm in the form of a (projection-free) reduction that achieves a $O(\kappa\sqrt{dT})$ regret after $T$ rounds using a total of $O(T \ln T)$ calls to an approximate Membership Oracle $\mem_{\K}$, for some constant $\kappa >0$ that depends on the set $\K$. When $\K$ has a non-empty interior and contains the origin, $\kappa=R/r$ such that $\cB(r)\subseteq \K \subset \cB(R)$, where $\cB(x)$ denotes the Euclidean ball of radius $x$. In many practical settings, we show that $\kappa$ is at most $d^{1/2}$ (see Section \ref{sec:applications}). Furthermore, when using $O(d \ln T)$ calls to $\mem_{\K}$ per round (instead of $O(\ln T)$ calls) or $O(\ln T)$ calls to a Separation Oracle on $\K$ per round, our algorithm guarantees a $O(\kappa\sqrt{T})$ regret bound. Thus, our approach allows a trade-off between computation (i.e.~Oracle calls) and regret. 
We also present an algorithm for the case where the losses are strongly convex. The algorithm in question achieves a logarithm regret without requiring the parameter of strong convexity as input, while ensuring a $\wtilde{O}(\kappa\sqrt{dT})$ regret in the worst case. Crucially, the algorithms just mentioned are scale-invariant in the sense that the outputs of the algorithms do not change if the losses are scaled by some positive constant unknown to the algorithms---a desirable property in OCO \cite{cesa2007, orabona2016, ward2019}. 

Our guarantees for general OCO readily transfer to the stochastic setting via the well-known online-to-batch conversion technique \cite{cesa2004, shalev2011}, leading to a convergence rate of order $O(\kappa\sqrt{d/T})$ which is optimal in $T$. We also present an algorithm for $\beta$-smooth stochastic optimization that achieves the rate $O(\beta/T^2 + \sigma/\sqrt{T})$ (hiding the dependence in $d$ and $\kappa$ for simplicity), where $\sigma^2$ is the variance of the noise of the observed subgradients. 
Crucially, our algorithm does not require knowledge of either $\beta$ or $\sigma$, and is thus fully adaptive. The rates just mentioned are optimal in $T$, and all the algorithms require $O(T \ln T)$ calls to $\mem_{\K}$ after $T$ rounds.  

Finally, in the offline setting, a special case of the stochastic setting with zero noise, our rates for general [resp.~$\beta$-smooth] convex functions imply that at most $\wtilde O (d \kappa^2/\epsilon^2)$ [resp.~$\wtilde O (d \kappa /\sqrt{\epsilon})]$ Membership Oracle calls are required to find an $\epsilon$-sub-optimal point. Since state-of-the-art algorithms based on the cutting plane method require $O(d^2 \ln (1/\epsilon))$ calls to a Membership Oracle to achieve the same guarantee \cite{lee2018efficient}, our algorithm provides a viable alternative to the latter whenever $d \geq \Omega(\kappa^2/\epsilon^2)$ for general convex functions or when $d \geq  \Omega (\kappa/\sqrt{\epsilon})$ for smooth functions. We note that the algorithm of \cite{lan2016} has a similar guarantee to ours in the offline setting, though they require a Linear Optimization Oracle on $\K$ instead of a Membership Oracle. We summarize our contributions in Table \ref{tab:result}, where we compare our average regret/rates to the best known methods for general convex sets.

\begin{wrapfigure}{r}{0.29\linewidth}
	\centering
	\includegraphics[width=\linewidth,trim={0cm 5.1cm 15.5cm .4cm},clip]{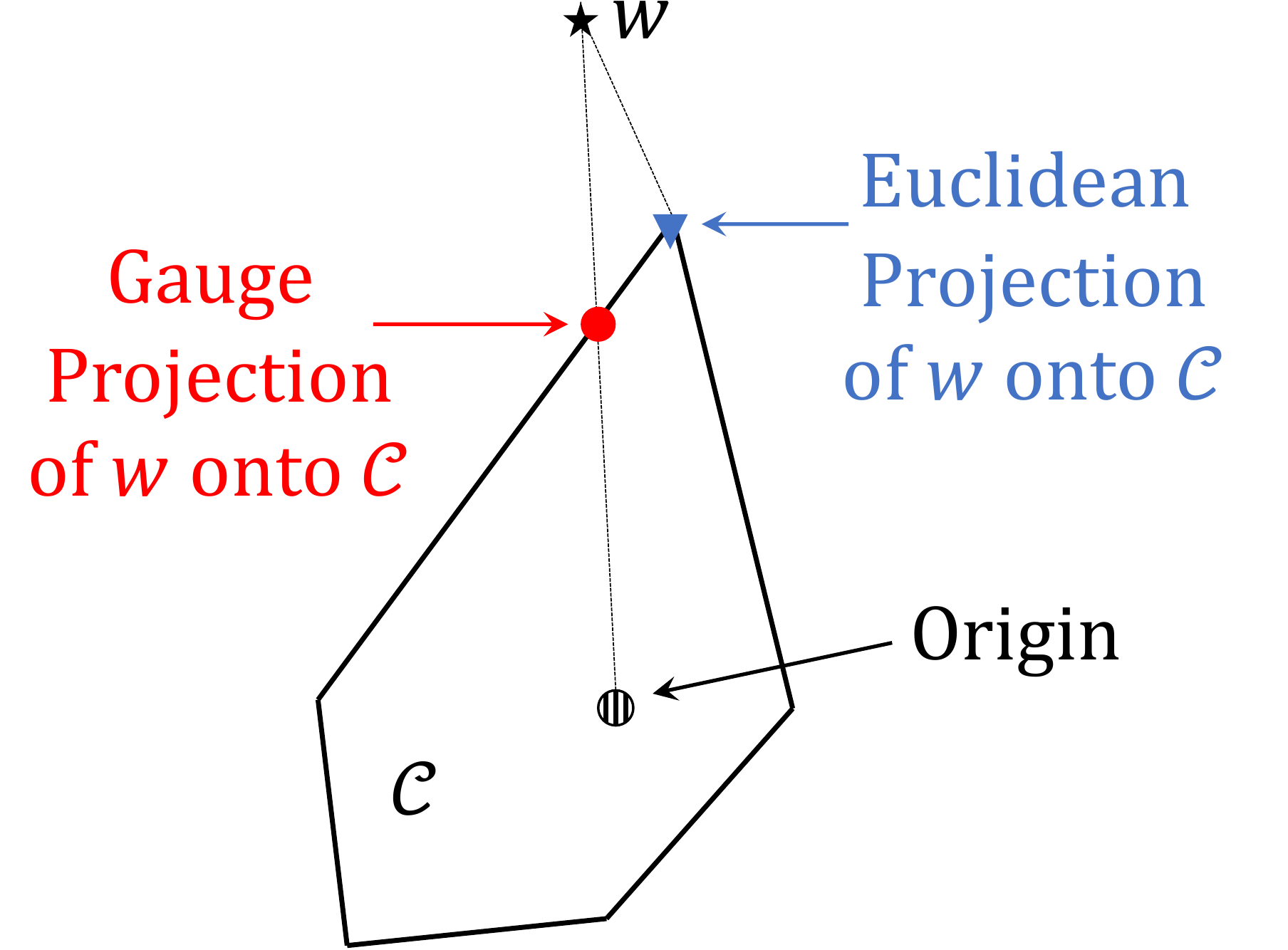}
	\caption{Gauge vs Euc. proj.}
	\label{fig:proj}
\end{wrapfigure}
We achieve all this by reducing an OCO problem on a general convex set $\K$ to an OCO problem on a ball, where projections can be performed efficiently. More specifically, we present an algorithm wrapper that takes in any base algorithm $\A$ defined on a Euclidean ball (where projections can be performed efficiently) containing $\K$ and turns it into an algorithm on $\K$ that does not incur any expensive projections. In particular, we replace Euclidean projections by, what we call, \emph{Gauge projections} that can be performed efficiently using a Membership Oracle (see Figure \ref{fig:proj} for an illustration of Gauge projections). What is more, the performance of the resulting algorithm, in terms of regret or rate, is at most a constant factor worse than the original algorithm $\A$.

\paragraph{Related Works.}
This paper continues a long a line of work in projection-free online learning. Table \ref{tab:result} summarizes the guarantees of different existing Frank-Wolfe-style algorithms. For general convex optimization (with general convex sets), the best known regret is of order $O(T^{3/4})$ and was achieved by \cite{hazan2012}. The corresponding Frank-Wolfe algorithm requires a single call to a Linear Optimization Oracle per round. Later, \cite{garber2016} introduced an algorithm that achieves the optimal $O(\sqrt{T})$ [resp. $O(\log T)$] regret in OCO with general [resp.~strongly convex] losses with one call to a LOO per round. We omitted this result from Table \ref{tab:result} since their algorithm and guarantees only apply to the case where the convex set is a polytope. What is more, they use an Oracle that requires at least $O( \min (dT, d^2))$ arithmetic operations per round (translating into $O(\min(n^4, T n^2))$ operations when $\K$ is a subset of $n\times n$ matrices) which may defeat the purpose of using a projection-free approach in the first place. 
\begin{table}
	\centering
	\begin{tabular}{llllll}
		\hline
	Setting &	Loss Function    &  \multicolumn{4}{l}{Average Regret/Rate}       \\
		&    & \multicolumn{2}{l}{Previous Best} &\multicolumn{2}{l}{ \bf{This paper} }  \\
		\hline 
		\hline
Online &	Non-Smooth	    &  $T^{-1/4}$ & \cite{hazan2012} &  $T^{-1/2}$& (Thms.~\ref{thm:genthm}-\ref{thm:genthmprob}) \\ 
Online &	Smooth &  $T^{-1/3}$ & \cite{hazan2020}  &  $T^{-1/2}$\ &(Thm.~\ref{thm:genthmeff})  \\ 
Online &	Strongly Convex &  $T^{-1/3}$ & \cite{kretzu2021} &  $T^{-1} \cdot \ln T $ &(Thm.~\ref{thm:mainstrong})  \\
Stochastic & 	Non-Smooth  &   $T^{-1/3}$&\cite{hazan2012} & $T^{-1/2}$&(Thms.~\ref{thm:genthm}-\ref{thm:genthmprob})  \\
Stochastic & 	Smooth   & $T^{-1/2}$&\cite{lan2017}  & $\beta T^{-2} + \sigma/T^{-1/2}$& (Thm.~\ref{thm:genthmsmooth}) \\
Stochastic &	Strongly Convex  &  $T^{-1/3}$& \cite{hazan2012} &  $T^{-1}\cdot  \ln T$&(Thm.~\ref{thm:mainstrong}) \\
Offline & 	Non-Smooth  &  $T^{-1/3}$& \cite{hazan2012}  &   $T^{-1/2}$& (Thms.~\ref{thm:genthm}-\ref{thm:genthmprob})  \\
Offline &	Smooth  & $\beta T^{-2}$&  \cite{lan2016}   &  $\beta T^{-2}$ & (Thm.~\ref{thm:genthmsmooth})  \\
		\hline
		\hline
	\end{tabular}
	\caption{Comparing the average regret/rate between our results and the previous best in the literature (we hide any dependence in $d$ for clarity). We only included results that hold for general bounded convex sets $\K$. The results reported for our paper [resp.~previous papers] require $O(T \ln T)$ [resp.~$O(T)$] calls to a Membership [resp.~Optimization] Oracle on $\K$ after $T$ rounds. For the smooth stochastic case, $\beta$ is the smoothness parameter and $\sigma$ corresponds to the noise in the stochastic model.}
	\label{tab:result}
\end{table}
Closer to our approach is that of \cite{mahdavi2012} and \cite{levy2019} who essentially also considered Membership and Separation Oracles for their algorithms instead of a Linear Optimization Oracle on $\K$. In their setting, the set $\K$ is of the form $\{\x \in \reals^d \mid h(\x)\leq 0\}$, where $h:\reals^d \rightarrow \reals$ is a smooth convex function, which translates into the set $\K$ being smooth. The approach of \cite{mahdavi2012} requires at least one orthogonal projection onto $\K$. \cite{levy2019} avoid this by providing a fast approximate projection and achieve optimal regret bounds for both convex and strongly convex losses while only requiring a single call to a Separate and Membership Oracle for $\K$. However, their guarantees become vacuous when the set $\K$ is non-smooth. Their algorithm may be viewed as complementary to that of \cite{garber2016} that works for polytopes. We did not include the guarantees of \cite{levy2019} in Table \ref{tab:result} as it is not applicable to general convex sets. Our approach can viewed be as a generalization of that of \cite{mahdavi2012,levy2019} to non-smooth sets. 

For OCO with smooth [resp.~strongly] convex losses \cite{hazan2020} [resp.~\cite{kretzu2021}] provide algorithms that achieve a $O(T^{2/3})$ regret (up to multiplicative factors in the dimension), requiring a single call to a LOO per round. Unlike our algorithm, their method requires knowledge of the strong convexity/smoothness parameter. Further, their regret bound is suboptimal to ours in terms of the dependence on the number of iterations (see Table \ref{tab:result}).

For smooth projection-free stochastic optimization, the best rate we are aware of is of order $O(\beta/T^2 + \sigma/\sqrt{T})$, where $\beta$ is the smoothness parameter and $\sigma$ is the stochastic noise. The first algorithm to achieve this makes $O(T)$ LO Oracle calls after $T$ rounds and is due to \cite{lan2017}. This improves on the previous best $T^{-1/4}$ rate due to \cite{hazan2012} that uses the same number of Oracle calls. We achieve a similar rate as that of \cite{lan2017}, but unlike the latter we do not require knowledge of $\beta$ or $\sigma$. Furthermore, our algorithm requires at most $O(T \ln T)$ calls to a Membership Oracle instead of a LOO. The story is similar for offline smooth convex optimization where we achieve the optimal accelerated rate of $O(\beta/T^2)$ for $\beta$-smooth function without requiring knowledge of $\beta$, unlike the algorithm of \cite{lan2016}.

Finally, the core idea behind our algorithm is based on an exiting unconstrained-to-constrained reduction due to \cite{cutkosky2018black,cutkosky2020parameter}, which uses carefully designed surrogate losses. The key novelty in our work is a choice of surrogate losses that depend on the Gauge function of the set $\K$ (see Definition \ref{def:gauge} below). In particular, we choose this dependence in a way that allows us to replace potentially expensive (Euclidean) projections by inexpensive Gauge projections (see Definition \ref{def:proj} and Figure \ref{fig:proj}), provided one has access to a Membership Oracle. We are able to perform efficient Gauge projections using recent techniques for estimating subgradients of convex functions that are not necessarily differentiable \cite{lee2018efficient}. Our algorithm for strongly-convex OCO [resp.~smooth stochastic optimization] are based on exiting reductions due to \cite{cutkosky2018black} [resp.~\cite{cutkosky2020parameter}]. 

\paragraph{Outline.}
In Section \ref{sec:prelim}, we describe our setting and provide some standard definitions in convex analysis. In Section \ref{sec:gauge0}, we present some important properties of the Gauge function of a set and describe tools we use to approximate its values and subgradients; these will be required by our projection-free algorithms. In Section \ref{sec:oco}, we present our main method that reduces constrained OCO on an arbitrary compact set $\K$ to OCO on a Euclidean ball containing $\K$ (or any convex set containing $\K$). We instantiate our algorithm for the case of strongly convex losses in Subsection \ref{sec:strong}, where we show that our algorithm achieves a logarithmic regret. We extend our approach to the stochastic and offline settings in Subsection \ref{sec:stochastic}, where we recover optimal accelerated rates. In Section \ref{sec:gauge}, we provide the details of our algorithms used to approximate the Gauge function, its subgradients, and Gauge projections. In Section \ref{sec:applications}, we show how our reduction can be applied in practice in light of our main Assumption \ref{ass:assum}. Missing proofs from the main body are included in Appendix \ref{sec:proofs}.

	\section{Setting and Definitions}
	\label{sec:prelim}
We consider the standard OCO setting where at each \emph{round} $t\geq 1$, a learner (the algorithm) outputs an iterate $\x_t$ in some convex set $\K\subset \reals^d$, then the environment reveals a convex loss $\ell_t\colon \K \rightarrow \reals$, and the learner suffers loss $\ell_t(\x_t)$. The goal of the learner is to minimize the \emph{regret} after $T\geq 1$ rounds against any comparator $\x\in \K$:
\begin{align}
\text{Regret}_T(\x) \coloneqq \sum_{t=1}^T \ell_t(\x_t) - \sum_{t=1}^T \ell_t(\x).\nn
\end{align}
\noindent We say that $(\ell_t)$ is an \emph{adversarial loss sequence} when $\ell_t$, $t\geq 1$, may depend on the learner's iterate $\x_t$ as well as the history $(\x_{s},\ell_s)_{s\in[t-1]}$. These are the types of losses we will consider in our general OCO setting. We note that our bounds are often on the \emph{linearized regret} $\sum_{t=1}^T \inner{\g_t}{\x_t-\x}$, where, for all $t\geq 1$, $\g_t$ is in the subdifferential $\partial \ell_t(\x_t)$ of $\ell_t$ at $\x_t$ (see e.g.~\cite{hiriart2004} for a definition), and $\inner{\cdot}{\cdot}$ denotes the standard inner product. Such bounds automatically transfer to a bound on the regret $\text{Regret}_T(\x)$ by convexity of the losses $(\ell_t)$; for any $\x\in \K$ and $t\geq 1$, we have $\ell_t(\x_t)-\ell_t(\x)\leq \inner{\g_t}{\x_t- \x}$, for all $\g_t\in \partial \ell_t(\x_t)$. Bounding the linearized regret is standard in OCO (see e.g.~\cite{hazan2019}). In some settings, we will consider $B$\emph{-Lipschitz}, $B>0$ [resp.~$\mu$-\emph{strongly convex}, $\mu >0$], losses $(\ell_t)$ which are those that satisfy, for all $t\geq 1$, $\x\in \K$ and $\g_t\in \partial \ell_t(\x)$:
\begin{align}
 	\|\g_t \|\leq B, \quad  [\text{resp.}\  \forall \y\in \K,  \quad  \ell_t(\y) \geq  \ell_t(\x) + \inner{\g_t}{\y - \x}+ \mu \|\x - \y\|^2/2,] \label{eq:lipstrong}
\end{align}
where $\|\cdot\|$ is the Eucledian norm.

In this work, we will not assume knowledge of the horizon $T$, and so our regret bounds hold for all $T$'s simultaneously; these are so-called \emph{anytime} regret bound. We will allow the outputs $(\x_t)$ of the learner and the losses $(\ell_t)$ to be random. In this case, for any $t\geq1$, we let $\cG_{t}$ denote the $\sigma$-algebra generated by $(\x_s, \ell_s)_{s\in[t]}$, and we denote $\E_t[\cdot]\coloneqq \E[\cdot \mid \cG_{t}]$. When for any round $t$ the output $\x_t$ of an algorithm $\A$ is a deterministic function of the history $(\x_s, \ell_s)_{s\in[t-1]}$, we say that $\A$ is a \emph{deterministic}.
Throughout, we assume that $\K$ is ``sandwiched'' between two Euclidean balls $\cB(r)$ and $\cB(R)$ with $0<r<R$: 
\begin{assumption}
		We assume that $\K \subseteq \reals^d$ is a closed convex set such that for some $0<r\leq R$, we have
		\begin{align}
			\cB(r)\subseteq  \K \subseteq \cB(R). \label{eq:lowradius}
		\end{align}
	\end{assumption}
\noindent We note that this assumption comes with no loss of generality as one can easily re-parametrize any OCO setting to satisfy this assumption without any significant computational overhead. In Section \ref{sec:applications}, we discuss the details of how this can be done in the general OCO setting as well as in the popular settings of Table \ref{tab:comp}. In Section \ref{sec:applications}, we also derive upper bounds on the \emph{condition number} $\kappa\coloneqq R/r$, which appears as multiplicative factor in our regret bounds, in the different settings of Table \ref{tab:comp}. 
	
Our goal in this paper is to design algorithms with low regret in OCO without incurring expensive (Euclidean) projections onto the set $\K$. For this, we will use a Membership Oracle for the set $\K$ to perform what we call Gauge Projections. These projections can be done very efficiently using a small number of calls to the Membership Oracle as we detail in Section \ref{sec:gauge}. A Membership Oracle is a map $\cM_{\K}\colon \reals^d \rightarrow \{0,1\}$, where for any $\x\in \reals^d$, $\cM_{\K}(\x) =1$ if and only if $\x\in \K$. Such Oracles are often the building blocks of many offline optimization algorithms \cite{grotschel1993,bubeck2015,lee2018efficient}. In many practical settings, one can only efficiently test membership approximately, and so for the sake of generality we will only assume access to such approximate Oracles. To formally define the notion of approximate Membership Oracle, we let $\cB(\K,\delta)\coloneqq \{\x \in \reals^d : \exists \y \in \K  \text{ such that } \|\x-\y\|\leq \delta \}$, where we recall that $\|\cdot\|$ denotes the Euclidean norm. We also define $\cB(\K,-\delta)\coloneqq \{\x \in \reals^d : \forall \y \in \cB(\delta), \ \x+\y\in  \K \}$.
	\begin{definition}[Approximate Membership Oracle]
		\label{def:mem}
	 Let $\delta >0$ and $\K$ be a convex set such that $\cB(\delta)\subseteq \K\subset \reals^d$. Then, the map $\mem_{\K}(\cdot;\delta)\colon \reals^d \rightarrow \{0,1\}$ is a $\delta$-approximate Membership Oracle if for all $\x\in \reals^d$,
	 \begin{align}
	 \mathbb{I}_{\{\x \in  \cB(\K,-\delta) \}} \wedge	 \mathbb{I}_{\{\x \in \cB(\K, \delta)\}} \leq 	\mem_{\K}(\x;\delta) \leq  \mathbb{I}_{\{\x \in \cB(\K, -\delta) \}} \vee \mathbb{I}_{\{\x \in  \cB(\K, \delta)\}}. \label{eq:mem}
	 	\end{align}
	\end{definition}
\noindent Some definitions of (approximate) Membership Oracles only require \eqref{eq:mem} to hold with high probability (see e.g.~\cite{lee2018efficient}). Our analysis can easily be extended to this case, but we make the deterministic assumption in \eqref{eq:mem} to simplify presentation. We will use an approximate Membership Oracle to build a $\delta$-approximate Linear Optimization Oracle $\cO_{\K^\circ}(\cdot; \delta)$ for the polar set $\K^\circ$ (definition below), where for any $\w$ in a ball containing $\K$, the random output $(\gamma, \s)$ of $\cO(\w;\delta)$ essentially satisfies $ \sup_{\bu \in  \K} \inner{\w}{\bu}\leq \gamma \leq \sup_{\bu \in  \K} \inner{\w}{\bu}+\delta$ and $\inner{\s}{\w}\geq  \sup_{\bu \in  \K} \inner{\w}{\bu}-\delta$ with high probability\footnote{The guarantee provided by $\gamma$ (i.e.~$\inner{\w}{\bu}\leq \gamma \leq \sup_{\bu \in  \K} \inner{\w}{\bu}+\delta$) will be crucial in our analysis, which is why we make our LO Oracle $\cO_{\K^\circ}$ return a scalar-vector pair instead of a single vector as is more common for such Oracles.}.  

We now present some standard definitions in convex analysis (see e.g.~\cite{hiriart2004}). 
	\begin{definition}[Polar Set and Normal Cone] \label{def:normal}Let $\K\subseteq \reals^d$ be a convex set. 
		\begin{itemize}[leftmargin=*]
			\item The polar set $\K^\circ$ of $\K$ is defined as $\K^\circ \coloneqq \{\s \in \reals^d : \inner{\s}{\bu}\leq 1, \forall \bu \in \K\}$. 
			\item The normal cone of $\K$ at $\bu\in \K$ is the set $\cN_{\K}(\bu)\coloneqq \{\s\in \K:\inner{\s}{\w- \bu}\leq 0, \forall \w \in \K \}$.
		\end{itemize}
	\end{definition}
\noindent The notation of a polar set is key to our approach. We will essentially build an efficient Linear Optimization Oracle on $\K^\circ$ that we will then allow us to perform, what we call, Gauge projections (formal definition in Section \ref{sec:gauge}) instead of the often-expensive Eucledian projections. Gauge projections are defined with respect to a quasi-metric based on the Gauge function of the set $\K$:
	\begin{definition}[Gauge and Support Functions]
	\label{def:gauge}
		Let $\K\subseteq \reals^d$ be a convex set. 
		\begin{itemize}[leftmargin=*]
			\item The Gauge function $\gamma_{\K}$ of $\K$ is defined as $\gamma_{\K}(\w)\coloneqq \inf\{\lambda>0:\w \in \lambda \K\}$, for $\w\in \reals^d$.
			\item The support function $\sigma_{\K}$ of $\K$ is defined as $\sigma_{\K}(\w)\coloneqq \sup_{\bu\in \K} \inner{\bu}{\w}$, for $\w\in \reals^d$. 
		\end{itemize}
	\end{definition}
\noindent The Gauge function satisfies all properties of a norm except for symmetry, which requires the set $\K$ to be centrally symmetric. Fortunately, we do not need this property to perform efficient projections using the quasi-metric $\mathrm{d}_{\K}(\bu,\w)\coloneqq \gamma_{\K}(\bu-\w)$. The concept of support function in the previous definition will also be key in our analysis (it is the Gauge function of the polar set $\K^\circ$ as we state in the next lemma). We now introduce the concept of Gauge projection/distance to a set:
	\begin{definition}
		\label{def:proj}
Let $\K$ be as in Assumption \ref{ass:assum}, the Gauge projection on $\K$ is the set-valued map $\Pi^{\mathrm{G}}_{\K}: \reals^d \rightrightarrows \K$ defined by $\Pi^{\mathrm{G}}_{\K}(\w) \coloneqq \argmin_{\bu \in \K} \gamma_{\K}(\w-\bu)$. We also define the Gauge distance to $\K$ as the map $S_{\K}:\reals^d \rightarrow \reals_{\geq 0}$:
\begin{align}
			S_{\K}(\w)=  \inf_{\bu\in \K}\gamma_{\K}(\w - \bu). \label{eq:theS}
		\end{align}
	\end{definition}
	\noindent
Gauge projections, which can be done efficiently using a Membership Oracle (see Section \ref{sec:gauge}), will be all we need to ensure the iterates of our algorithms are in $\K$ thanks to carefully designed surrogate losses.

	\paragraph{Additional Notation.} We let $\ln_+(u) \coloneqq 0 \vee\ln (u)$, for $u>0$. For $(\x_t)\subset \reals^d$ and $t\in \mathbb{N}$, we write $\x_{1:t} \coloneqq (\x_1,\dots, \x_t) \in \reals^{d\times t}$. 
For any real function $f\colon \cA \rightarrow \reals$, we denote by $\argmax_{a\in \cA} f(a)$ [resp.~$\argmin_{a\in \cA} f(a)$] the subset $\cK$ of $\cA$ such that $f(b)= \sup_{a\in \cA} f(a)$ [resp.~$f(b)= \inf_{a\in \cA} f(a)$], for all $b\in \cK$. For any $p\in \mathbb{R}\cup \{\infty\}$, $\bu\in \reals^d$, and $r\geq 0$ we denote by $\cB_p(\bu,r)$, the $\|\cdot\|_p$-ball of radius $r$ centered at $\bu$, where $\|\cdot\|_p$ denotes the $p$-norm. When $\bu =\bm{0}$, we simply write $\cB_p(r)$ for $\cB_p(\bm{0},r)$. For the special case where $p=2$, we let $\|\cdot\|\coloneqq \|\cdot\|_2$ and $\cB(\cdot)\coloneqq \cB_2(\cdot)$. We denote by $\Pi_{\cB(r)}\colon \reals^d \rightarrow \cB(r)$ the Eucledian projection operator onto the ball $\cB(r)$. For any set $\cK$, we denote by $\operatorname{conv}\cK$ [resp.~$\operatorname{bd} \K$] the convex hull [resp.~boundary] of $\cK$. We let $\iota_\cK\colon \reals^d \rightarrow \{0,+\infty\}$ be the indicator function of the set $\cK$, where $\iota_{\cK}(\w)=0$ if and only if $\w\in \cK$. Finally, for any mathematical statement $\cE$, we let $\mathbb{I}_{\cE}=1$ if $\cE$ is true, and 0 otherwise.

	\section{Preliminaries: Gauge Function Approximation and More}
	\label{sec:gauge0}
	Some of the technical tools introduced in this section will be detailed further in Section \ref{sec:gauge}. At the heart of our approach is the design of surrogate losses for online and stochastic optimization that depend on the Gauge distance function $S_{\K}$ (see \eqref{eq:theS}) in a way that leads to an algorithm that performs Gauge projections instead of potentially expensive Euclidean projections (more details in the next section). Naturally, this will require the ability to efficiently evaluate $S_{\K}$ and its subgradients. The next lemma (whose proof is in Section \ref{sec:gauge}) shows that the latter can be expressed as concise functions of $\gamma_{\K}$ (the Gauge function):
	\begin{lemma}
		\label{lem:projectiononK}
		For any set $\K$ satisfying Assumption \ref{ass:assum}, and any $\bu\not\in\K$ and $\w\in \reals^d$, we have 
		\begin{gather}
	\frac{\bu}{\gamma_{\K}(\bu)} \in \Pi^{\mathrm{G}}_{\K}(\bu);	\ \	\ \ S_\K(\w) = 0\vee (\gamma_{\K}(\w) -1);   \ \  \& \ \ \partial S_\K(\w) = \left\{ \begin{array}{ll} \partial \gamma_{\K}(\w)=  \argmax_{\s\in \K^\circ}\inner{\s}{\w}, & \text{if } \w\notin \K;\\ \{ \bm{0}\},& \text{otherwise}.  \end{array}   \right. \nn
		\end{gather} 
	\end{lemma}
\noindent Note that for any $\bu\not\in \K$, Lemma \ref{lem:projectiononK} also provides the expression of a Gauge projection point in $\Pi^{\text{G}}_{\K}(\bu)$, which will be leveraged by our algorithm. Next, we describe how we build algorithms for approximating $\gamma_{\K}$ and its subgradients using our approximate Membership Oracle $\mem_{\K}$. As Lemma \ref{lem:projectiononK} shows, this is all that is needed to approximate $S_{\K}$ and its subgradients. 

\paragraph{Approximating $\gamma_{\K}$ and its subgradients.}  By definition of the Gauge function, we have for any $\w\in \reals^d$, \begin{align}
		\gamma_{\K}(\w) = \inf \{ \lambda \in \reals_{\geq 0} \mid \w \in \lambda \K \}= 1/\sup \{ \nu \in \reals_{\geq 0} \mid  \nu \w \in \K \}. \label{eq:gauge0}
	\end{align} 
Using the Membership Oracle $\mem_{\K}$, we can approximate the largest $\nu\geq 0$ such that $\nu \w \in \K$ via bisection, which will lead to an approximation of $\gamma_{\K}(\w)$ by \eqref{eq:gauge0}. This is exactly what we do in Algorithm \ref{alg:gauge} in Section \ref{sec:gauge} below. Algorithm \ref{alg:gauge} ($\gau_{\K}$) requires at most $\ceil{\log_2((4\kappa)^2/\delta)}+1$ calls to the Membership Oracle $\mem_{\K}(\cdot; r\delta/(4\kappa)^2)$ to return a $\delta$-approximate value of $\gamma_{\K}$, where $\kappa\coloneqq R/r$ and $r,R>0$ are as in \eqref{eq:lowradius}. In Lemma \ref{lem:gauge} of Section \ref{sec:gauge} below, we state the full guarantee of Algorithm \ref{alg:gauge}.
	
In addition to approximating $\gamma_{\K}$, we will also need to approximate its subgradients. By Lemma \ref{lem:projectiononK}, the subdifferential of $S_{\K}$ coincides with that of $\gamma_{\K}$ at any point $\w$ outside $\K$. The lemma also shows that finding a subgradient of $\gamma_{\K}$ is equivalent to performing linear optimization on the polar set $\K^\circ$. In Section \ref{sec:gauge}, we present two algorithms, $\O_{\K^\circ}$ (Alg.~\ref{alg:optimize}) and $\onedopt$ (Alg.~\ref{alg:optimizenew})), based on \cite[Algorithm~2]{lee2018efficient}, which use $\gau_{\K}$ (Alg.~\ref{alg:gauge}) and a random partial difference along different coordinates to approximate the subgradients of $\gamma_{\K}$. The first algorithm $\O_{\K^\circ}$ requires at most $O(d \ln (d\kappa/ \delta))$ calls to the approximate Membership Oracle $\mem_{\K}$ to find a ``$\delta$-approximate'' subgradient; the precise guarantee is stated in Proposition~\ref{prop:subgradient}. The algorithm $\onedopt$ is a stochastic version of $\O_{\K^\circ}$ that picks a single coordinate $I$ along which to estimate the subgradient of the Gauge function $\gamma_{\K}$. As a result, $\onedopt$ requires at most $O(\ln (d\kappa/ \delta))$ calls to the approximate Membership Oracle $\mem_{\K}$, which will provide a computational benefit over $\O_{\K^\circ}$ at the cost of a slightly worse performance (see Section \ref{sec:oco} for a discussion). Furthermore, the output of $\onedopt$ is equal to that of $\O_{\K^\circ}$ in expectation for the same input (see Lemma~\ref{lem:celeb} for the precise guarantee). 
\paragraph{Additional properties of the Gauge function.} We now present a set of key properties of Gauge and Support functions, which will be useful in the analysis to come (the proof is in Appendix \ref{sec:teclemmas}):
	\begin{lemma}
		\label{lem:properties}
		Let $\w\in \reals^d\setminus \{\bm{0}\}$ and $0< r\leq R$. Further, let $\K$ be a closed convex set such that $\cB(r) \subseteq \K \subseteq \cB(R)$. Then, the following properties hold:
		\begin{enumerate}[label=(\alph*)]
			\item $\sigma_{\K^\circ}(\w)= \gamma_{\K}(\w)$ and $(\K^\circ)^\circ =\K$.
			\item $\sigma_{\K}(\alpha \w) = \alpha \sigma_{\K}(\w)$ and $\partial \sigma_{\K}(\alpha \w)= \partial \sigma_{\K}(\w) = \argmax_{\bu \in \K} \inner{\bu}{\w}$, for all $\alpha\geq 0$.  
			\item $r \|\w\|\leq \sigma_{\K}(\w) \leq R\|\w\|$, $\|\w\|/R\leq \gamma_{\K}(\w) \leq \|\w\|/r$, and $\cB(1/R)\subseteq \K^\circ \subseteq \cB(1/r)$.
			\item $\inner{\w}{\bu}\leq \sigma_{\K}(\w) \cdot \gamma_{\K}(\bu)$, for all $\bu\in \reals^d$.  (Cauchy Schwarz)
			\item $\sigma_{\K}(\w+\bu)\leq\sigma_{\K}(\w)+\sigma_{\K}(\bu)$, for all $\bu\in \reals^d$.  (Sub-additivity)
		\end{enumerate}
	\end{lemma}
\noindent We now move to the main section of the paper where we present our projection-free reduction that uses the tools we described in this section.

	\section{Efficient Projection-Free Online and Stochastic Convex Optimization}
	\label{sec:oco}
	In this section, we will use our approximate Linear Optimization Oracles on $\K^\circ$ described in the previous section (and detailed in Section \ref{sec:gauge} below) to build efficient algorithms for OCO with optimal regret bounds. Our final algorithms make at most $O(T \ln T)$ calls to the Membership Oracle $\mem_{\K}$ after $T\geq 1$ rounds, and are also time-uniform in the sense that they do not require a horizon $T$ as input.
	
	In  Subsection \ref{sec:general}, we will consider the case where the losses ($\ell_t$) are convex without additional curvature assumptions. In this case, we are able to design scale-invariant algorithms (Alg.~\ref{alg:projectionfreewrapper} in the settings of Theorems \ref{thm:genthm}-\ref{thm:genthmprob}) with an adaptive regret bound; one that scales with $O(\sqrt{V_T})$, where $V_T \coloneqq \sum_{t=1}^T \|\g_t\|^2$ and $(\g_t)$ are the observed subgradients at the iterates of the algorithm. Such bounds can be much tighter than the standard $O(\sqrt{T})$ bounds in settings with smooth losses \cite{srebro2010} (see also discussion in \cite{cutkosky2018black}). In particular, in the stochastic setting with a $\beta$-smooth objective function, our adaptive bound automatically implies a rate of order $O(\beta/T+ \sigma/\sqrt{T})$, where $\sigma$ is related to the stochastic noise \cite[Corollary 6]{cutkosky2019}. Thus, when there is no noise, i.e.~in the offline setting, the algorithm achieves the fast $O(\beta/T)$ rate without knowing the smoothness parameter $\beta$. In Section, \ref{sec:stochastic} we show an even faster rate using a slighly modified algorithm.  
	
	In Subsection \ref{sec:strong}, we design an algorithm (Alg.~\ref{alg:projectionfreewrapperstong}) that achieves a logarithmic regret whenever losses are strongly convex, while maintain a $\wtilde O(\sqrt{T})$ regret in the worse case. Additionally, our algorithm does not require any curvature parameter as input to achieve a logarithm regret. Finally, in Appendix \ref{sec:scaleinvariant} we present a reduction that makes our algorithm for the strongly convex case scale-invariant in the sense that scalling the losses by some positive constant does not change the outputs of the algorithm (consequently the algorithm does not require any scale information as input). 
	
	\begin{algorithm}[H]
		\caption{Projection-Free Algorithm-Wrapper for OCO on $\K$.}
		\label{alg:projectionfreewrapper}
		\begin{algorithmic}[1]
			\Require A OCO algorithm $\A$ on $\cB(R)\supseteq \K$.
			\Statex~~~~~~~~~~ A tolerance sequence $(\delta_t)\subset (0,1/3)$.
			\Statex~~~~~~~~~~ LO Oracle $\cO_{\K^\circ}$ on $\K^\circ$ such that $\cO_{\K^\circ}(\w;\delta)\in  \reals_{\geq0}\times \reals^d$, $\forall \w\in \cB(R), \delta \in (0,1)$.
			\Statex~~~~~~~~~~ \algcomment{In the ideal case, $(\gamma, \s)=\cO_{\K^\circ}(\w;\delta) \implies \gamma = \inner{\s}{\w} \text{ and } \w\in \argmin_{\bu\in \K^\circ}\inner{\bu}{\w}$, $\forall \w\in \cB(R), \delta \in (0,1)$} 
			\vspace{0.2cm} 
			\State Initialize $\A$ and set $\w_{1} \in \cB(R)$ to $\A$'s first output.
			\For{$t=1,2,\dots$}
			\State Set $(\gamma_t,\s_t) = \cO_{\K^\circ}(\w_t;\delta_t)$
			\State Set $\bnu_t= \mathbb{I}_{\{\gamma_t\geq 1\}} \s_t$. \algcomment{Subgradient~of $S_{\K}$ at $\w_t$}
			\State Play $\x_t=\mathbb{I}_{\{\gamma_t\geq 1\}} \w_t/\gamma_t +\mathbb{I}_{\{\gamma_t< 1\}}\w_t.$ \algcomment{Gauge projection of $\w_t$ onto $\K$}
			\State Observe subgradient $\g_t \in \partial \ell_t(\x_t)$.
			\State  \label{line:2} Set $ \tilde{\g}_t = \g_t -  \mathbb{I}_{\inner{\g_t}{\w_t}<0}    \inner{\g_t}{\x_t} \bnu_t$ \algcomment{Subgradient of $\tilde \ell_t(\w) := \inner{\g_t}{\w} - \mathbb{I}_{\inner{\g_t}{\w_t} <0}  \inner{\g_t}{{\x}_t}  S_{\K}(\w)$ at $\w_t$} 
			\State Set $\A$'s $t$th loss function to $f_t : \w \mapsto \inner{\tilde{\g}_t}{\w}$.
			\State Set $\w_{t+1} \in \cB(R)$ to $\A$'s $(t+1)$th output given the history $((\w_i, \x_i,f_i)_{i\leq t})$.
			\EndFor
		\end{algorithmic}
	\end{algorithm}	
\noindent 	To avoid expensive projections, our algorithms make use of convex surrogate losses of the form $\tilde \ell(\w) \coloneqq\inner{\g}{\w} + b S_{\K}(\w)$, for $\w\in \reals^d$ and $b\geq 0$, where $S_{\K}$ is the Gauge distance function defined in \eqref{eq:theS} and $\g$ is a subgradient of some loss of interest $\ell\colon \K \rightarrow \reals$. Note that $\tilde \ell$ is not constrained to the set $\K$, unlike the actual loss of interest $\ell$. The choice of such a surrogate loss function is inspired by existing constrained-to-unconstrained reductions in OCO due to \cite{cutkosky2018black,cutkosky2020parameter}. Similar to the latter, our projection-free reduction allows us to bound the regret of our algorithm (in this case Alg.~\ref{alg:projectionfreewrapper}) by the regret of any OCO subroutine $\A$ that is fed subgradients of the surrogate losses. Here, the iterates of $\A$ need not be constrained to the set $\K$ since the domain of the surrogate losses is technically unconstrained. Thus, to build efficient projection-free OCO algorithms on $\K$ with optimal regret, it suffices to pick a sub-algorithm $\A$ that I) has an optimal regret bound on a set $\cW$ containing $\K$ and II) that does not incur any expensive projections. We will pick $\cW=\cB(R)\supseteq \K$ so that Euclidean projections onto $\cW$, which may be required by $\A$, can be performed efficiently.
	
	As mentioned in the preceeding paragraph, to excecute our projection-free reduction, we need to feed the subroutine $\A$ of Algorithm \ref{alg:projectionfreewrapper} subgradients of convex surrogate losses, which are of the form $\tilde \ell(\cdot) \coloneqq\inner{\g}{\cdot} + b S_{\K}(\cdot)$, $b\geq 0$. This requires evaluating the subgradients of the Gauge distance function $S_{\K}$. As we showed in Lemma \ref{lem:projectiononK}, evaluating such subgradients reduces to linear optimization on $\K^\circ$. The latter can be performed efficiently using a Membership Oracle as we described in Section \ref{sec:gauge0} (and detail in Section \ref{sec:gauge}). For clarity, we will first present the implications of our projection-free reduction when Algorithm \ref{alg:projectionfreewrapper} has access to a perfect LO Oracle $\cO_{\K^\circ}$ on $\K^\circ$; that is, we assume that $\cO_{\K^\circ}$ satisfies:
	\begin{align}
\forall \w \in \reals^d, \delta >0 , \quad  \left[(\gamma, \s)=	\cO_{\K^\circ}(\w; \delta)\right] \implies  \left[ \s \in  \argmax_{\x \in \K^\circ} \inner{\x}{\w} \ \ \text{and} \ \  \gamma =\inner{\s}{\w} \right]. \quad \label{eq:circle}
		\end{align} 
	The second argument $\delta>0$ of $\cO_{\K^\circ}$ does not play a relevant role when the latter is a perfect LO Oracle. 
		\begin{lemma}
		\label{lem:mainreductionperfect}
		Let $\kappa\coloneqq R/r$, where $r,R>0$ are as in \eqref{eq:lowradius}. Further, for $t\ge1$, let $\w_t, \tilde{\g}_t, \x_t$, and $\g_t\in \partial \ell_t(\x_t)$ be as in Algorithm \ref{alg:projectionfreewrapper} with $\cO_{\K^\circ}$ as in \eqref{eq:circle}; subroutine $\A$ set to any OCO algorithm on $\cB(R)$; and any tolerance sequence $(\delta_s)\subset (0,1/3)$. Then, $(\x_t)\subset \K$ and $\|\tilde{\g}_t\|\leq   (1+\kappa) \|\g_t\|$. Further, we have
		\begin{align}
			\forall \x\in \K,\quad 
			\inner{\g_t}{\x_t-\x} \leq \tilde\ell_t(\w_t) - \tilde \ell_t(\x)\leq \inner{\tilde{\g}_t}{\w_t-\x}, \label{eq:individualperfect} 
		\end{align}	
	where $\tilde \ell_t(\w) \coloneqq \inner{\g_t}{\w} - \mathbb{I}_{\inner{\g_t}{\w_t} <0} \cdot \inner{\g_t}{{\x}_t} \cdot S_{\K}(\w)$.
	\end{lemma}
\noindent Lemma \ref{lem:mainreductionperfect} allows us to bound the instantaneous (linearized) regret $\inner{\g_t}{\x_t - \x}$ of Algorithm \ref{alg:projectionfreewrapper} with respect to the instantaneous regret $\inner{\tilde\g_t}{\w_t -\x}$ of subroutine $\A$. Thus by summing \eqref{eq:individualperfect} in Lemma \ref{lem:mainreductionperfect} for $t=1,2,\dots$, we can bound the regret of Alg.~\ref{alg:projectionfreewrapper} with respect to that of $\A$. Further, Lemma \ref{lem:mainreductionperfect} shows that the scale of the subgradients $(\tilde \g_t)$ of the losses fed to $\A$ is at most a constant times the scale of the subgradients $(\g_t)$ at the iterates $(\x_t)$ of Algorithm \ref{alg:projectionfreewrapper}; in particular, $\|\tilde \g_t\| \leq (1+\kappa)\|\g_t\|$, for all $t\geq 1$, where $\kappa \coloneqq R/r$. This latter fact ensures that the regret bound of $\A$ is not too large as a function of the scale of $(\g_t)$. 

	\begin{proof}[{\bf Proof of Lemma \ref{lem:mainreductionperfect}}]
		Let $\gamma_t$, $\bnu_t$, $ \w_t,\tilde{\g}_t$, and ${\x}_t$ be as in Algorithm \ref{alg:projectionfreewrapper}. 
By the expression of the subdifferential of $S_{\K}$ in Lemma~\ref{lem:projectiononK} and the definition of $\tilde \ell_t$, we have that $\g_t - \mathbb{I}_{\inner{\g_t}{\w_t} <0}  \inner{\g_t}{{\x}_t}  \bnu_t \in \partial \tilde \ell_t(\w_t)$. Thus, since $-\mathbb{I}_{\inner{\g_t}{\w_t} <0}  \inner{\g_t}{{\x}_t} >0$ (see definition of $\x_t$ in Alg.~\ref{alg:projectionfreewrapper}), $\tilde \ell_t$ is convex and so
		\begin{align}
			\forall \x\in \K,\quad \tilde \ell_t(\w_t) - \tilde \ell_t(\x) \leq \inner{\g_t - \mathbb{I}_{\inner{\g_t}{\w_t} <0}  \inner{\g_t}{{\x}_t}  \bnu_t}{\w_t-\x}= \inner{\tilde\g_t}{\w_t -\x}. \nn
		\end{align}
		It remains to show that $\inner{\g_t}{\x_t -\x}\leq \tilde \ell_t(\w_t) - \tilde \ell_t(\x)$, for all $t\geq 1$ and $\x \in \K$. First note that for all $\x \in \K$, we have 	$S_{\K}(\x)=0$, and so
		\begin{align}
			\tilde \ell_t(\x) = \inner{\g_t}{\x}, \quad \forall \x \in \K. \label{eq:firstequalityperfect}
		\end{align}
		We will now compare $\inner{\g_t}{\x_t}$ to $\tilde \ell_t(\w_t)$ by considering cases. Suppose that $\gamma_t <1$. In this case, we have $\x_t = \w_t$ and so $
			\inner{\g_t}{\x_t} = \inner{\g_t}{\w_t} = \tilde \ell_t(\w_t).$ Now suppose that $\gamma_t \geq 1$ and $\inner{\g_t}{\w_t}\geq 0$. By the fact that $\gamma_{\K}(\w_t)=\gamma_t\geq 1$ and $\x_t = \w_t/\gamma_t$, we have
		\begin{align}
			\inner{\g_t}{\x_t} \leq \inner{\g_t}{\w_t} = \tilde \ell_t(\w_t). \label{eq:firstsecondperfect}
		\end{align}
		Now suppose that $\gamma_t\ge 1$ and $\inner{\g_t}{\w_t}<0$. Using the fact that $\x_t=\w_t/\gamma_t=\w_t/\gamma_{\K}(\w_t)$ and that $S_{\K}(\w_t)= \gamma_{\K}(\w_t)-1$ (Lemma \ref{lem:projectiononK}), we have 
		\begin{align}
		\inner{\g_t}{\x_t}+	\inner{\g_t}{\x_t}\cdot S_{\K}(\w_t)  =   \inner{\g_t}{\w_t/\gamma_{\K}(\w_t)} +  \inner{\g_t}{\w_t} - \inner{\g_t}{\w_t/\gamma_{\K}(\w_t)} =  \inner{\g_t}{\w_t}. \nn 
		\end{align}
	Rearranging this, we get
		\begin{gather}	
			\inner{\g_t}{\x_t}  =  \inner{\g_t}{\w_t} -  \inner{\g_t}{\x_t} S_{\K}(\w_t) = \tilde{\ell}_t(\w_t).\label{eq:secondperfect}
		\end{gather}
By combining \eqref{eq:firstequalityperfect}, \eqref{eq:firstsecondperfect}, and \eqref{eq:secondperfect}, we obtain:
		\begin{align}
			\inner{\g_t}{\x_t -\x}  \leq \tilde \ell_t(\w_t) - \tilde \ell_t(\x)  \leq \inner{\tilde{\g}_t}{\w_t -\x}, \quad \forall \x \in \K. \nn
		\end{align}
		This shows \eqref{eq:individualperfect}. It remains to bound $\|\tilde{\g}_t\|$ in terms of $\|\g_t\|$ and show that $\x_t \in \K$. When $\gamma_t <1$, we have $\tilde \g_t =\g_t$ and $\x_t=\w_t$, and so $\|\tilde\g_t\|=\|\g_t\|$. Furthermore, since $\gamma_{\K}(\w_t) = \gamma_t<1$, the definition of the Gauge function implies that $\w_t \in \K$ and so $\x_t\in \K$ (since $\x_t = \w_t$). Now suppose that $\gamma_t \geq 1$. In this case, we have $\x_t=\w_t/\gamma_t= \w_t/\gamma_{\K}(\w_t)$ and $\tilde\g_t=\g_t - \mathbb{I}_{\inner{\g_t}{\w_t}<0} \inner{\g_t}{\x_t} \bnu_t$. Using this and that $\gamma_{\K}(\w_t)=\gamma_t\geq 1$ we get
		\begin{align}
	\|\tilde{\g}_t\|  =  \|\g_t- \mathbb{I}_{\inner{\g_t}{\w_t} <0}  \inner{\g_t}{{\x}_t}  \bnu_t\| = 	\|\g_t\| + \|\g_t\| \frac{\| \w_t\|}{\gamma_{\K}(\w_t)}  \|\bnu_t\|  \stackrel{(*)}{\leq}	\|\g_t\| \left(1+  \frac{\| \w_t\|}{r} \right) \leq  (1 +\kappa )\|\g_t\|. \nn
		\end{align}
	where $(*)$ follows by the fact that $\|\bnu_t\|\leq 1/r$ (by Lemma \ref{lem:properties}-(c) and the fact that $\bnu_t \in \K^\circ$). Further, when $\gamma_t\geq 1$, we have $\gamma_\K(\x_t)=\gamma_\K(\w_t/{\gamma_{\K}(\w_t)})=\gamma_\K(\w_t)/{\gamma_{\K}(\w_t)}= 1$, and so $\x_t \in \K$. 
	\end{proof}	
\noindent Before stating the full regret bound of Algorithm \ref{alg:projectionfreewrapper} in the next subsection and discuss its computational complexity, we first state a version of Lemma~\ref{lem:mainreductionperfect} when the LO Oracle $\cO_{\K^\circ}$ is set to the approximate LO Oracle $\O_{\K^\circ}$ (Alg.~\ref{alg:optimize}) we build in Section \ref{sec:gauge}:
	\begin{lemma}
		\label{lem:mainreduction}
		Let $\kappa\coloneqq R/r$ and $\A$ be any OCO algorithm on $\cB(R)$. Further, for $t\ge1$, let $\w_t, \tilde{\g}_t, \x_t$, and $\g_t\in \partial \ell_t(\x_t)$ be as in Alg.~\ref{alg:projectionfreewrapper} with $\cO_{\K^\circ}\equiv \O_{\K^\circ}$ (Alg.~\ref{alg:optimize}) and any tolerance sequence $(\delta_s)\subset (0,1/3)$. Then, $(\x_t)\subset \K$, and $\forall t\geq 1$, there exists a variable $\Delta_t \in \left[0,{15^2 d^{4} \kappa^3}{\delta^{-2}_t}\right]$ s.t.~$\E_{t-1}[\Delta_t] \leq \delta_t$, $\|\tilde{\g}_t\|\leq   (1+\Delta_t+\kappa) \|\g_t\|$, 
		\begin{align}
	\text{and}\quad 		\forall \x\in \K,\quad 
			\inner{\g_t}{\x_t-\x} \leq \tilde\ell_t(\w_t) - \tilde \ell_t(\x)+\delta_t R \|\g_t\| \leq \inner{\tilde{\g}_t}{\w_t-\x} + (2\delta_t +\Delta_t) R \|\g_t\|. \label{eq:individual} 
		\end{align}	
	\end{lemma}
	\noindent The proof is in Appendix \ref{sec:mainreductionproof}. Next, we apply this result in the context of OCO with general losses.
	\subsection{Algorithm for General Online Convex Optimization}
	\label{sec:general}
	In Algorithm \ref{alg:projectionfreewrapper}, $(\x_t)\subset \K$ and $(\w_t)\subset \cB(R)$ are the iterates of Alg.~\ref{alg:projectionfreewrapper} and $\A$, respectively, where $\A$ is the subroutine of Alg.~\ref{alg:projectionfreewrapper}. In this case, Lemma \ref{lem:mainreduction}, and in particular \eqref{eq:individual}, implies that the instantaneous linearized regret of Alg.~\ref{alg:projectionfreewrapper} at some round $t\geq 1$ is bounded in terms of the instantaneous regret of the subroutine $\A$, which is fed approximate subgradients of the surrogate losses $(\tilde \ell_t)$. Thus, by summing \eqref{eq:individual} over $t=1,\dots,T$, we can bound the regret of Algorithm \ref{alg:projectionfreewrapper} in terms of that of $\A$. We formalize this result for a general subroutine $\A$ in Appendix \ref{sec:genred} (Proposition \ref{prop:generalred}).
	
	For the sake of clarity, we will now instantiate Algorithm \ref{alg:projectionfreewrapper} with a specific subroutine $\A$. In particular, we will set $\A$ to be the Follow-The-Regularized-Leader Proximal (\ftrl) \cite{mcmahan2017survey}. \ftrl{} is summarized in Algorithm \ref{alg:FTRL-proximal} in Appendix \ref{app:OCO}. The algorithm takes in a parameter $R>0$ and performs at most one Euclidean projection onto the ball $\cB(R)$ per round, which can be done efficiently (requiring only $O(d)$ arithmetic operations per projection). In particular, for any $\w \in \reals^d \setminus \cB(R)$, $\Pi_{\cB(R)}(\w) = R\w/\|\w\|$,  and $\Pi_{\cB(R)}(\w)=\w$ for $\w\in \cB(R)$. Another advantage of \ftrl{} is that it has a regret bound of the form $O(\sqrt{V_T})$, where $V_T\coloneqq \sum_{t=1}^T \|\g_t\|^2$ and $(\g_t)$ are the observed subgradients (see Proposition \ref{prop:ftrl-proximal} for the precise statement of the regret bound). This is known as an adaptive regret bound, and can be much smaller than the standard $O(\sqrt{T})$ worst-case regret bounds; e.g.~when the losses are smooth \cite{srebro2010,cutkosky2018black}. We now state the regret bound of Algorithm \ref{alg:projectionfreewrapper}
	when $\A$ is set to \ftrl: 
	\begin{theorem}
		\label{thm:genthm}
		Let $\delta \in(0,1/3)$ and $\kappa\coloneqq R/r$, with $r$ and $R$ as in \eqref{eq:lowradius}. Let $(\ell_t)$ be any adversarial sequence of convex losses on $\K$ and $(\x_t)$ be the iterates of Alg.~\ref{alg:projectionfreewrapper} in response to $(\ell_t)$. Further, let $\g_s \in \partial \ell_s(\x_s)$, $s\geq 1$. If Alg.~\ref{alg:projectionfreewrapper} is run with subroutine $\A$ set to \ftrl{} with parameter $R$; $\cO_{\K^\circ}\equiv \O_{\K^\circ}$ (Alg.~\ref{alg:optimize}); and $\delta_t= \delta/t^2$, $t\geq 1$, then for all $T\geq 1$ and $\rho \in (0,1)$, $(\x_t)_{t\in[T]} \in \K$ and with probability at least $1-\rho$:
		\begin{align}
	\forall \x \in \K, \ \   \sum_{t=1}^T \inner{\g_t}{\x_t-\x} \leq 4(1+\kappa)R\sqrt{\sum_{t=1}^T \|\g_t\|^2}  + R\cdot  (12+10\delta/\rho) \cdot  \max_{t\leq T} \|\g_t\|.\label{eq:regretftrl}
		\end{align}
	\end{theorem}
	\begin{proof}[{\bf Proof}]
		By Lemma \ref{lem:mainreduction}, we have, for all $\w$ and $t\geq 1$, $\inner{\g_t}{\x_t - \x}  \leq \inner{\tilde \g_t}{\w_t-\x} + (2\delta_t +\Delta_t)R \|\g_t\|$, where $(\Delta_t)\subset \reals_{\geq 0}$ is a sequence of positive random variables satisfying $\E_{t-1}[\Delta_t]\leq \delta_t$, for all $t\in[T]$. Summing this inequality for $t=1,\dots,T$, we obtain, for all $\x\in \K$, 
		\begin{align}
			\sum_{t=1}^T	\inner{{\g}_t}{\x_t - \x} & \leq \sum_{t=1}^T \inner{\tilde{\g}_t}{\w_t-\x}+\sum_{t=1}^T(2\delta_t +\Delta_t)R \|\g_t\|, \nn \\
			& \leq 2 R \sqrt{2\sum_{t=1}^T \| {\tilde{\g}}_t\|^2} + \sum_{t=1}^T(2\delta_t +\Delta_t)R \|\g_t\|,\label{eq:regretftrlold}  \\
			&  \leq 4 (1+\kappa)R \sqrt{\sum_{t=1}^T \| {{\g}}_t\|^2} + \sum_{t=1}^T(2\delta_t +5\Delta_t)R \|\g_t\|, \label{eq:show}
		\end{align}
		where \eqref{eq:regretftrlold} follows by our choice of the subroutine $\A$ and the regret bound of \ftrl{} in Proposition \ref{prop:ftrl-proximal}, and the last inequality follows by the fact that $\|\tilde \g_t\| \leq (1+\kappa+\Delta_t) \|\g_t\|$ (Lemma \ref{lem:mainreduction}) and the fact that $(a+b)^2 \leq 2 a^2 + 2 b^2$ for all $a,b\in \reals_{>0}$. By Lemma \ref{lem:mainreduction}, we also have $\x_t \in \K$, for all $t\geq 1$. 
		
		We now instantiate \eqref{eq:show} with the specific choice of tolerance sequence $(\delta_t)$ in theorem's statement, where we explicitly bound the right-most sum in \eqref{eq:regretftrlold} involving $(\Delta_t)$. For this, we let $X_t \coloneqq \sum_{i=1}^t (\Delta_i - \bar\delta_i )$, where $\bar \delta_i \coloneqq \E_{i-1}[\Delta_i]\leq \delta_i$. The process $(X_t)$ is a martingale; that is, for all $i\geq 1$, we have, $\E_i[X_t]  = X_i,$ for all $i<t$. Thus, by Doob's martingale inequality \cite[Theorem 4.4.2]{durrett2019} we have, for any $\rho\in(0,1)$ and $T\geq 1$:
		\begin{align}
			\P\left[ \sum_{t=1}^T\Delta_t \geq (1+1/\rho) \sum_{t=1}^T \delta_t  \right] \leq \P\left[ X_T\geq   \sum_{t=1}^T \delta_t/\rho  \right]   \leq \P\left[\max_{t\leq T} X_t \geq  \sum_{t=1}^T \delta_t/\rho  \right] \leq   \frac{\rho  \E\left[X_T\vee 0  \right]}{ \sum_{t=1}^T \delta_t} \leq \rho, \nn
		\end{align}
	where the last inequality follows by the fact that $\E[X_T\vee0]\leq \E[\sum_{t=1}^T \Delta_t]\leq \sum_{t=1}^T \delta_t$. Using this and the fact that $\sum_{t=1}^{\infty}1/t^2\leq 2$ in combination with \eqref{eq:show}, we obtain \eqref{eq:regretftrl}. 
	\end{proof}
\noindent Note that at each round $t$, the instance of Algorithm \ref{alg:projectionfreewrapper} in Theorem \ref{thm:genthm} invokes $\O_{\K^\circ}(\cdot;\delta_t)$ once. Thus, in light of Remark \ref{rem:complexity} in Section \ref{sec:gauge}, the algorithm makes at most $O( d T\ln (dT\kappa/\delta)$ calls to the Membership Oracle $\mem_{\K}$ after $T$ rounds. We also note that if one has an optimization Oracle for $\K^\circ$ [resp.~a separation Oracle for $\K$ \cite{lee2018efficient}], then Algorithm \ref{alg:projectionfreewrapper} can be executed with only ${O}(1)$ [resp.~$\wtilde O(1)$] Oracle calls per round while maintaining the same regret as in Theorem \ref{thm:genthm}.
	
	\paragraph{A More Efficient Implementation.} We can reduce the number of Oracle calls to a total of $O(T \ln (d T \kappa/\delta))$ by using the ``one-dimensional'' stochastic version of $\O_{\K^\circ}$ we present in Section \ref{sec:gauge}; i.e. $\onedopt$ (Alg.~\ref{alg:optimizenew}). This will come at the cost of a $O(\sqrt{d})$ multiplicative factor in the regret bound. We now state the guarantee of Algorithm \ref{alg:projectionfreewrapper} with $\cO_{\K^\circ} \equiv  \onedopt$ (the proof in in Appendix \ref{sec:genthmeffproof}):
	\begin{theorem}
		\label{thm:genthmeff}
		Let $\delta \in (0,1/3)$ and $\kappa\coloneqq R/r$, with $r$ and $R$ as in \eqref{eq:lowradius}. Let $(\ell_t)$ be any adversarial sequence of convex losses on $\K$ and $(\x_t)$ be the iterates of Alg.~\ref{alg:projectionfreewrapper} in response to $(\ell_t)$. Further, let $\g_s \in \partial \ell_s(\x_s)$, $s\geq 1$. If Alg.~\ref{alg:projectionfreewrapper} is run with subroutine $\A$ set to \ftrl{} (Alg.~\ref{alg:FTRL-proximal}) with parameter $R$; $\cO_{\K^\circ}\equiv \onedopt$; and $\delta_t= \delta/t^2$, $t\geq 1$, then $(\x_t)\subset \K$ and 
		\begin{align}
			\forall \x \in \K, \quad \sum_{t=1}^T\E[\inner{{\g}_t}{\x_t - \x}]  \leq   4 R \sqrt{(1+d\cdot (\kappa+\delta)^2) \sum_{t=1}^T   \E \left[ \| {{\g}}_t\|^2\right]}  + 6 \delta R \cdot \E\left[\max_{t\in[T]}\|\g_t\|\right].\label{eq:regretoned}
		\end{align}
	\end{theorem}
\noindent The instance of Algorithm \ref{alg:projectionfreewrapper} in Theorem \ref{thm:genthmeff} invokes $\onedopt(\cdot;\delta_t)$ at each iteration $t$. Thus, in light of Remark \ref{rem:complexityoned} in Section \ref{sec:gauge}, the algorithm makes at most $O(T\ln (dT\kappa/\delta)$ calls to the Membership Oracle $\mem_{\K}$ after $T$ rounds. We also remark that our choice of tolerance sequence $(\delta_t)$ in Theorems \ref{thm:genthm} and \ref{thm:genthmeff} is too conservative, requiring the Membership Oracle to be more accurate than necessary to achieve a $O(\sqrt{T})$ regret. In fact, our choice of $(\delta_t)$ in Theorems \ref{thm:genthm} and \ref{thm:genthmeff} ensures that the errors involved in the approximations of the subgradients of the surrogate losses add up to a lower order term in the regret bound; i.e.~the right-most terms in \eqref{eq:regretftrl} and \eqref{eq:regretoned}. We can choose a larger sequence of tolerances as long as the sum of the approximation errors is of order $O(\sqrt{T})$. Next we derive a high probability regret bound for Algorithm \ref{alg:projectionfreewrapper} and show that one can pick $\delta_t$ as large as $O(t^{-1/2})$.
 
\paragraph{High Probability Regret Bound.}
We will now derive a high probability regret bound for Algorithm \ref{alg:projectionfreewrapper} where $\cO_{\K^\circ}\equiv \onedopt$ (Alg.~\ref{alg:optimizenew}) and $\A$ is set to \ftrl. For the sake of simplicity, we will assume that the losses are $B$-Lipschitz (see \eqref{eq:lipstrong}). Technically, this condition is not needed to derive the result in our next theorem (up to log factors in the regret), but we make it to simplify the probabilistic argument we follow in the proof of the latter. We note that the algorithm does not require knowledge of $B$.
\begin{theorem}
	\label{thm:genthmprob}
	Let $\delta \in(0,1/3)$, $B>0$, and $\kappa\coloneqq R/r$, with $r$ and $R$ as in \eqref{eq:lowradius}. Let $(\ell_t)$ be any adversarial sequence of $B$-Lipschitz convex losses on $\K$ and $(\x_t)$ be the iterates of Alg.~\ref{alg:projectionfreewrapper} in response to $(\g_t \in \partial \ell_t(\x_t))$. If Alg.~\ref{alg:projectionfreewrapper} is run with subroutine $\A$ set to \ftrl{} with parameter $R$; $\cO_{\K^\circ}\equiv \onedopt$ (Alg.~\ref{alg:optimizenew}); and $\delta_t= \delta/t^{-1/2}$, $t\geq 1$, then $(\x_t)\subset \K$ and for all $\rho\in(0,1)$, $T\geq d\ln ({1}/{\rho})$, and $\x \in \K$,
	\begin{align}
		\P\left[\sum_{t=1}^T \inner{\g_t}{\x_t - \x}\leq  8 R B(\kappa+\delta)\sqrt{ d  T (3+2\ln(1/\rho)) }    + 2 d B R \cdot (2\kappa+2\delta  +3\delta \sqrt{T}/\rho) \right] \geq 1 - \rho.\nn
	\end{align}
\end{theorem}
\noindent  The proof of the theorem is in Appendix \ref{sec:genthmprobproof}. Once again, the instance of Algorithm \ref{alg:projectionfreewrapper} in Theorem \ref{thm:genthmprob} requires only $O(T \ln (T d\kappa/\delta ))$ calls of the Membership Oracle $\mem_{\K}$ (see the discussion proceeding Theorem \ref{thm:genthmeff}). 
    \paragraph{Optimality of the bounds and computational trade-offs.} 
  All the regret bounds in Theorems \ref{thm:genthm}-\ref{thm:genthmprob} have an optimal dependence in $T$. As for the dependence in $d$, we see that the regret bounds of Algorithm~\ref{alg:projectionfreewrapper} in the settings of Theorems \ref{thm:genthmeff} and \ref{thm:genthmprob} can be improved by a factor of $O(\sqrt{d})$ when making $\wtilde{O}(d)$ calls to $\mem_{\K}$ per round (as in the setting of Theorem \ref{thm:genthm}) or $\wtilde{O}(1)$ calls to a Separation Oracle for $\K$, if available. We also note that $\onedopt$ (Alg.~\ref{alg:optimizenew}), which is used in the settings of Theorems \ref{thm:genthmeff} and \ref{thm:genthmprob}, randomly selects a coordinate in $[d]$ and estimates the subgradient of $\gamma_{\K}$ in that coordinate's direction. A version of this algorithm, call it $k\text{D}$-$\text{OPT}_{\K^\circ}$, that samples $i\in [\lfloor d/k \rfloor]$ uniformly at random and selects coordinates $\{i k + j-1: j \in [k]\}\cap [d]$ to estimate the subgradients of $\gamma_{\K}$ would lead to a regret bound for Algorithm \ref{alg:projectionfreewrapper}, with $\cO_{\K^\circ}\equiv k\text{D}$-$\text{OPT}_{\K^\circ}$, of order $O(\sqrt{d T/k})$, while making $\wtilde{O}(k)$ calls to $\mem_{\K}$ per round. This leads to a natural trade-off between computation (Oracle calls) and regret.
  
Finally, we note that there is another potential dependence in $d$ in the regret bounds through the condition number $\kappa$. In Section \ref{sec:applications}, we bound this quantity for the popular settings listed on Table \ref{tab:result}. We are able to show that $\kappa$ is often less than $d^{1/2}$ in many settings. Nevertheless, we note that $\kappa$ is present in our bounds because of our pessimistic upper bounds on the norms of the subgradients of the surrogate losses; i.e.~$\|\tilde \g_t\|\leq (1+\Delta_t + \kappa) \|\g_t\|$, for all $t\geq 1$ (see Lemma \ref{lem:mainreduction}). In fact, we do not expect the magnitude of $(\tilde \g_t)$ to be often much larger than that of $(\g_t)$ in practice; recall that $\tilde \g_t\neq \g_t$ only if the iterate $\w_t$ of the subroutine $\A$ is outside $\K$.

\paragraph{Implications for the stochastic and offline settings.}
The results we presented so far are also relevant in the stochastic and offline settings thanks to the online-to-batch conversion technique \cite{cesa2004, shalev2011}. In the latter settings (which we describe in more detail in Section \ref{sec:stochastic}), the losses $(\ell_t)$ are i.i.d.~and satisfy $\E[\ell_t]\equiv f$ for some fixed convex function $f\colon \K \rightarrow \reals$. Thus, if we let $\text{Regret}_T(\cdot)$ be the regret of Algorithm~\ref{alg:projectionfreewrapper} in response to $(\ell_t)$ and $x_*\in \argmin_{\x\in \K} f(\x)$, then the average iterate $\bar \x_T$ of Alg.~\ref{alg:projectionfreewrapper} after $T$ rounds satisfies
\begin{align}
\E[f(\bar \x_T)] - \inf_{\x \in \K} f(\x)\stackrel{(*)}{\leq}  \frac{1}{T}\sum_{t=1}^T \E[\ell_t(\x_t) -  \ell_t(\x_*)] \leq \frac{\text{Regret}_T(\x_*)}{T},   \label{eq:regretstoch}
\end{align}
where $(*)$ follows by Jensen's inequality and the fact that $\x_t$ is independent of $\ell_t$. Plugging the regret bounds of Alg.~\ref{alg:projectionfreewrapper} in the settings of Theorems \ref{thm:genthmeff}-\ref{thm:genthmprob} into \eqref{eq:regretstoch} leads to a $O(\sqrt{d\kappa^2/T})$ rate in stochastic and offline optimization, which is optimal in $T$\footnote{The rate $O(1/\sqrt{T})$ is optimal when no further assumptions on $f$ are made (other than convexity).}. Therefore, Alg.~\ref{alg:projectionfreewrapper} in the latter settings requires $\wtilde O(d\kappa^2/\epsilon^2)$ calls to the Membership Oracle to attain an $\epsilon$-suboptimal point in offline optimization. Thus, whenever $d\geq \Omega(\kappa^2/\epsilon^2)$, our algorithm is a viable alternative to those based on the cutting plane method, which require at least $O(d^2 \ln (1/\epsilon))$ calls to a Membership Oracle to achieve the same guarantee. 

When the objective function $f$ is $\beta$-smooth our adaptive bounds\footnote{Adaptive in the sense that they scale with $\sqrt{\sum_{t=1}^T \|\g_t\|^2}$ instead of $\sqrt{T}$.} in Theorems \ref{thm:genthm} and \ref{thm:genthmeff} together with an extension of the online-to-batch conversion analysis \cite[Corollary 6]{cutkosky2019} automatically imply a rate of $O(\beta/T+ \sigma/\sqrt{T})$, where $\sigma^2$ is the variance of the stochastic noise (see Section \ref{sec:stochastic} for a precise definition). Thus, in the offline setting (i.e.~$\sigma=0$), our algorithm achieves the fast $O(\beta/T)$ rate without knowing the smoothness parameter $\beta$. In Section \ref{sec:stochastic}, we show how this rate can be improved further.

We now move to the setting where the losses can be strongly convex.  
	\subsection{Algorithm for Strongly Convex Online Optimization}
	\label{sec:strong}
	In this section, we will build a subroutine $\A$ for Algorithm \ref{alg:projectionfreewrapper} that will enable the latter to ensure a logarithmic regret when the losses are strongly convex, while maintaining the worst-case $O(\sqrt{T})$ regret up to log-factors. The subroutine in question is displayed in Algorithm \ref{alg:projectionfreewrapperstong}. This subroutine is based on an exiting reduction due to \cite{cutkosky2018black} with the following key differences; I) we perform clipping of the subgradients $(\tilde \g_t)$ in Alg.~\ref{alg:projectionfreewrapperstong}; II) we slightly change the expression of the variables $(Z_t)$ in Line \ref{line:last} of Alg~\ref{alg:projectionfreewrapperstong}; and III) we use \freegrad{}, a \emph{parameter-free} OCO algorithm \cite{mhammedi2020}, as the underlying OCO subroutine. These differences will allow us to make our final algorithm scale-invariant in the sense that multiplying the losses by a positive constant does not change the outputs of the algorithm. Crucially, the algorithm does not require any prior scale information on the losses unlike many existing OCO algorithms (see e.g.~\cite{mhammedi2020}).    
	
		\begin{algorithm}[H]
		\caption{Subroutine $\A$ for Algorithm \ref{alg:projectionfreewrapper} for Strongly Convex Online Optimization.}
		\label{alg:projectionfreewrapperstong}
		\begin{algorithmic}[1]
			\Require Parameters $\epsilon,R>0$, with $R$ as in \eqref{eq:lowradius} and OCO Algorithm \freegrad{} (Alg.~\ref{alg:freegrad}).
			\Statex~~~~~~~~~~~ Input $(\x_t, f_t)$ from the Master Algorithm \ref{alg:projectionfreewrapper} at each round $t\geq 1$, where $\x_t \in \reals^d$ and $f_t\colon \reals^d\rightarrow \reals$.  
			\vspace{0.2cm} 
			\State  	 Initialize \freegrad{} with parameters $\epsilon, R>0$, and set $\bu_0$ to \freegrad's first output.
			\State Set $\tilde B_0=\epsilon$; $Z_0=2\epsilon^2$; and $\v_0=\bm{0}$.
			\For{$t=1,2,\dots$}
			\State Play $\w_t= \bu_{t+1} + \v_{t}/Z_{t}$ and observe $\tilde \g_t \in \partial f_t(\w_t)$. \algcomment{$f_t$ is specified by Alg.~\ref{alg:projectionfreewrapper}}
			\State  Set $\tilde B_t=\tilde B_{t-1}\vee \|\tilde \g_{t}\|$ and $\hat \g_{t} = \tilde \g_t \cdot \tilde B_{t-1}/\tilde B_t$.
			\State \label{line:bluestart}   Send linear loss $\w \mapsto  \inner{ \tilde \g_t}{\w}$ to \freegrad{} as the $t$th loss function.
			\State   Set $\bu_{t+1} \in \cB(R)$ to \freegrad's $(t+1)$th output given the history $((\bu_i, \w\rightarrow \inner{\tilde \g_i}{\w}))_{i\leq t}$.	
			\State Set $Z_t=Z_{t-1}+\|\hat \g_{t}\|^2 + \tilde B_t^2 - \tilde B_{t-1}^2$ and $\v_t =\v_{t-1}+ \|\hat \g_{t}\|^2 \x_{t}$. \algcomment{$\x_t$ is specified by Alg.~\ref{alg:projectionfreewrapper}} \label{line:last}
			\EndFor
		\end{algorithmic}
	\end{algorithm}
\noindent	The underlying subroutine \freegrad{} is displayed in Algorithm \ref{alg:freegrad} in Appendix \ref{app:OCO}. The key feature of \freegrad{} that allows us to adapt to strong convexity is that its regret is bounded from above by $O(\|\w\| \sqrt{V_T})$, up to log-factors in $\|\w\|$ and $V_T$, where $V_T=\sum_{t=1}^T\|\g_t\|^2$ and $(\g_t)$ are the observed subgradients at the iterates of the algorithm (the precise statement of the regret bound is differed to Appendix \ref{app:OCO}). Thus, similar to \ftrl{}, \freegrad{} also has an adaptive regret that can be much smaller than the worst-case $O(\sqrt{T})$; e.g.~when the losses are smooth \cite{srebro2010}. Furthermore, \freegrad's regret bound scales with the norm of the comparator $\|\w\|$, and thus becomes small for comparators close to the origin. This property will be crucial to prove a logarithmic regret for strongly convex losses \cite{cutkosky2018black}. We now state our main result for this subsection (the proof is in Appendix \ref{sec:mainstrongproof}):
	\begin{theorem}
		\label{thm:mainstrong}
		Let $\mu,B>0$, $\delta \in(0,1/3)$, and $\kappa\coloneqq R/r$, where $r$ and $R$ as in \eqref{eq:lowradius}. Suppose that Algorithm \ref{alg:projectionfreewrapper} is run with $\cO_{\K^\circ}\equiv \onedopt$ (Alg.~\ref{alg:optimizenew}); $\delta_t = \delta/t^2$, $\forall t \geq 1$; and sub-routine $\A$ set to Alg.~\ref{alg:projectionfreewrapperstong} with parameter $\epsilon>0$. Then, for any adversarial sequence of convex [resp.~$\mu$-strongly convex] $B$-Lipschitz losses $(\ell_t)$ on $\K$ the iterates $(\x_t)$ of Algorithm \ref{alg:projectionfreewrapper} satisfy $(\x_t)\subset \K$, and for all $T\geq 1$ and $\x\in \K$, we have, 
		\begin{align}
			\sum_{t=1}^T \E[\ell_t(\x_t)-\ell_t(\x)]\leq U_T R B \sqrt{T} +  R B U_T^2/\nu, \quad  \left[\text{resp.} \ \sum_{t=1}^T \E[\ell_t(\x_t)-\ell_t(\x)]\leq \left( \frac{R}{\nu}+ \frac{B}{2\mu} \right)  B  U_T^2 \right],\nn 
		\end{align}
		where $U_T = O\left(\nu d^{1/2} \ln (e+ \frac{\kappa d R  T B}{\epsilon\delta}) \right)$; $\nu \coloneqq  1/(R\wedge 1) + \kappa + \delta$; and $\g_t\in \partial \ell_t(\x_t)$, $\forall t\geq 1$.
	\end{theorem}
	\noindent We note that the instance of Algorithm \ref{alg:projectionfreewrapper} in the preceding theorem automatically adapts to the strong convexity constant $\mu>0$ of the losses $(\ell_t)$, where it achieves a logarithmic regret. For general convex losses, the algorithm ensures the optimal worst-case $O(\sqrt{T})$ regret up to log-factors. 
	
	\paragraph{Computational Complexity.}
	The instance of Algorithm \ref{alg:projectionfreewrapper} in Theorem \ref{thm:mainstrong} makes the same number of calls to the Oracle $\onedopt$ as in the setting of Theorems \ref{thm:genthmeff} and \ref{thm:genthmprob}. Thus, this instance makes at most $O(T \ln (d \kappa T/\delta))$ calls to the Membership Oracle $\mem_{\K}$. Further, we note that the sequence $(\delta_t)$ may be set to $\delta_t= \delta/t^{-1/2}$ [resp.~$\delta_t= \delta_t/t$], for all $t\geq 1$, for general convex [resp.~strong-convex] functions, allowing the Membership Oracle $\mem_{\K}$ to be less accurate while maintaining the same regret bound as Theorem \ref{thm:mainstrong} up to constant factors (see discussion after Theorem \ref{thm:genthmeff}). Finally, we restricted the losses to be $B$-Lipschitz in Theorem \ref{thm:mainstrong} only to simplify the proofs---the algorithm need not know $B$.

	\paragraph{Scale-invariance.} Though the instance of Alg.~\ref{alg:projectionfreewrapper} in Theorem \ref{thm:mainstrong} does not require a bound on the norm of the gradients (an information typically required by other OL algorithms \cite{orabona2016}), the algorithm is not scale-invariant; multiplying the sequence of losses $(\ell_t)$ by some fixed constant changes the iterates $(\x_t)$ of the algorithm. In general, this is an undesirable property for OL algorithms \cite{orabona2016}.

	Note also that the regret bound in Theorem \ref{thm:mainstrong} can technically be unbounded due to the fraction $B/\epsilon$ in the expression of $U_T$. This fraction can be arbitrarily large if $\epsilon$ (a parameter of the algorithm) is too small relative to the Lipschitz constant $B$. Such a problematic ratio has appeared in previous works such as \cite{ross2013normalized,Wintenberger2017,Kotlowski17,mhammedi2019,KempkaKW19}. To tame such a ratio, \cite{mhammedi2019} and \cite{mhammedi2020} presented a technique for certain OCO algorithms, such as $\metagrad$ \cite{Erven2016,van2021metagrad} and \freegrad, based on the idea of restarting these algorithms whenever the ratio between the maximum norm of the observed subgradients and the norm of the initial subgradient is too large. In Appendix \ref{sec:scaleinvariant}, we extend this technique and present a general reduction (Algorithm \ref{alg:red}) that makes a large class of Online Learning algorithms (including the instance of Alg.~\ref{alg:projectionfreewrapper} in the setting of Theorem \ref{thm:mainstrong}) scale-invariant and gets rid of problematic ratios in their regret bounds.
	
\paragraph{Optimality and link to stochastic optimization.}
The regret bounds in Theorem \ref{thm:mainstrong} are optimal in $T$. For the dependence in $d$ in the regret bounds, the computational trade-offs, and the link to stochastic and offline optimization, see the discussion at the end of Section \ref{sec:general}.	
	\subsection{Efficient Projection-Free Smooth Stochastic Optimization}
	\label{sec:stochastic}
	So far, we have mainly considered the setting where $(\ell_s)$ is a sequence convex losses that may be chosen in an adversarial fashion, and the goal was to choose a sequence of iterates $(\x_t)\subset \K$ such that the cumulative loss $\sum_{t=1}^T \ell_t(\x_t)$ is small. In this subsection, we are interested in minimizing a fixed convex differentiable function $f: \K \mapsto \reals$, with only access to a Stochastic Gradient Oracle for the function $f$. Formally, we assume there exists a $\sigma>0$ such that for any round $t\geq 1$ and some query point $\x_t \in \K$, we have access to a subgradient $\g_s \in \partial \ell_s(\x_t)$, where \begin{gather}\ell_s(\x)\coloneqq f(\x) + \inner{\x}{\bxi_s}, \label{eq:modifiedloss} \intertext{and $(\bxi_s \in \reals^d)$ are i.i.d.~random vectors satisfying}
		\E[\bxi_s] =\bm{0} \quad \text{and} \quad \E[\|\bxi_s\|^2] \leq \sigma^2, \quad \forall s\geq 1. \label{eq:stochsetting}
	\end{gather}
We will also assume that the function $f$ is $\beta$-smooth; that is, $f$ is differentiable\footnote{To avoid boundary issues, we assume (similar to \cite[Section B.4.1]{hiriart2004}) that $\K$ is contained in a open set $\Omega$ on which $f$ is differentiable.} on $\K$ and
	\begin{align}
		\forall \x,\y\in \K,\ \ f(\y) \leq f(\x) + \inner{\nabla f(\x)}{\y - \x} + \frac{\beta}{2} \|\y - \x\|^2.\label{eq:smoothness}
	\end{align}
	It will be instrumental to use the following consequence of \eqref{eq:smoothness} (see e.g.~\cite{srebro2010smoothness, levy2018online, cutkosky2018distributed}); if $f$ is $\beta$-smooth, then $\|\nabla f(\x) -\nabla f(\x_*)\|^2 \leq 2 \beta \cdot (f(\x)- f(\x_*))$, where $x_* \in \argmin_{\x\in \K} f(\x)$. Thus, when $\x_*$ is in the interior of $\K$, we have $\nabla f(\x_*)=\bm{0}$, and so it follows that
	\begin{align}
		\|\nabla f(\x)\|^2 \leq 2 \beta \cdot (f(\x)- f(\x_*)). \label{eq:crebro}
	\end{align}
	For the sake of clarity, we now summarize the assumptions we make on the loss process in this section:
	\begin{assumption}
		\label{ass:assum}
		The sequence $(\ell_s)$ in Algorithm \ref{alg:projectionfreewrappersmooth} satisfies \eqref{eq:modifiedloss} with I) $\bxi_1, \bxi_2, \dots \in \reals^d$ i.i.d.~vectors as in \eqref{eq:stochsetting}; and II) $f$ is $\beta$-smooth, for $\beta >0$, and satisfies $\text{int}(\K) \cap \argmin_{\x\in\K} f(\x)\neq \emptyset$, where $\text{int}(\K)$ denotes the interior of $\K$.
	\end{assumption}
\noindent	Without loss of generality (by making $R$ larger if necessary), we assume that there exists $R'\leq R$ such that:
		\begin{align}
	\cB(r)\subseteq \K	\subseteq \cB(R') \subseteq \cB(R)/(1+\nu), \quad \text{where} \quad 	 \nu\coloneqq 4 \sqrt{2}(1+d\kappa), \quad  \text{and} \quad \kappa \coloneqq R'/r. \label{eq:newnewrad}
	\end{align}
	We now state the main result of this section (the proof is in Appendix \ref{sec:genthmsmoothproof}): 
	\begin{theorem}
		\label{thm:genthmsmooth}
		Let $\delta \in(0,1/3)$, and $r,R,R',\nu$, and $\kappa$ be as in \eqref{eq:newnewrad}. Further, suppose that Alg.~\ref{alg:projectionfreewrapper} is run with $\cO_{\K^\circ}\equiv \onedopt$ (Alg.~\ref{alg:optimizenew}); $\delta_t = \delta/t^3$, $\forall t \geq 1$; and sub-routine $\A$ set to Alg.~\ref{alg:projectionfreewrappersmooth}. Then, under Assumption \ref{ass:assum}, the iterates $(\x_t)$ of Alg.~\ref{alg:projectionfreewrappersmooth}, satisfy $(\x_t)\subset \K$, and for all $T\geq 1$,
		\begin{align}
			\E\left[f(\x_T)-f(\x_*)\right] & \le  \frac{2\nu \epsilon  R'+\nu^2 \beta (R')^2U_T +  3\delta  R (\ln T+6d) \sqrt{\sigma^2  + 2\beta R'}}{T^2} + \frac{2 \nu R'\sigma\sqrt{U_T}}{\sqrt{T}},\nn
		\end{align}
		where $\x_* \in \argmin_{\x\in\K} f(\x)$ and $U_T \coloneqq \ln \left(1+2\epsilon^{-2} (\sigma^2 +2 R' \beta)T^3 \right)$.
	\end{theorem}
\noindent The proof of Theorem \ref{thm:genthmsmooth} is based on an extension of a result due to \cite{cutkosky2019}. The instance of Algorithm~\ref{alg:projectionfreewrapper} in Theorem \ref{thm:genthmsmooth} invokes $\onedopt(\cdot;\delta_t)$ at each iteration $t$. Thus, in light of Remark \ref{rem:complexityoned} in Section \ref{sec:gauge}, the algorithm makes at most $O(T\ln (dT\kappa/\delta)$ calls to the Membership Oracle $\mem_{\K}$ after $T$ rounds.

\paragraph{Optimality of the rate and application to the offline setting.} The rate in Theorem \ref{thm:genthmsmooth} is optimal in $T$ and implies the fast $O(d^2 \beta  \kappa^2/T^2)$ rate in the offline smooth setting (i.e.~$\sigma =0$), which is also optimal in $T$. We note that the leading $d^2$ in this rate can be removed if one uses $\cO_{\K^\circ}=\O_{\K^\circ}$ (Alg.~\ref{alg:optimize}) instead of $\cO_{\K^\circ}=\onedopt$ in Algorithm \ref{alg:projectionfreewrapper} at the cost of more calls to the Membership Oracle; $\wtilde{O}(d)$ calls to $\mem_{\K}$ per round instead of $\wtilde{O}(1)$ (see discussion at the end of Section \ref{sec:general} on such computational trade-offs). Thus, Algorithm~\ref{alg:projectionfreewrapper} in the setting of Theorem~\ref{thm:genthmsmooth} with $\cO_{\K^\circ}$ set to $\O_{\K^\circ}$ (instead of $\onedopt$) reaches an $\epsilon$-sub-optimal point in offline smooth optimization after $\wtilde O(d\kappa/\sqrt{\epsilon})$ calls to $\mem_{\K}$. Since state-of-the-art algorithms based on the cutting plane method require $O(d^2 \ln (1/\epsilon))$ calls to a Membership Oracle to reach an $\epsilon$-sub-optimal point \cite{lee2018efficient}, our algorithm provides a viable alternative to the latter whenever $d \geq  \Omega(\kappa/\sqrt{\epsilon})$ and the objective function is smooth. As we shall see in Section \ref{sec:applications}, $\kappa$ is less than $d^{1/2}$ in many settings of interest. We also recall that the presence of $\kappa$ in our bounds is due to an over conservative upper bound on the norms of the subgradients of the surrogate losses, and so we expect the rates of our algorithms to scale better with $d$ in practice. 
	
\noindent 

			\begin{algorithm}[H]
		\caption{Subroutine $\A$ for Algorithm \ref{alg:projectionfreewrapper} for Stochastic Convex Optimization.}
		\label{alg:projectionfreewrappersmooth}
		\begin{algorithmic}[1]
			\Require $r,R,R',\nu$, and $\kappa$ as in \eqref{eq:newnewrad}.
			\Statex~~~~~~~~~~~ Input $(\x_t, f_t)$ from the Master Algorithm \ref{alg:projectionfreewrapper} at each round $t\geq 1$, where $\x_t \in \reals^d$ and $f_t\colon \reals^d\rightarrow \reals$.  
			\vspace{0.2cm} 
			\State  	 Initialize \ftrl{} with parameter $R'>0$ and set $\bu_1$ to \ftrl{}'s first output (i.e.~$\bu_1=\bm{0}$).
			\State Set $\Lambda_1=0$; $Z_0=\epsilon^2$; and $\w_1=\bm{0}$.
			\For{$t=1,2,\dots$}
			\State Output $\w_t$ and observe $\tilde \g_t \in \partial f_t(\w_t)$. \algcomment{$f_t$ is provided by Alg.~\ref{alg:projectionfreewrapper}}
\State   Set $\bu_{t+1} \in \cB(R')$ to \ftrl{}'s $(t+1)$th output given the history $((\bu_i,\w\mapsto i\cdot \inner{\tilde \g_i}{\w}))_{i\leq t}$.	
			\State Set $Z_t = Z_{t-1} + \Lambda_{t} \|\g_{t}\|^2$ and $\eta_t = \nu R'/\sqrt{Z_t}$.
			\State Set $\Lambda_{t+1} = \Lambda_{t}+ t+1$ and $\mu_{t+1} = (t+1)/\Lambda_{t+1}$.
			\State \label{line:blueend} Set $\w_{t+1} = (1-\mu_{t+1}) ({\x}_{t}- \eta_{t}\g_{t}) + \mu_{t+1} \bu_{t+1}$.	\algcomment{$\x_t$ is provided by Alg.~\ref{alg:projectionfreewrapper}}	
			\EndFor
		\end{algorithmic}
	\end{algorithm}

		\section{Efficient Gauge Projections using a Membership Oracle $\cM_{\K}$}
	\label{sec:gauge}
	In this section, we build explicit algorithms that use $\mem_{\K}$ to efficiently approximate $\gamma_{\K}$ and its subgradients. As a result of this (and thanks to Lemma \ref{lem:projectiononK}), we show that Gauge projections can be performed efficiently for any bounded convex set $\K$ satisfying Assumption \ref{ass:assum} using a Membership Oracle---requiring only $\wtilde{O}(d)$ calls to the latter, where $\wtilde{O}$ hides log factors in the tolerated approximation error. Thanks to Lemma \ref{lem:projectiononK}, Gauge projections have a very simple interpretation; the projection of a point $\w\notin \K$ onto $\K$ is the point of intersection of the ray $\{\lambda \w: \lambda \geq 0\}$ and the boundary of $\K$ (see Figure \ref{fig:proj}). Such a point always exists under Assumption \ref{ass:assum}. Gauge projections are all we need to ensure the iterates of our new algorithms are within $\K$ thanks to the carefully designed surrogate losses in Section \ref{sec:oco}. Our starting point is Lemma \ref{lem:projectiononK} which we now prove:
	\begin{proof}[{\bf Proof of Lemma \ref{lem:projectiononK}}]
		Suppose that $\gamma_{\K}(\w)>1$ (note that by Assumption \ref{ass:assum}, we have $\gamma_{\K}(\w)<+\infty$). Note that this is equivalent to $\w\not\in \K$ by the definition of the Gauge function. We will show that $\w/\gamma_{\K}(\w) \in \argmin_{\bu\in \K}\gamma_{\K}(\w-\bu)$, which is equivalent to showing that
		\begin{align}
			\w/\gamma_{\K}(\w)\in \argmin_{\bu \in \reals^d} \left\{ \gamma_{\K}(\w-\bu)+ \iota_{\K}(\bu)\right\}  =  \argmin_{\bu \in \reals^d}\left\{ \sigma_{\K^\circ}(\w - \bu) + \iota_{\K}(\bu)\right\}, \nn
		\end{align}
		where the last equality follows by Lemma \ref{lem:properties}-(a). Thus, $\w/\gamma_{\K}(\w) \in \argmin_{\bu\in \K}\gamma_{\K}(\w-\bu)$ if 
		\begin{align}
			\bm{0} \in \partial \phi(\w/\gamma_{\K}(\w)), \quad \text{where} \quad \phi(\bu)\coloneqq \sigma_{\K^\circ}(\w - \bu) + \iota_{\K}(\bu).\label{eq:alternative} 
		\end{align}
		Using the sub-differential chain-rule and the fact that $\partial \iota_{\K}(\bu)=\cN_{\K}(\bu)$ (see e.g.~\cite{hiriart2004}), where $\cN_{\K}(\bu)$ is the normal cone at $\bu$ (see Definition \ref{def:normal}), we have \begin{align}\partial \phi(\bu) = - \partial\sigma_{\K^\circ}(\w-\bu) + \cN_{\K}(\bu).\label{eq:subdiff}
		\end{align}
		Let $\s_* \in \partial \sigma_{\K^\circ}(\w- \w/\gamma_{\K}(\w))$. We will show that $\s_*\in \cN_{\K}(\w/\gamma_{\K}(\w))$, where we recall that \begin{align} \cN_{\K}(\w/\gamma_{\K}(\w))= \{\s \in \reals^d : \inner{\s}{\y} \leq \inner{\s}{\w/\gamma_{\K}(\w)}, \forall \y \in \K \},\label{eq:normalcone}
		\end{align} which will imply \eqref{eq:alternative} thanks to \eqref{eq:subdiff}. 
		By Lemma \ref{lem:properties}-(b), we have $\s_* \in\partial \sigma_{\K^\circ}(\w-\w/\gamma_{\K}(\w)) = 	\partial \sigma_{\K^\circ}(\w) = \argmax_{\s\in \K^\circ} \inner{\s}{\w},$ and so by Lemma \ref{lem:properties}-(a), we get $\inner{\s_*}{\w} =\sigma_{\K^\circ}(\w)= \gamma_{\K}(\w)$. On the other hand, by definition of the polar set $\K^\circ$, we have $\inner{\s_*}{\y}\leq 1$, for all $\y \in \K$. Using this and the fact that $\inner{\s_*}{\w}= \gamma_{\K}(\w)$, we get 
		\begin{align}
			\forall \y\in \K,\quad 	\inner{\s_*}{\y}\leq 1 =  \inner{\s_*}{\w}/\gamma_{\K}(\w)=\inner{\s_*}{\w/\gamma_{\K}(\w)}.\nn
		\end{align}
		Combining this with the definition of the normal cone $\cN_{\K}(\w/\gamma_{\K}(\w))$ in \eqref{eq:normalcone}, we get that $\s_* \in \cN_{\K}(\w/\gamma_{\K}(\w))$. This shows that $\s_* \in \partial \sigma_{\K^\circ}(\w-\w/\gamma_{\K}(\w))\cap  \cN_{\K}(\w/\gamma_{\K}(\w))$, and so $\bm{0}\in \partial \phi(\w/\gamma_{\K}(\w))$ by \eqref{eq:subdiff}. This in turn implies that $$\w/\gamma_{\K}(\w) \in \argmin_{\bu\in \K} \gamma_{\K}(\w-\bu).$$ 
		From this result and the definition of the proximity function $S_{\K}(\w)$, we get 
		\begin{align}
			S_{\K}(\w)  = \gamma_{\K}(\w - \w/\gamma_{\K}(\w))
			=\gamma_{\K}(\w \cdot (1-1/\gamma_\K(\w))) =(1-1/\gamma_{\K}(\w)) \cdot  \gamma_{\K}(\w) = \gamma_{\K}(\w)-1,\label{eq:express}
		\end{align}
		where the penultimate inequality follows by the current assumption that $\gamma_{\K}(\w) >1$ and the positive homogeneity of Gauge functions (Lemma \ref{lem:properties}-(b)). It remains to consider the case where $\gamma_{\K}(\w)\leq 1$. In this case, we have $\w\in \K$, and so 
		\begin{align}
			S_{\K}(\w) =\inf_{\bu\in \K} \gamma_{\K}(\w -\bu) = \gamma_{\K}(\w-\w)= 0.\label{eq:constant}
		\end{align}
		In combination with \eqref{eq:express}, \eqref{eq:constant} shows that $S_{\K}(\w)=0\vee (\gamma_{\K}(\w)-1)$, which is a convex function (since it is the maximum of two convex functions). Finally, \eqref{eq:express} [resp.~\eqref{eq:constant}] shows that $\partial S_{\K}(\w) = \argmax_{\s\in \K^{\circ}} \inner{\s}{\w}$ when $\w\not\in \K (\equiv \gamma_{\K}(\w)> 1)$ [resp.~$\partial S_{\K}(\w)= \{\bm{0}\}$ when $\w\in \K$].
	\end{proof}
	\noindent	As mentioned earlier, Lemma \ref{lem:projectiononK} shows the crucial fact that $\w/\gamma_{\K}(\w) \in \Pi^{\mathrm{G}}_{\K}(\w)$, for any $\w\not\in\K$, and so to perform approximate Gauge projections (which we need to build our efficient algorithms), it suffices to approximate $\gamma_{\K}$. We technically also need to compute approximate subgradients of $S_{\K}$, which according to Lemma  \ref{lem:projectiononK} reduces to linear optimization on $\K^\circ$. Since we do not assume access to a Linear Optimizer Oracle over $\K^\circ$, we will estimate subgradients of $S_{\K}$ using our Membership Oracle $\mem_{\K}$. 
	
	\begin{algorithm}[h]
		\caption{$\mathsf{GAU}_{\K}$: Approximate Gauge Function using Membership Oracle and the Bisection Method. }
		\label{alg:gauge}
		\begin{algorithmic}[1]
			\Require Input point $\w \in \cB(6R/5)$ with $R$ as in \eqref{eq:lowradius}. 
			\Statex~~~~~~~~~~ Input $\delta \in(0,1)$.
			\Statex~~~~~~~~~~ $\epsilon$-Approximate Membership Oracle $\mem_{\K}(\cdot;\epsilon)$ for $\K$ and $\epsilon>0$ (see Definition \ref{def:mem}).
			\vspace{0.2cm} 
			\State Set $\epsilon= \delta r /( 4\kappa)^2$. \algcomment{$\kappa \coloneqq R/r$, where $r$ and $R$ are as in \eqref{eq:lowradius}}
			\If{$\mem_{\K}(2\w;\epsilon)=1$ or $\|\w\| \leq  r/2$}   \label{line:memset}
			\State Return $\tilde \gamma=0$.\label{line:memsetnext}
			\EndIf
			\State Set $\alpha=0$, $\beta=2$, and $\mu=(\alpha+\beta)/2$. 
			\While{$\beta-\alpha> {\delta}/({8\kappa^2})$} \algcomment{$\kappa \coloneqq R/r$, where $r$ and $R$ are as in \eqref{eq:lowradius}}
			\State Set $\alpha=\mu$ if $\mem_{\K}(\nu \w; \epsilon)=1$; and $\beta=\mu$ otherwise.\label{line:cond}
			\State Set $\mu=(\alpha+\beta)/2$.
			\EndWhile
			\State Return $\tilde \gamma = (\alpha - \delta/(8\kappa^2))^{-1}$. 
		\end{algorithmic}
	\end{algorithm}
	\subsection{Approximating the Gauge Function $\gamma_{\K}$ using $\cM_{\K}$}  By definition of the Gauge function, we have for any $\w\in \reals^d$ \begin{align}
		\gamma_{\K}(\w) = \inf \{ \lambda \in \reals_{\geq 0} \mid \w \in \lambda \K \}= 1/\sup \{ \nu \in \reals_{\geq 0} \mid  \nu \w \in \K \}. \label{eq:gauge}
	\end{align} 
	Using the Membership Oracle $\mem_{\K}$, we can approximate the largest $\nu\geq 0$ such that $\nu \w \in \K$ via bisection, which will lead to an approximation of $\gamma_{\K}(\w)$ by \eqref{eq:gauge}. This is exactly what we do in Algorithm \ref{alg:gauge}. We now state the guarantee of Algorithm \ref{alg:gauge} (the proof is Appendix \ref{sec:gaugeproof}):
	\begin{lemma}[$\gau_{\K}\colon$ Approximate Gauge Function]
		\label{lem:gauge}
		Let $r,R>0$ be as in \eqref{eq:lowradius}. For any $\delta \in(0,1)$ and $\w\in \cB(6R/5)$, the output $\tilde\gamma= \gau_{\K}(\w;\delta)$ of Algorithm \ref{alg:gauge} satisfies
		\begin{align}
			\gamma_{\K}(\w) \leq \tilde \gamma  \leq  \gamma_{\K}(\w)+ \delta, \quad \text{if }\  \gamma_{\K}(\w)\geq 9/16\ \text{ or }\ \tilde \gamma\geq 1. \nn
		\end{align} 
		Furthermore, Algorithm \ref{alg:gauge} calls the Membership Oracle $\mem_{\K}(\cdot; r\delta/(4\kappa)^2)$ at most $\ceil{\log_2((4\kappa)^2/\delta)}+1$ times.
	\end{lemma} 
	\begin{algorithm}
		\caption{$\O_{\K^\circ}$: Approximate Linear Optimization Algorithm on $\K^\circ$.}
		\label{alg:optimize}
		\begin{algorithmic}[1]
			\Require Input point $\w \in \cB(R)$ and $\delta \in(0,1/3)$. 
			\vspace{0.2cm} 
			\State Set $\varepsilon = \frac{r^2 \delta^3}{10^3 d^{7/2} R^2}$, $\nu_1=\frac{r \delta}{10 d}$, and $\nu_2=\sqrt{\frac{\varepsilon \nu_1 r}{d^{1/2} }}.$ \algcomment{$r$ and $R$ are as in \eqref{eq:lowradius}}
			\State Sample $\bu \in \cB_{\infty}(\w,\nu_1)$ and $\z\in \cB_{\infty}(\bu,\nu_2)$ independently and uniformly at random.
			\For{$i=1,2,\dots,d$}
			\State Let $\w'_i$ and $\w_i$ be the end point of the interval $\cB_{\infty}(\bu,\nu_2) \cap \{\z+\lambda \e_i : \lambda \in \reals\}$.
			\State Set $\tilde s_i =\frac{1}{2\nu_2} (\mathsf{GAU}_{\K}( \w'_i;\varepsilon) -\mathsf{GAU}_{\K}(\w_i;\varepsilon))$. \algcomment{$\gau_{\K}$ is as in Algorithm \ref{alg:gauge}}  
			\EndFor
			\State Set $ \tilde \s = (\tilde s_1, \dots,  \tilde s_d)^\top$ and $ \tilde\gamma = \gau_{\K}(\w;\delta)$.
			\State Return $(\tilde \gamma, \tilde \s)$. 
		\end{algorithmic}
	\end{algorithm}
	\subsection{Approximating the Subgradients of $\gamma_{\K}$ using $\cM_{\K}$} In addition to approximating $\gamma_{\K}$, we will also need to approximate its subgradients, which would then lead to approximate subgradients of the Gauge distance function $S_{\K}$ by Lemma \ref{lem:projectiononK}. The lemma also implies that approximating subgradients of $\gamma_{\K}$ essentially comes down to performing linear optimization on $\K^\circ$. 
Algorithm \ref{alg:optimize} ($\O_{\K^\circ}$), which is based on \cite[Alg.~2]{lee2018efficient}, uses $\gau_{\K}$ and a random partial difference in each coordinate to approximate the subgradients of $\gamma_{\K}$. In the next proposition, we state the precise guarantee of the algorithm. The proof of the proposition, which is in Appendix \ref{sec:subgradientproof}, is somewhat technical and relies heavily on existing results due to \cite{lee2018efficient} that we restate in Appendix \ref{app:opt}. 
	
	\begin{proposition}[$\O_{\K^\circ}\colon$ Approximate LOO on $\K^\circ$]
		\label{prop:subgradient}
		Let $\kappa \coloneqq R/r$ with $r,R>0$ as in \eqref{eq:lowradius}. For any $\w\in \cB(R)$ and $\delta\in (0,1/3)$, let $(\tilde \gamma,\tilde \s)$ be the output of Alg.~\ref{alg:optimize} with input $(\w,\delta)$. Then, $\|\tilde \s\|_{\infty}< +\infty$ almost surely, and there exists a positive random variable $\Delta\in[0,{15^2 d^{4}\kappa^3}{\delta^{-2}}]$ satisfying $\E[\Delta]\leq \delta$, and such that if $\tilde \gamma\geq 1$, then 
		\begin{align}
			\begin{array}{c}
				\gamma_{\K}(\bu) \geq \gamma_{\K}(\w)+ \inner{\tilde \s}{\bu - \w}  -   \Delta \cdot\max ( 1, {\|\bu\|}/{R}) , \quad \forall \bu\in \reals^d;  \\[2pt]
				\| \tilde \s\|_{\infty} \leq \dfrac{\delta}{R} +\dfrac{1}{r};   \quad  \|\tilde \s\|\leq \dfrac{\Delta}{R} + \dfrac{1}{r} ; \ \  \text{and} \ \  \|\tilde \s\|^2 \leq  \left(\dfrac{2}{r}+\dfrac{\delta}{R}\right)\dfrac{ \Delta}{R} +\dfrac{1}{r^2}.
			\end{array}
			\label{eq:bisection} 
		\end{align}
	\end{proposition}      
	
	\begin{remark}[Complexity of $\O_{\K^\circ}$]
		\label{rem:complexity} In the setting of Proposition \ref{prop:subgradient}, Algorithm \ref{alg:optimize} $(\O_{\K^\circ})$ requires $2d\cdot (\ceil{\log_2((4\kappa)^2/\varepsilon)}+1)$ calls to the Membership Oracle $\mem_{\K}(\cdot; r\varepsilon/(4\kappa)^2)$, where $\varepsilon= r^2 \delta^3/(10^3 d^{7/2} R^2)$, and one call to $\mem_{\K}(\cdot; r\varepsilon/(4\kappa)^2)$. This follows by Lemma \ref{lem:gauge} and the fact that $\O_{\K^\circ}$ makes $2\cdot d$ calls to $\gau_{\K}(\cdot;\varepsilon)$ and one call to $\gau_{\K}(\cdot;\delta)$. 
	\end{remark}
	
	\noindent In light of Remark \ref{rem:complexity}, Proposition \ref{prop:subgradient} implies that approximating the subgradient of $\gamma_{\K}$ at a point $\w\in \cB(R)$ up to some random error $\Delta\geq 0$ satisfying $\E[\Delta]\leq \delta$, requires only $O(d \ln (d \kappa/\delta))$ calls to the Membership Oracle $\mem_{\K}$. This means that the approximation error decreases exponentially fast with the number of calls to $\mem_{\K}$, which allows us to build our efficient projection-free algorithms in Section \ref{sec:oco}.	
	
	\begin{algorithm}
		\caption{$\onedopt$: `One-Dimensional' Stochastic Version of $\O_{\K^\circ}$.}
		\label{alg:optimizenew}
		\begin{algorithmic}[1]
			\Require Input point $\w \in \cB(R)$ and $\delta \in(0,1/3)$. 			
			\vspace{0.2cm} 
			\State Set $\varepsilon = \frac{r^2 \delta^3}{10^3 d^{7/2} R^2}$, $\nu_1=\frac{r \delta}{10 d}$, and $\nu_2=\sqrt{\frac{\varepsilon \nu_1 r}{d^{1/2} }}.$ \algcomment{$r$ and $R$ are as in \eqref{eq:lowradius}}
			\State Sample $I\in [d]$, $\bu \in \cB_{\infty}(\w,\nu_1)$, and $\z\in \cB_{\infty}(\bu,\nu_2)$ independently and uniformly at random.
			\State Let $\w'_I$ and $\w_I$ be the end point of the interval $\cB_{\infty}(\bu,\nu_2) \cap \{\z+\lambda \e_I : \lambda \in \reals\}$.
			\State \label{line:estimate}Set $\tilde s_I =\frac{1}{2\nu_2} (\mathsf{GAU}_{\K}( \w'_I;\varepsilon) -\mathsf{GAU}_{\K}(\w_I;\varepsilon) )$. \algcomment{$\gau_{\K}$ is as in Algorithm \ref{alg:gauge}} 
			\State Set $ \hat \gamma = \gau_{\K}(\w;\delta)$ and $ \hat \s= d \tilde s_I\cdot  \e_I$.
			\State Return $(\hat \gamma, \hat \s)$. 
		\end{algorithmic}
	\end{algorithm}
	
	In some settings, calling a Membership Oracle $\Omega(d)$ times per iteration might still be too expensive. A way around this is to use a stochastic version $\onedopt$ of $\O_{\K^\circ}$ that calls $\mem_{\K}$ at most $\wtilde O(1)$ times and has the same output as $\O_{\K^\circ}$ in expectation. Algorithm \ref{alg:optimizenew}, \onedopt, achieves this by randomly sampling a coordinate $I\in[d]$ for estimating the subgradient of $\gamma_{\K}$ (Line \ref{line:estimate} of Alg.~\ref{alg:optimizenew}) and using importance weights. We now state the guarantee of Algorithm \ref{alg:optimizenew} (the proof is in Appendix \ref{sec:celebproof}):
	\begin{lemma}
		\label{lem:celeb}
		Let $\delta\in (0,1/3)$, $\w\in \cB(R)$, and $\kappa \coloneqq R/r$, where $r,R>0$ are as in \eqref{eq:lowradius}. Further, let $(\tilde \gamma ,  \tilde \s)=\O_{\K^\circ}(\w;\delta)$ and $( \hat\gamma,   \hat\s)=\onedopt(\w;\delta)$ (Alg.~\ref{alg:optimizenew}). Then, $\|\hat \s\|<+\infty$ a.s.; $\tilde \gamma =\hat \gamma$; and if $\hat \gamma \geq 1$, it follows that
		\begin{align}
			\E[  \hat \s] =  \E[\tilde \s],   \quad    \text{and}\quad     \E[\|  \hat \s\|^2]\leq  d\cdot \left({1}/{r}+ {\delta}/{R} \right)^2. \nn   	
		\end{align}
	\end{lemma}
	\begin{remark}[Complexity of $\onedopt$]
		\label{rem:complexityoned} In the setting of Lemma \ref{lem:celeb}, Algorithm \ref{alg:optimizenew} $(\onedopt)$ requires  $2\cdot (\ceil{\log_2((4\kappa)^2/\varepsilon)}+1)$ calls to the Membership Oracle $\mem_{\K}(\cdot; r\varepsilon/(4\kappa)^2)$, where $\varepsilon= r^2 \delta^3/(10^3 d^{7/2} R^2)$, and one call to $\mem_{\K}(\cdot;\delta)$. This follows by Lemma \ref{lem:gauge} and the fact that $\onedopt$ calls $\gau_{\K}(\cdot;\varepsilon)$ twice and $\gau_{\K}(\cdot;\delta)$ once. 
	\end{remark}
	
	\noindent Note that $\onedopt$ randomly selects a single coordinate $I$ along which to estimate the subgradient of the Gauge function $\gamma_{\K}$. Generalizing this idea to $k\leq d$ coordinates, one can build a version of $\onedopt$, call it $k\text{D}$-$\text{OPT}_{\K^\circ}$, that samples $i\in [\lfloor d/k \rfloor]$ uniformly at random and selects coordinates $\{i k + j-1: j \in [k]\}\cap [d]$ along which to estimate the subgradients of $\gamma_{\K}$. In this case, $k\text{D}$-$\text{OPT}_{\K^\circ}$ makes $\wtilde{O}(k)$ calls to the Membership Oracle and would lead to a natural trade-off between computation (i.e.~Oracle calls) and regret (see discussion in Section \ref{sec:general}).

		\section{Applying the Projection-Free Reduction in Practice}
		\label{sec:applications}
	In this section, we consider various popular settings where our algorithm may be applied. We note that Assumption \ref{ass:assum}, i.e.~the condition that $\cB(r)\subseteq \K \subseteq \cB(R)$ may not always be satisfied in these settings, but one can easily reparametrize the problem to satisfy the condition. We will derive explicit reparametrizations for the popular settings studied in \cite{hazan2012,jaggi2013} (those in Table \ref{tab:comp}). But, first we will present a general reparametrization recipe.
	
	Suppose that the available losses $(f_t)$ are defined on a convex set $\cK\subset \reals^d$ that does not necessarily satisfy \eqref{eq:lowradius} (e.g.~if $\cK$ has an empty interior like the simplex). Further, suppose that we have a Membership Oracle $\cM_{\cK}$ for $\cK$ and a subgradient Oracle for $(f_t)$. We will show that one can easily reparametrize the problem on a set $\K$ satisfying \eqref{eq:lowradius} and whose Membership Oracle can easily be constructed from that of $\cK$.   
	
	Let $\cH$ be the subspace generated by the span of $\cK$; that is, $
	\cH \coloneqq \{\lambda \x : \lambda \in \reals \text{ and } \x \in \text{conv}\cK\}.$
	Further, let $(\bu_i)_{1\leq i\leq m}$, $m\leq d$, be an orthogonal basis of $\cH$ (this can be computed offline once for the given problem), and let $\bm{c}$ be a point in the relative interior of $\cK$. Then, one can work with the surrogate losses $\ell_t: \reals^m \rightarrow \reals \cup \{+\infty\}$ defined by 
	\begin{align}
		\ell_t(\x)=\left\{ \begin{array}{ll}   f_t\left(\bm{c}+\sum_{i=1}^m x_i\bu_i  \right) , & \text{if } \bm{c}+ \sum_{i=1}^m x_i\bu_i  \in \cK; \\ +\infty, & \text{otherwise}.
		\end{array}\right.\nn
	\end{align}
	The convexity of $\ell_t$ follows immediately from that of $f_t$. We also note that since $\bm{c}$ was chosen in the relative interior of $\cK$, there exists $r,R>0$ such that the domain $\K \subset \reals^m$ of $\ell_t$ satisfies $\cB(r)\subseteq \K\subseteq \cB(R)$; that is, Assumption \ref{ass:assum} is satisfied. We now show that a Membership Oracle for $\K$ [resp.~subgradient Oracle for $\ell_t$] can easily be constructed from a Membership Oracle for $\cK$ [resp.~subgradient Oracle for $f_t$].
	
	Starting with the Membership Oracle: for any $\x\in \K$, a Membership Oracle $\cM_{\K}$ for $\K$ can be implemented as $\cM_{\K}(\x)= \cM_{\cK}(\bm{c} + \sum_{i=1}^m \bu_i x_i)$. Now, given a subgradient Oracle for $f_t$, it is also easy to get a subgradient Oracle for $\ell_t$; for any $\x\in \reals^m$ such that $\y \coloneqq\sum_{i=1}^m \bu_i x_i \in \cK$, we have $\partial \ell_t(\x) = \{ \sum_{i=1}^m \inner{\bu_i}{\bm{\zeta}} \bu_i: \bm{\zeta} \in \partial f_t(\y)\}$. Using this method, one needs to perform $O(m d)$ arithmetic operations at each round to compute a subgradient of $\ell_t$. It is possible to avoid this computational overhead by, instead, using a finite difference approach for estimating subgradients of $\ell_t$, requiring only $2m$ calls to a value Oracle for $f_t$ per round. Lemma \ref{lem:separate_conv_funcnew} provides the means for doing this\footnote{Technically, Lemma \ref{lem:separate_conv_funcnew} provides a way of approximating the subgradients of a function whose domain is unconstrained. However, one can extend the result to the constrained case by a careful treatment of the region around the boundary of the set.}. 
	
	We now show that in many practice cases there are natural parametrizations that do not require any expensive pre-processing step such as identifying a basis for the span of $\cH$. In Table \ref{tab:cond}, we summarize the upper bounds we derive on the condition number $\kappa$ for different sets of interest after reparametrization. In Table \ref{tab:comp}, we summarize the computational complexity of a Membership Oracle for these sets. 
	
\begin{table}[h]
\centering
\begin{tabular}{lll}
\hline
\bf{Domain} &  \multicolumn{2}{l}{Upper bound on $\kappa$ post-reparametrization}    \\
\hline 
\hline
$\ell_p$-ball in $\reals^d$	&  $d^{|1/p-1/2|}$ &   $O(d^{|1/p-1/2|})$, where $\K \subset\reals^d$  \\ 
\hline
Simplex $\Delta_d \in\reals^d$ & $2d$  &  $O(d)$, where $\K \subset\reals^d$ \\
\hline
Trace-norm-ball in $\reals^{m\times n}$ & $\sqrt{m \wedge n}$  &  $O(d^{1/4})$, where $\K \subset\reals^d$  \\
\hline
Op-norm-ball in $\reals^{m\times n}$ & $\sqrt{m \wedge n}$ &  $O(d^{1/4})$, where $\K \subset\reals^d$    \\
\hline
Conv-hull of Permutation Matrices in $\reals^{n\times n}$ &  $\sqrt{5}n$ &   $O(d^{1/2})$, where $\K \subset\reals^d$ \\ 
\hline
Convex-hull of Rotation Matrices in $\reals^{n\times n}$ &  $2n^{1/2}$  &  $O(d^{1/4})$, where $\K \subset\reals^d$  \\
\hline
PSD matrices in $\reals^{n\times n}$ with unit trace
&  $8 n^2$  &  $O(d)$, where $\K \subset\reals^d$  \\
\hline
PSD matrices in $\reals^{n\times n}$ with diagonals $\leq$ 1
  &  $4 n^{3/2}$ &  $O(d^{3/4})$, where $\K \subset\reals^d$  \\
\hline
The flow polytope with (\#nodes,\ \#edges)=($d$,$m$)
 & \multicolumn{2}{c}{\emph{Problem dependent}}    \\
\hline
The matroid polytope for matroid $M$; \#elem.=$d$ &  \multicolumn{2}{c}{\emph{Problem dependent}}   \\
 \hline
 \hline
\end{tabular}
\caption{Upper bounds on the condition number $\kappa$ for the different settings of Table \ref{tab:comp} after reparametrization.}
	\label{tab:cond}
\end{table}
	
	\subsection{$\ell_p$-Norm Balls} Consider the setting where the losses are defined on $\K=\{ \x: \|\x\|_p\leq 1\}$.  In this case, the span of $\K$ is $\reals^d$ and we can pick $\bm{c} =\bm{0}$. For $1\leq p \leq 2$ Assumption \ref{ass:assum} is satisfied with $r= d^{1/2 - 1/p}$ and $R= 1$; this follows by the fact that the $\ell_p$ norm satisfies $	\|\x\| \leq \|\x\|_p  \leq d^{1/p-1/2}\|\x\|, \forall \x\in \reals^d.$ When $2\leq p$ Assumption \ref{ass:assum} is satisfied with $r= 1$ and $R= d^{1/2 - 1/p}$. Thus, in either case, the condition number is $$\kappa=R/r\leq d^{|1/p - 1/2|},$$ and there is no need to reparametrize. The operation needed for the Membership Oracle is simply computing $\|\x\|_p$, whereas linear optimization on $\K^\circ$ amounts to evaluation the dual norm $\|\x\|_q$, where $1/q+1/p=1$ \cite{jaggi2013}. 
	
	\subsection{Simplex $\Delta_d$}
	Let $d>1$ and consider the setting where the losses $(f_t)$ are defined on the simplex $\Delta_d \coloneqq \{\x \in \reals_{\geq 0}^d :\mathbf{1}^\top \x =1 \}$. Since $\Delta_d$ has an empty interior, it does not satisfy Assumption \ref{ass:assum}. However, we can easily reparametrize the problem to ensure Assumption \ref{ass:assum}. 
	
	Let $\bm{c}\coloneqq (\frac{1}{2}+\frac{1}{2d}) \e_{d}+\sum_{i=1}^{d-1} \frac{1}{2d}\e_i$ and consider the sequence of reparametrized losses $(\ell_t)$ given by 
	\begin{align}
		\ell_t(\x) \coloneqq f_t\left(\bm{c}+J^\top \x  \right), \quad t \geq 1, \quad \text{where} \quad J \coloneqq \begin{bmatrix}
			I_{d-1}&  -\mathbf{1} \end{bmatrix}\in \reals^{d-1\times d}. \label{eq:simplex}
	\end{align}
	For any $t$, the function $\ell_t$ is convex and defined on the set \begin{align}\K\coloneqq \left\{\x\in  \reals^{d-1}: x_i \geq -\frac{1}{2 d} \text{ and } \mathbf{1}^\top \x  \leq \frac{1}{2}+\frac{1}{2d}\right\}.\label{eq:thesimplexC}\end{align}
We now show that for this reparametrized setting, Assumption \ref{ass:assum} holds with $r= 1/(2 d)$ and $R=1$, which implies a condition number of $\kappa = R/r \leq  2d$:
\begin{proposition}
		\label{prop:thesimplexC}
		The set $\K$ in \eqref{eq:thesimplexC} satisfies $\cB(1/(2d)) \subset \K \subset \cB(1)$.
	\end{proposition}
	\begin{proof}[{\bf Proof}]
Note that the vertices of the set $\K$ are $\v_1,\dots,\v_{d}$, where 
\begin{align}
\forall i\in[d-1],\quad  \v_{i} = \e_i -  \sum_{j\in [d-1]} \frac{\e_j}{2 d}  , \quad \text{and} \quad \v_{d} = -  \sum_{j\in [d-1]} \frac{\e_j}{2 d}.\nn 
\end{align}  
Since $\v_i \in  \cB(1)$ for all $i\in [d]$, we get that $\K\subseteq \cB(1)$. We now show that $\cB(1/(2d))\subseteq \K$. For this, we need to find the point $\bu$ in the boundary of $\K$ that is closest to the origin. This point must be the orthogonal projection of the origin onto one of the $(d-2)$-dimensional faces of $\K$; there are $d$-many such faces corresponding to one of the inequalities defining $\K$ being satisfied with equality. Thus, $\bu$ must satisfy one of the following:
\begin{itemize}
\item $\bu = \left(\frac{1}{d-1} - \frac{1}{2d} \right)\sum_{i\in[d-1]}\e_i$.
\item $\exists i\in[d-1]$, such that $\bu = -\frac{\e_j}{2d}$.
\end{itemize}
In all cases, we have $\|\bu\| \geq 1/(2d)$, and so this shows that $\cB(1/(2d))\subseteq \K$. 
\end{proof}
\noindent It is clear from the definition of the set $\K$ that the operations required to test the membership of a point $\x\in \reals^{d-1}$ are I) computing $\inner{\mathbf{1}}{\x}$; and II) evaluating $x_i$, for $i\in[d-1]$. Therefore, the corresponding computational complexity is $O(d)$. We now show how to build a subgradient Oracle for the reparametrized losses $(\ell_t)$. By the chain-rule, $\g$ is a subgradient of $\ell_t$ at $\x$ if and only if
	$\g = J \bm{\zeta},  \text{for } \bm{\zeta} \in \partial f_t(\bm{c}+J^\top \x ).$
	Here, $J \bm{\zeta}$ can be evaluated in $O(d)$.

	\subsection{Trace and Operator Norm Balls}
	\paragraph{Trace norm ball.} Let $m,n\geq 1$, $s\coloneqq m \wedge n$, and consider the setting where the losses are defined on the trace-norm ball $\K \coloneqq \{\x \in \reals^{m\times n} :  \sum_{i=1}^s \sigma_i(\x) \leq 1 \}$, where $\sigma_1(\x)\geq \dots \geq \sigma_s(\x)$ are the singular values of $\x$ in non-increasing order. Implementing a Membership Oracle $\cM_{\K}$ for $\K$ requires computing the sum of singular values of a matrix $\x$, and so the computational complexity of $\cM_{\K}$ is at most that of performing SVD. Since the trace-norm $\|\cdot\|_{\text{tr}}$ satisfies $$\frac{1}{\sqrt{s}}\|\x\|_{\text{tr}}  \leq  \sqrt{\sum_{i=1}^s \sigma_i(\x)^2}  \leq \|\x\|_{\text{tr}},$$ and $\sqrt{\sum_{i=1}^s \sigma_i(\x)^2}= \sqrt{\text{tr}(\x\x^\top)}=\|\x\|_{\text{F}}$ is just the Eucledian norm on $\reals^{m\times n}$ ($\|\cdot\|_{\text{F}}$ denotes the Frobenius norm), we have that Assumption \ref{ass:assum} is satisfied with $r=1$ and $R= \sqrt{m\wedge n}$, and so the condition number is $\kappa=\sqrt{m \wedge n}$.
	
	\paragraph{Operator norm Ball.} Let $m,n\geq 1$, $s\coloneqq m \wedge n$, and consider the setting where the losses are defined on the operator-norm ball $\K \coloneqq \{\x \in \reals^{m\times n} :  \sigma_1(\x) \leq 1 \}$, where $\sigma_1(\x)\geq \dots \geq \sigma_s(\x)$ are the singular values of $\x$. The operator norm $\|\cdot\|_{\text{op}}$ is the dual to the trace-norm $\|\cdot\|_{\text{tr}}$. Since $\|\x\|_{\text{op}}$ is the largest singular value of $\x$, we have $\|\x\|_{\text{op}}\leq \|\x\|_{\text{F}}\leq \sqrt{s}  \|\x\|_{\text{op}}$, and so Assumption \ref{ass:assum} is satisfied with $r=(m \wedge n)^{-1/2}$ and $R=1$. Implementing a Membership Oracle for $\K$ requires computing the largest singular value of a given matrix $\x$. It is possible to approximate the largest singular value up to error $\delta$ using $O(\text{nnz}(\x)/\sqrt{\delta})$ arithmetic operations, where $\text{nnz}(\x)$ represents the number of non-zeros of $\x$ \cite[Proposition 8]{jaggi2013}. That is, the complexity of implementing a $\delta$-approximate Membership Oracle $\cM_{\K}(\cdot;\delta)$ (see Definition \ref{def:mem}) is $O(\text{nnz}(\x)/\sqrt{\delta})$.

	\subsection{Convex-hull of Permutation Matrices} We now consider the setting where the losses $(f_t)$ are defined on the convex-hull of permutation matrices, also known as the Birkhoff polytope $\cK \coloneqq \{ \x \in \reals^{n \times n}_{\geq 0}:  \e_i^\top \x \mathbf{1} = \e_i^\top \x^\top \mathbf{1}=1 \}$. Assumption \ref{ass:assum} is not satisfies since the \emph{Birkhoff polytope} has an empty interior. However, as we did in the case of the simplex, we can easily reparametrize the problem to satisfy Assumption \ref{ass:assum}. We assume that $n> 1$ . Before presenting the reparametrized losses, we first introduce some notation. Let 
	\begin{align}
		\bm{c} \coloneqq \sum_{1\leq i\leq n}\left(\frac{1}{2} + \frac{1}{2n}\right) \e_{in}+ \sum_{1\leq j\leq n}\left(\frac{1}{2} + \frac{1}{2n}\right) \e_{nj}+\sum_{1\leq i ,j\leq n} \frac{\e_{ij}}{2n},\nn
		\end{align}
		and for any $\x \in \reals^{n-1 \times n-1}$ define the matrix $\bar \x\in \reals^{n \times n}$ such that 
		\begin{align}
	 \bar x_{ij} \coloneqq  \left\{ \begin{array}{ll} x_{ij}, &  \forall i,j \in[n-1]; \\ 0,& \text{otherwise}.\end{array} \right.	\label{eq:barx}
		\end{align}
With this, consider the sequence of reparametrized losses $(\ell_t)$ given by 
	\begin{gather}
		\ell_t(\x) \coloneqq f_t(\bm{c} + M \bar \x + \bar \x M^\top), \ \ \forall \x \in \reals^{n-1 \times n-1}, \nn \\ \text{where}\quad  M \coloneqq \begin{bmatrix} J^\top & \bm{0} \end{bmatrix}  \in \reals^{n \times n}  \ \ \text{and} \ \ J \text{ as in \eqref{eq:simplex} with $d=n$}. \label{eq:theM}
	\end{gather}
For any $t$, the function $\ell_t$ is convex and defined on the set 
	\begin{align}
		\K \coloneqq \left\{ \x \in \reals^{n-1 \times n-1}:\forall i,j\in[n-1], \ \  x_{ij}\geq - \frac{1}{2 n};\ \sum_{k=1}^{n-1} x_{ik}\leq \frac{1}{2}+\frac{1}{2n};\  \text{ and } \  \sum_{k=1}^{n-1} x_{kj} \leq \frac{1}{2} +\frac{1}{2n} \right\}.\nn
	\end{align}
We now show that for this reparametrized setting, Assumption \ref{ass:assum} holds with $r= 1/(2 n)$ and $R=\sqrt{5}/2$. Note that the vertices of the set $\K$ are $\v_1,\dots,\v_{(n-1)^2 +1}$, where 
\begin{align}
\forall i\in[(n-1)^2],\quad  \v_{i} =  \tilde\e_i -  \sum_{j\in [(n-1)^2] } \frac{\tilde\e_{j}}{2 n}, \quad \text{and} \quad \v_{(n-1)^2 +1} = - \sum_{j\in [(n-1)^2]} \frac{\tilde\e_j}{2 n},\nn 
\end{align}  
where $\tilde \e_i= \e_{pq}$ and $p,q$ are the unique integers satisfying $i = p + (n-1) (q-1)$ and  $p,q\in [n-1]$. Since $\v_i \in  \cB(\sqrt{5}/2)$, for all $i\in (n-1)^2 +1$, we get that $\K\subseteq \cB(\sqrt{5}/2)$. We now show that $\cB(1/(2n))\subseteq \K$. For this, we need to find the point $\bu$ in the boundary of $\K$ that is closest to the origin. This point must be the orthogonal projection of the origin onto one of the $((n-1)^2-1)$-dimensional faces of $\K$; there are $(n^2-1)$-many such faces corresponding to one of the inequalities defining $\K$ being satisfied with equality. Thus, $\bu$ must satisfy one of the following:
\begin{itemize}
	\item $\exists i \in[n-1]$, such that
	$\bu = \left(\frac{1}{n-1} - \frac{1}{2n} \right)\sum_{j\in[n-1]}  \e_{ij}$.
	\item 	$\exists j \in[n-1]$, such that $\bu = \left(\frac{1}{n-1} - \frac{1}{2n} \right)\sum_{i\in[n-1]}  \e_{ij}$. 
	\item $\exists i \in [(n-1)^2]$, such that $\bu = \frac{-\tilde \e_i}{2n}$.
\end{itemize}
In all cases, we have $\|\bu\| \geq 1/(2n)$, and so this shows that $\cB(1/(2n))\subseteq \K$. Thus, the condition number for this reparametrized setting is $$\kappa = R/r \leq  \sqrt{5}n.$$ 
	It is clear from the definition of the set $\K$ that the operations required to test membership for a point $\x\in \reals^{n \times n}$ are I) computing $\e_i^\top \x \mathbf{1}$ and $\e_i^\top \x^\top \mathbf{1}$ for $i,j\in[n]$; and II) evaluating $x_{ij}$, for $i,j\in[n]$. Thus the computational complexity of testing membership in $\K$ is $O(n^2)$. We now show how to build a subgradient Oracle for the reparametrized losses $(\ell_t)$. By the chain-rule, $\g$ is a subgradient of $\ell_t$ at $\x$, if and only if, for all $i,j \in [n-1]$, 
	\begin{align}
	g_{ij} =  2 \zeta_{ij} -\zeta_{nj} - \zeta_{in}, \quad \text{for}\quad  \bm{\zeta} \in \partial f_t(\bm{c}+M \bar\x +\bar \x M^\top), \nn 
	\end{align}
where $\bar \x$ and $M$ are as in \eqref{eq:barx} and \eqref{eq:theM}, respectively. Since $M$ has $2(n-1)$ non-zero entries, $\g$ can be computed in $O(n^2)$ time (this is linear in the dimension of $\K$).

	\subsection{Convex-hull of Rotation Matrices}
	We now consider the setting where the losses $(f_t)$ are defined on the convex-hull of rotation matrices; that is, 
	\begin{align}
		\K \coloneqq \operatorname{conv} \mathrm{SO}(n), \quad \text{where}\quad  	\mathrm{SO}(n) \coloneqq \{ \x \in \reals^{n \times n}:  \x^\top \x = I,\ \ \text{det}(\x)=1 \}.\nn
	\end{align}	
	This set satisfies Assumption \ref{ass:assum} with $r=1-2/n$ and $R=n$ (implying a condition number of at most $\kappa = n^2/(n-2) = O(n)$), and so there is not need to reparametrize as we show next:
	\begin{proposition}
		\label{prop:gros}
		Let $n> 2$. The convex hull $\K$ of Orthogonal matrices in $\reals^{n\times n}$ satisfies
		\begin{align}
			\cB(1/2) \subseteq \K \subseteq \cB(\sqrt{n}). \nn
		\end{align}
	\end{proposition}
	\begin{proof}[{\bf Proof}]
		Let $\mathrm{O}(n)$ the set of orthogonal matrices in $\reals^{n \times n}$. By \cite[Proposition 4.6]{saunderson2015}, we have 
		\begin{align}
			\operatorname{conv} \mathrm{SO}(n) = 	(\operatorname{conv} \mathrm{O}(n))  \cap ((n-2)   \mathrm{SO}^{-}(n)^\circ  ), \label{eq:setequality}
		\end{align}
		where $\mathrm{SO}^{-}(n) \coloneqq \{ \x \in \reals^{n \times n}:  \x^\top \x = I,\ \ \text{det}(\x)=-1\}$. It is know that $\operatorname{conv} \mathrm{O}(n)$ coincides with the operator-norm ball \cite{saunderson2015}, and so we have (see paragraph on the operator norm ball above) \begin{align}\cB(1)\subseteq \operatorname{conv} \mathrm{O}(n). \label{eq:orth}
		\end{align}
		We will now show that $\cB(1/\sqrt{n})\subseteq \mathrm{SO}^{-}(n)^\circ$ from which we conclude that $\cB(1/2)\subset \cB(1 \vee(n^{1/2}-2n^{-1/2})) \subset	\operatorname{conv} \mathrm{SO}(n)$ using \eqref{eq:setequality}, \eqref{eq:orth}, and the fact that $n>2$. Let $D$ be the diagonal matrix satisfying $D_{ii}=1$ for all $i\in[n-1]$ and $D_{nn}=-1$. It is known that $\mathrm{SO}^{-}(n) = D \cdot \mathrm{SO}(n)$ (see e.g.~\cite{saunderson2015}). Therefore, since $D\in \mathrm{O}(n)$, we have 
		\begin{align}
			\sup_{\x \in \mathrm{SO}^{-}(n)} \|\x\|_{\text{op}} =  \sup_{\x \in \mathrm{SO}(n)} \|\x\|_{\text{op}} = 1.\nn
		\end{align}
		Thus, by the fact that $\|\cdot\|_{\text{F}}\leq \sqrt{n}\|\cdot \|_{\text{op}}$, we have $\mathrm{SO}^{-}(n) \subseteq \cB(\sqrt{n})$, which implies that $\cB(1/\sqrt{n})\subseteq \mathrm{SO}^{-}(n)^\circ$. In fact, since $\mathrm{SO}^{-}(n) \subseteq \cB(\sqrt{n})$, we have $\inner{\bu}{\x}\leq 1$, for all $\bu \in \cB(1/\sqrt{n})$ and $\x\in \mathrm{SO}^{-}(n)$, and so $\cB(1/\sqrt{n})\subseteq \mathrm{SO}^{-}(n)^\circ$ by definition of a Polar set. Combining this with the fact that $\cB(1)\subseteq  \operatorname{conv} \mathrm{O}(n)$ and \eqref{eq:setequality}, we get that 
		\begin{align}
		\cB(1/2) \stackrel{n>2}{\subset}	\cB(1\vee (n^{1/2}-2n^{-1/2})) \subseteq  \operatorname{conv} \mathrm{SO}(n) =\K.\nn	
		\end{align}
		We now show that $\K\subseteq \cB(n)$. This follows by the fact that $\K = \text{conv} \ \mathrm{SO}(n) \subset \text{conv} \ \mathrm{O}(n)$ and that $\text{conv} \ \mathrm{O}(n) \subseteq \cB(\sqrt{n})$ since $\text{conv} \ \mathrm{O}(n)$ is the operator norm ball.
	\end{proof}
\noindent	To assess the complexity of a Membership Oracle for $\text{conv} \ \mathrm{SO}(n)$, we use the characterization of $\mathrm{SO}(n)$ in \eqref{eq:setequality}. Also, as argued in the proof of the previous proposition, $\operatorname{conv} \mathrm{O}(n)$ coincides with the operator-norm ball \cite{saunderson2015}. Thus, in light of \eqref{eq:setequality}, to test if $\x$ is in $ \operatorname{conv} \mathrm{SO}(n)$, it suffices to test if $\x$ is in the operator norm ball and in the set $\mathrm{SO}^{-}(n)^{\circ}$, simultaneously. The complexity of the former test is at most that of SVD \cite{saunderson2015}. We now show that the complexity of testing for $\x\in \mathrm{SO}^{-}(n)^{\circ}$ is also at most that of SVD up to a constant factor. First, we note that testing membership for $\mathrm{SO}^{-}(n)^\circ$ can be performed using a single call to a Linear Optimization Oracle on $\mathrm{SO}^{-}(n)$ (by leveraging the definition of a polar set). Furthermore, LO on $\mathrm{SO}^{-}(n)$ can be done using one call to a LOO on $\mathrm{SO}(n)$. The latter follows by the fact that $\mathrm{SO}^{-}(n) = D \cdot \mathrm{SO}(n)$ (see e.g.~\cite{saunderson2015}), where $D$ is the diagonal matrix defined in the proof of Proposition \ref{prop:gros}, and so 
	\begin{align}
		\cO_{\mathrm{SO}^{-}(n)}(\x) = \sup_{\y \in \mathrm{SO}^{-}(n)} \inner{\y}{\x} = \sup_{\y \in \mathrm{SO}(n)} \inner{\y}{D\x} = \cO_{\mathrm{SO}(n)}(D\x).\nn
	\end{align}
	Finally, since the complexity of linear optimization on $\mathrm{SO}^{-}(n)$ is at most the cost of SVD \cite{jaggi2013}, we conclude, in light of \eqref{eq:setequality}, that the complexity of a Membership Oracle for $\operatorname{conv}  \mathrm{SO}(n)$ is also at most that SVD up to a constant factor. 
	
	\subsection{PSD Matrices with Unit Trace}
	We now consider the set PSD matrices with unit trace. This set does not satisfy Assumption \ref{ass:assum} and so we need to reparametrize. It will be useful to introduce the operator ${U}\colon \reals^{n(n-1)/2}\rightarrow \reals^{n\times n }$, where for each $\z\in \reals^{n(n-1)/2}$, ${U}(\z)$ is the upper-triangular matrix whose $i$th column is equal to $(z_{i(i-1)/2+1}, \dots,z_{i(i+1)/2}, 0, \dots, 0)^\top \in \reals^n$. 
	Further, for any $\x\in \reals^n$, we let $\text{diag}(\x)$ be the matrix whose diagonal constructed from the vector $\x$, and define
	\begin{align}
		\Theta(\y,\z) \coloneqq  \text{diag}(J^\top \y) +U(\z)+ U(\z)^\top,\nn
		\end{align}
	for all $\y\in \reals^{n-1}$ and $\z\in \reals^{n(n-1)/2}$, where $J$ is as in \eqref{eq:simplex} with $d=n$. With this, we consider the set of reparametrized losses $(\ell_t)$ given by 
	\begin{align}
		\ell_t(\x) \coloneqq f_t (\bm{c} + \Theta(\y, \z)), \quad \text{where} \quad \bm{x}\coloneqq (\y,\z)\in \reals^{n-1} \times \reals^{n(n-1)/2} \quad \text{and}\quad \bm{c}\coloneqq \left(\frac{1}{2} + \frac{1}{2n}\right) \e_{nn}+\sum_{i=1}^{n-1} \frac{\e_{ii}}{2n}.\nn
	\end{align}
	For any $t$, the function $\ell_t$ is convex and defined on the set  
	\begin{align}
		\K \coloneqq \left\{ (\y,\z) \in \reals^{n-1} \times \reals^{n (n-1)/2}: \bm{c}+ \Theta(\y,\z)  \succeq 0 \right\}.  \label{eq:theC} 
	\end{align}
	Furthermore, this set satisfies Assumption \ref{ass:assum} with $r=n^{-3/2}/4$ and $R=2\sqrt{n}$, leading to a condition number of at most $\kappa = 8 n^2$ for the set (this is linear in the dimension of $\K$):
	\begin{proposition}
		\label{prop:theC}
		The set $\K$ in \eqref{eq:theC} satisfies $\cB(n^{-3/2}/4) \subseteq \K \subseteq \cB(2\sqrt{n})$.
	\end{proposition}
\noindent	To implement a Membership Oracle for $\K$ one needs to be able to test if a matrix of the form $\bm{c}+ \Theta(\y,\z)$ is positive definite. Since this matrix is symmetric, it suffices to check if the smallest eigenvalue of $\Theta(\y,\z)$ is non-negative. We now present a way of approximating the smallest eigenvalue of a symmetric matrix, which will then lead to an approximate Membership Oracle for $\K$. Given a symmetric matrix $M$ and $\delta>0$, first approximate its largest singular value $\sigma_1(M)$ up to error $\delta/2$. This can be done using $\wtilde {O}(\text{nnz}(M)/\sqrt{\delta})$ arithmetic operations (see \cite{kuczynski1992} and \cite[Proposition 8]{jaggi2013}). Next, approximate the largest singular value $\sigma_1(M')$ of $M' \coloneqq M -  \sigma_1(M)\cdot I_n$ up to error $\delta/2$. This also requires $\wtilde {O}(\text{nnz}(M)/\sqrt{\delta})$ arithmetic operations. Now, since the smallest eigenvalue of $M$ is given by $\lambda_{\min}(M) = \sigma_1(M)-\sigma_1(M')$, we can compute a $\delta$-approximate value of $\lambda_{\min}(M)$, and thus implement a $\delta$-approximate Membership Oracle, using $\wtilde{O}(\text{nnz}(M)/\sqrt{\delta})$ arithmetic operations.
	
	We now show how to build a subgradient Oracle for the reparametrized losses $(\ell_t)$. By the chain-rule, $\g\coloneqq(\g_{y},\g_{z})$ is a subgradient of $\ell_t$ at $\x\coloneqq (\y,\z)$ if and only if
	$$\g_y = J \text{diag}^{-1}(\bm{\zeta}) \quad \text{and} \quad  \g_z \coloneqq U^{-1}(\bm{\zeta}),  \quad \text{for } \bm{\zeta} \in \partial f_t(\bm{c}+\Theta(\y,\z)),$$
	where $J$ is as \eqref{eq:simplex} and $U^{-1}\colon \reals^{n \times n} \rightarrow \reals^{n(n-1)/2}$ [resp.~$\text{diag}^{-1}\colon\reals^{n \times n} \rightarrow \reals^n$] is any operator satisfying $U^{-1}\circ U(\z) =\z$, for all $\z$ [resp.~$\text{diag}^{-1}\circ \text{diag}(\x)=\x$, for all $\x$]. Thus, the subgradient Oracle for $\ell_t$ requires only an additional $O(n^2)$ operations (this is linear in the dimension of $\K$). We now prove Proposition \ref{prop:theC}:
	\begin{proof}[{\bf Proof of Proposition \ref{prop:theC}}] By the fact that $\|\text{diag}(J^\top \y)\|_{\text{F}}\leq 2 \sqrt{n} \|\y\|$ and $\|U(\z)+U(\z)^\top\|_{\text{F}} \leq 2 \|U(\z)\|_{\text{F}}$, for all $\y \in \reals^{n-1}$ and $\z \in \reals^{n(n-1)/2}$, we have
		\begin{align}
\|\Theta(\y,\z)  \|_{\text{F}}/(2\sqrt{n})  \leq \|(\y, \z)\| = \|\y\| + \|\z\| \leq  \|\Theta(\y,\z)  \|_{\text{F}}.\label{eq:sandwitch}
		\end{align}
We will now bound the norm of $\Theta(\y, \z)$. For any orthogonal matrix $H \in \cO(n)$, $\w\in \reals^{n\times n}$, and $\bm{\lambda} \in \reals^{n-1}$ , define 	
\begin{gather}
\Phi_{H}(\w) \coloneqq - \e_{nn}/(2n) + H^\top \w H;\quad  \quad \Psi(\bm{\lambda}) \coloneqq \text{diag}(\e_{n}/2+ J^\top \bm{\lambda}), \nn \\
\text{and} \quad \K' \coloneqq \left\{ \bm{\lambda} \in \reals^{n-1}: \ \lambda_i \geq \frac{-1}{2n} \ \ \text{and} \ \ \mathbf{1}^\top \bm{\lambda} \leq \frac{1}{2} + \frac{1}{2n} \right\},  \nn 
\end{gather} 
where $J$ is as in \eqref{eq:simplex} with $d=n$. Since a similarity transformation does not change the trace, or the eigenvalues for that matter, we have $\bm{c}+\Phi_H\circ \Psi(\bm{\lambda})\in  \K,$ $\forall H \in \cO(n), \forall \bm{\lambda}\in \K'$. In particular, this implies that
\begin{align}
\bigcup_{H \in \cO(n)} \Phi_H\circ \Psi(\K') \subseteq \Theta(\K).
 \label{eq:inclusion}
\end{align}
We will now show that for any $H \in \cO(n)$ and $\bm{\lambda}$ on the boundary of $\K'$, we have $\|\Phi_H\circ \Psi(\bm{\lambda})- \Phi_H\circ \Psi(\bm{0})\|_{\text{F}} \geq 1/(2n)$. This, combined with \eqref{eq:inclusion},  would imply that $\cB(1/(2 n)) \subseteq \Theta(\K)$. Let $H \in \cO(n)$ and $\bm{\lambda}\in \text{bd}\ \K'$. Since $\Phi_H$ is an isometry with respect to the distance induced by the operator norm, we have 
\begin{align}
\|\Phi_H\circ \Psi(\bm{\lambda})- \Phi_H\circ \Psi(\bm{0})\|_{\text{F}}\geq \|\Phi_H\circ \Psi(\bm{\lambda})- \Phi_H\circ \Psi(\bm{0})\|_{\text{op}} = \| \Psi(\bm{\lambda}) -  \Psi(\bm{0})  \|_{\text{op}} = \|J^\top \bm{\lambda}\| \geq \|\bm{\lambda}\| \geq 1/(2n),\nn
\end{align} 
where the last inequality follows by Proposition \ref{prop:thesimplexC}. Thus, we have that $\cB(1/(2 n)) \subseteq \Theta(\K)$. Combining this with \eqref{eq:sandwitch} implies that $\cB(n^{-3/2}/4) \subseteq \K$.

 We will now show that $\K \subset \cB(2\sqrt{n})$. For this, we will first show that \eqref{eq:inclusion} holds with equality. Let $\Theta'(\cdot, \cdot)\coloneqq \e_{nn}/2 + \Theta(\cdot, \cdot)$, and observe that the constraint defining the set $\K$ in \eqref{eq:theC} translates to $I_n/(2n) + \Theta'(\cdot, \cdot)\succeq 0$. Let $(\y,\z)\in \K$. Since $\Theta'(\y, \z)$ is a real symmetric matrix, there exists an orthogonal matrix $H$ such that $\Lambda \coloneqq H \Theta'(\y,\z)H^\top$ is a diagonal matrix. Let $\lambda_1,\dots, \lambda_n$ be the diagonal elements of $\Lambda$. Since these are the eigenvalues of $\Theta'(\y,\z)$, the fact that $I/(2n) + \Theta'(\y,\z)  \succeq 0$ implies $\lambda_i \geq -1/(2n)$, for $i\in[n]$. Furthermore, since $\text{tr}(\Theta'(\y,\z))=1/2$, we have $\sum_{i=1}^{n} \lambda_i = 1/2$, and so since $\lambda_n \geq - 1/(2n)$, it follows that $\sum_{i=1}^{n-1} \lambda_i \leq 1/2 + 1/(2n)$. Thus, we have $\bm{\lambda}' \coloneqq (\lambda_1,\dots,\lambda_{n-1}) \in \K'$ and  
		\begin{align}
		\e_{nn}/2+\Theta(\y,\z) = \Theta'(\y,\z) = H^{\top}\text{diag}(e_n/2+ J^{\top}\bm{\lambda}') H = \e_{nn}/2+\Phi_H\circ \Psi(\bm{\lambda}').\nn
		\end{align} 
		Therefore, we have $\Theta(\K) \subseteq \bigcup_{H\in \cO(n)} \Phi_H\circ \Psi(\K')$, and so by \eqref{eq:inclusion}, we have that $\Theta(\K)= \bigcup_{H\in \cO(n)} \Phi_H\circ \Psi(\K')$. Now, for any $H \in \cO(n)$ and $\bm{\lambda}\in \K'$, we have
		\begin{align}
\|\Phi_H\circ \Psi(\bm{\lambda})- \Phi_H\circ \Psi(\bm{0})\|_{\text{F}} &\leq \sqrt{n}\|\Phi_H\circ \Psi(\bm{\lambda})- \Phi_H\circ \Psi(\bm{0})\|_{\text{op}}, \nn \\ &  = \sqrt{n}\| \Psi(\bm{\lambda}) -  \Psi(\bm{0})  \|_{\text{op}}, \nn \\ &  =\sqrt{n} \|J^\top \bm{\lambda}\| \leq \sqrt{n}\|\bm{\lambda}\| +\sqrt{n}|\inner{\mathbf{1}}{\bm{\lambda}}|  \leq 2\sqrt{n},\nn
		\end{align}
		where the last inequality follows by Proposition \ref{prop:thesimplexC}, and the fact that $\inner{\mathbf{1}}{\bm{\lambda}}\leq 1/2 + 1/(2n)\leq 1$ by definition of $\K'$. Therefore, by \eqref{eq:sandwitch}, we have $\K\subseteq \cB(2\sqrt{n})$. 
\end{proof}
\label{sec:PSDtrace}
	
	\subsection{PSD Matrices with Bounded Diagonals}
	We now consider the set of PSD matrices with bounded diagonals; that is, $\cK\coloneqq \{\x \in \reals^{n\times n}:  \x \succeq 0, \text{and}\  0\leq  x_{ii} \leq 1 , \text{for all } i \in[n]\}$. This set does not satisfy Assumption \ref{ass:assum} and so we need to reparametrize. As in the case of PSD matrices with unit trace, we let ${U}\colon \reals^{n(n-1)/2}\rightarrow \reals^{n\times n }$ be the operator such that for each $\z\in \reals^{n(n-1)/2}$, ${U}(\z)$ is the upper-triangular matrix whose $i$th column is equal to $(z_{i(i-1)/2+1}, \dots,z_{i(i+1)/2}, 0, \dots, 0)^\top \in \reals^n$. 
	Also, for any $\y\in \reals$, we let $\text{diag}(\y)$ be the matrix whose diagonal constructed from the vector $\y$, and define $$\Xi(\y,\z) \coloneqq  \text{diag}(\y) +U(\z)+ U(\z)^\top.$$ With this, we consider the set of reparametrized losses $(\ell_t)$ given by 
	\begin{align}
		\ell_t(\x) \coloneqq f_t (\bm{c} + \Xi(\y, \z)), \quad \text{where} \quad \bm{x}\coloneqq (\y,\z) \quad \text{and}\quad  \bm{c}\coloneqq I_n/2.\nn
	\end{align}
	For any $t$, the function $\ell_t$ is convex and defined on the set  
	\begin{align}
		\K \coloneqq \left\{ (\y,\z) \in \reals^{n} \times \reals^{n (n-1)/2}: \bm{c}+ \Xi(\y,\z)  \succeq 0, \quad y_{ii} \leq 1/2, \forall i \in[n]  \right\},   \label{eq:theCC} 
	\end{align}
	Furthermore, this set satisfies Assumption \ref{ass:assum} with $r=1/4$ and $R=n^{3/2}/2$, and so the condition number is $\kappa =2 n^{3/2}$:
	\begin{proposition}
		\label{prop:theCC}
		The set $\K$ in \eqref{eq:theCC} satisfies $\cB(1/4) \subseteq \K \subseteq \cB(n^{3/2})$.
	\end{proposition}
\noindent	To implement a Membership Oracle for $\K$ one needs to be able to test if a matrix of the form $\bm{c}+ \Xi(\y,\z)$ is positive definite. Since this matrix is symmetric, it suffices to check that the smallest eigenvalue of $\Xi(\y,\z)$ is non-negative. This can be done in the same way as in Subsection \ref{sec:PSDtrace} (PSD matrices with unit trace), and so a $\delta$-approximate Membership Oracle for $\K$ can be implemented using $\wtilde{O}({\text{nnz}(\x)/\sqrt{\delta}})$ arithmetic operations for any input $\x\in \reals^{n}\times \reals^{n(n-1)/2}$ and tolerance $\delta>0$.

	We now show how to build a subgradient Oracle for the reparametrized losses $(\ell_t)$. By the chain-rule, $\g\coloneqq(\g_{y},\g_{z})$ is a subgradient of $\ell_t$ at $\x\coloneqq (\y,\z)$ if and only if
	$$\g_y = \text{diag}^{-1}(\bm{\zeta}) \quad \text{and} \quad  \g_z \coloneqq U^{-1}(\bm{\zeta}),  \quad \text{for } \bm{\zeta} \in \partial f_t(\bm{c}+\Xi(\y,\z)),$$
	where $U^{-1}\colon \reals^{n \times n} \rightarrow \reals^{n(n-1)/2}$ [resp.~$\text{diag}^{-1}\colon\reals^{n \times n} \rightarrow \reals^n$] is an operator satisfying $U^{-1}\circ U(\z) =\z$, for $\z$ [resp.~$\text{diag}^{-1}\circ \text{diag}(\x)=\x$, for all $\x$]. Thus, the subgradient Oracle for $\ell_t$ only requires an additional $O(n^2)$ operations (this is linear in the dimension of $\K$).
	
	\begin{proof}[{\bf Proof of Proposition \ref{prop:theCC}}]  By the fact that $\|\text{diag}(\y)\|_{\text{F}}=\|\y\|$ and $\|U(\z)+U(\z)^\top\|_{\text{F}} \leq 2 \|U(\z)\|_{\text{F}}$, for all $\y \in \reals^{n}$ and $\z \in \reals^{n(n-1)/2}$, we have 
		\begin{align}
			\|\Xi(\y,\z)  \|_{\text{F}}/2     \leq 	\|(\y, \z)\|\coloneqq  \|\y\| + \|\z\| \leq  \|\Xi(\y,\z)  \|_{\text{F}}.\label{eq:sandwitch2}
		\end{align}
		We will now bound the norm of $\Xi(\y, \z)$. Let $(\y, \z)\in \K$. Since $\Xi(\y, \z)$ is a real symmetric matrix, there exists an orthogonal matrix $H$ such that $\Lambda \coloneqq H \Xi(\y,\z)H^\top$ is a diagonal matrix. Let $\lambda_1,\dots, \lambda_n$ be the diagonal elements of $\Lambda$. Since these are also the eigenvalues of $\Xi(\y,\z)$, the fact that $I_n/2+ \Xi(\y,\z)  \succeq 0$ implies that $\lambda_i \geq -1/2$, for all $i\in[n]$. Furthermore, since $y_{ii}\leq 1/2$ for all $i\in[n]$, we have $\text{tr}(\Xi(\y,\z))\leq n/2$ and so $\sum_{i=1}^{n} \lambda_i \leq n/2$. That is, 
		\begin{align}
			\bm{\lambda} \in \K' \coloneqq \left\{ \x\in \reals^{n}: \ x_i \geq \frac{-1}{2} \ \ \text{and} \ \ \mathbf{1}^\top \x \leq \frac{n}{2} \right\}. \nn
		\end{align} 
		The argument above implies that \begin{gather} \Xi(\K)\subseteq \bigcup_{H\in \mathrm{O}(n)} \Phi_H(\K'), \quad 
			\text{where}  \quad  	\Phi_H(\bm{\lambda}') =  H^\top \text{diag}\left(\bm{\lambda}\right) H. \nn
		\end{gather} 
		Using that $\|\bm{\lambda}\|\leq n$, for all $\bm{\lambda} \in \K'$, and the fact that multiplication by an orthogonal matrix does not change the operator norm, we have 
		\begin{align}
			n\geq  \sup_{\bm{\lambda} \in \K'}\|\bm{\lambda}\|=   \sup_{\bm{\lambda} \in \K', H \in \mathrm{O}(n)}\| \Phi_{H}(\bm{\lambda})\|_{\text{op}} \geq  \sup_{\x \in \K }\|\Xi(\x)\|_{\text{op}} \geq  \sup_{\x \in \K }\|\Xi(\x)\|_{\text{F}}/\sqrt{n}	 \geq   \sup_{\x \in \K }\|\x\|/\sqrt{n},\nn
		\end{align}
		where the last inequality follows by \eqref{eq:sandwitch2}. This implies that $$\K\subseteq\cB(n^{3/2}).$$ 
		We now show that $\cB(1/4) \subseteq  \K$. For this, we need to evaluate the quantity $\inf_{\x \in \text{bd} \K} \|\x\|$. Let $\x \in \text{bd} \K$. The fact that $\x$ is on the boundary of $\K$ implies that at least one of the inequality constraints in the definition of $\K$ must be satisfied with equality; that is, one of the following must be true:
	\begin{enumerate}[label=(\alph*)]
			\item $\exists i \in[n]$, such that $x_{ii} =1/2$.
			\item $\exists \bu \in \reals^n$, such that $\|\bu\|=1$ and $\Xi(\y,\z)\bu = -\bm{c}\bu$.
		\end{enumerate}
		If (a) is true, then $\|\x\|\geq 1/2$. In case (b) holds, then by \eqref{eq:sandwitch2}, we have 
			\begin{align}
		\|\x\| \geq \|\Xi(\y,\z)\|_{\text{F}}/2 \geq  \|\Xi(\y,\z)\|_{\text{op}}/2\geq \|\Xi(\y,\z)\bu\|/2 = \|\bm{c}\bu\|/2 = 1/4.\nn
		\end{align}	
Since $\x$ was chosen arbitrarily on the boundary of $\K$, we have that $\cB(1/4) \subseteq \K$, which completes the proof.
	\end{proof}
\subsection{The Flow and Matroid Polytopes}
	\paragraph{Flow Polytope.}
For this polytopes, we do not present an explicit parametrization since it is highly dependent on the specific problem at hand. We only study the complexity of the Membership Oracles for this case. The flow polytope represents the convex hull of indicator vectors corresponding to paths in a directed acyclic graph with $d$ nodes and $m$ edges. This polytope can be described with $O(m+d)$ linear inequalities \cite{hazan2012,meszaros2019}. Therefore, a Membership Oracle for this polytope can be implemented using $O(d+m)$ arithmetic operations. Linear optimization on the flow polytope has the same complexity up to log-factor \cite{schrijver2003}.

\paragraph{Matroid Polytope.}
Same as in the case of the flow polytope, we only comment on the computational complexity of a Membership Oracle. A Matroid Polytope is the convex hull of indicator vectors corresponding to the independent sets $A\in I$ of a matroid $M= (E,I)$. The polytope can be described using $O(2^d)$ linear inequalities where $d = |E|$. Thus, the naive implementation of the Membership Oracle that checks all these linear inequalities would be intractable. We will present an alternative approach to designing a $\delta$-approximate Membership Oracle for $M$ that requires only $O(d^3 + d^2 \ln(d) \Cost(\cI_{M}))\ln (1/\delta)$ arithmetic operations, where $\Cost(\cI_{M})$ is the computational cost (number of arithmetic operations) of testing if a subset of $E$ is independent (i.e.~an element of $I$). 

Let $\K$ denote the matroid polytope corresponding to our matroid $M=(E,I)$. As we have shown in Section \ref{sec:gauge}, the complexity of a Membership Oracle on $\K$ is the same as linear optimization on $\K^\circ$. Furthermore, a $\delta$-approximate Linear Optimization Oracle for $\K^\circ$ can be implemented using $O(d^3 + d \Cost(\cS_{\K^\circ})) \ln (1/\delta)$ arithmetic operations, where $\Cost(\cS_{\K^\circ})$ is the computational cost (number of arithmetic operations) of a Separation Oracle on $\K^\circ$ \cite{lee2018efficient}. The fact that $(\K^{\circ})^\circ=\K$ for a closed convex set and our results from Section \ref{sec:gauge} (see also \cite{grotschel1993,molinaro2020}) imply that the complexity of $\cS_{\K^\circ}$ is the same as linear optimization on $\K$. The latter can be performed using $O(d \ln (d) \Cost(\cI_M))$ arithmetic operations \cite[Section 40.1]{schrijver2003}. All in all, a Membership Oracle for $\K$ can be implemented using $O(d^3 + d^2 \ln(d) \Cost(\cI_{M}))$ arithmetic operations (omitting log-factors in $1/\delta$). 
\section{Discussion}
In this paper, we presented a novel projection-free reduction that allowed us to turn any OCO algorithm defined on a Euclidean ball $\cB$, to an algorithm on any convex set $\K$ contained in $\cB$, without sacrificing the performance of the original algorithm by much. Thanks to this, we were able to build explicit algorithms that achieve optimal regret bounds in OCO without incurring expensive projections. In particular, we swapped expensive Euclidean projections for, what we call, Gauge projections, and show that the latter can be performed efficiently using a Membership Oracle; our final algorithms make at most $O(T \ln T)$ calls to such an Oracle after $T$ rounds. We also extended our results to the stochastic and offline settings where we recovered optimal rates in terms of the number of iterations. 
\paragraph{Advantages over Frank-Wolfe variants.} One advantage of our reduction is that it allows one to achieve the optimal regret bound in OCO (in the number of iterations), unlike existing Frank-Wolfe variants (see Table \ref{tab:result}). Another advantage of our approach is that it does not require linear optimization over $\K$, which can be expensive in some settings; our method only requires a Membership Oracle. Furthermore, since our reduction works for any base algorithm, one can achieve different types of guarantees (e.g.~a dynamic/ anytime regret or one that scales with the norm of the comparator \cite{mhammedi2020,cutkosky2020parameter}) by substituting $\A$ in Alg.~\ref{alg:projectionfreewrapper} with an algorithm that is known to achieve the desired guarantee.  
\paragraph{Advantages of Frank-Wolfe variants.} The main downside of our approach is that it is parametrization dependent, and so the user needs to do the extra work of choosing a parametrization. We have done this for popular settings in Section \ref{sec:applications}. Furthermore, as we pointed out in the introduction, there are setting where linear optimization on a set can be done more efficiently than evaluating membership. This is the case for sets that have a combinatorial structure such as Matroids. In this case, Frank-Wolfe-style algorithms may be more practical. Finally, Frank-Wolfe iterates can often be expressed explicitly as a convex combination of a finite number of points on the boundary of the set of interest, which can be desirable in some applications \cite{garber2016faster}.  

\paragraph{Future directions.} As reflected by our derivations in Section \ref{sec:applications}, the parametrization we choose for the problem can affect the constant $\kappa$ (i.e.~the condition number) in the bound, and so it is important to choose a parametrization that minimizes $\kappa$ as much as possible. The origin of $\kappa$ in our bound is due to an upper bound on the norms of the subgradients of the surrogate losses (see Proposition \ref{prop:subgradient}). We explained in Section \ref{sec:general} that these upper bounds, which lead to the linear dependence in $\kappa$ in our regret bounds, may be too conservative in practice. A future direction would be to explore if such bounds can be improved. We also note that our reduction can easily be extended to the setting where the losses are exp-concave instead of strongly convex using existing results due to \cite{mhammedi2019}. One can also extend our reduction to the Bandit setting \cite{chen2019projection,garber2020improved}. Finally, it remains to test our method on real-world datasets. 

\clearpage
\appendix

\section*{Appendices} 
\label{sec:appendix}
\addcontentsline{toc}{section}{Appendices}

	\section{Technical Lemmas}
	\label{sec:teclemmas}
	This section contains some technical lemmas we need to prove our results.
	\begin{lemma}
		\label{lem:doob}
		Let $(Y_t)\subset \reals_{\geq 0}$ be a sequence of random variable satisfying $\E[Y_t \mid \cG_{t-1}]\leq \delta_t$, for all $t\geq 1$, for some sequence $(\delta_t)\subset \reals_{\geq 0}$. Then, for any $\rho\in(0,1)$ and $T\geq 1$, we have with probability at least $1-\rho$, 
		\begin{align} 
			\P\left[ \sum_{t=1}^T Y_t \geq (1+1/\rho) \sum_{t=1}^T \delta_t  \right]  \leq \rho.\nn
		\end{align}
	\end{lemma}
	\begin{proof}[{\bf Proof}]
		Let $X_t \coloneqq \sum_{i=1}^t (Y_i - \bar\delta_i )$, where $\bar \delta_i \coloneqq \E[Y_i \mid \cG_{i-1}]\leq \delta_i$. The process $(X_t)$ is a martingale; that is, for all $i\geq 1$, we have, for all $i<t$,  
		\begin{align}
			\E[ X_t \mid \cG_{i}]  = \sum_{s=1}^i ( Y_s- \bar \delta_s) = X_i.\nn
		\end{align}
		Thus, by Doob's martingale inequality \cite[Theorem 4.4.2]{durrett2019}, we have, for any $\rho\in(0,1)$, and $T\geq 1$
		\begin{align}
			\P\left[ \sum_{t=1}^T Y_t \geq (1+1/\rho) \sum_{t=1}^T \delta_t  \right] \leq \P\left[ X_T\geq   \sum_{t=1}^T \delta_t/\rho  \right]   \leq \P\left[\max_{t\leq T} X_t \geq  \sum_{t=1}^T \delta_t/\rho  \right] \leq   \frac{\rho \E\left[X_T\vee 0  \right]}{ \sum_{t=1}^T \delta_t} \leq \rho.  \nn
		\end{align}
	\end{proof}
\begin{theorem}
	\label{thm:freedman}
	Let $\cF_1, \dots ,\cF_n$ be a filtration, and $X_1, \dots, X_n$ be real random variables such that $X_i$ is $\cF_i$-measurable, $\E[X_i \mid \cF_{i-1}]=0$, $|X_i|\leq b$, and $\sum_{i=1}^n \E[X_i^2 \mid \cF_{i-1}]\leq V$ for some $b,V\geq 0$. Then, for any $\delta \in(0,1)$, with probability at least $1-\delta$, 
	\begin{align}
		\sum_{i=1}^n X_i \leq 2 \sqrt{V_n \ln (1/\delta)} +b \ln (1/\delta).	\nn
	\end{align} 
\end{theorem}

\begin{lemma}
\label{lem:concave}
For any $a\geq1$, $b>0$, the map $f:x \mapsto \sqrt{x \cdot\ln (a+b x)}$ is concave for $x>0$.
\end{lemma}
\begin{proof}[{\bf Proof}]
Let $a\geq 1$ and $b>0$. We will show that the second derivative of $f$ is negative for all $x>0$. We have 
\begin{align}
\forall x >0, \quad 	f''(x) & =\frac{-(a+b x)^2 \log ^2(a+b x)+2 a b x \log (a+b x)-b^2 x^2}{4 (a+b x)^2 (x \log (a+b
		x))^{3/2}},\nn \\
	& \leq \frac{-a^2 \log ^2(a+b x)+2 a b x \log (a+b x)-b^2 x^2}{4 (a+b x)^2 (x \log (a+b
		x))^{3/2}},\nn \\
	& = - \frac{- (a\log (a+b x)-b x)^2}{4 (a+b x)^2 (x \log (a+b
		x))^{3/2}} \leq 0.\nn
	\end{align}
\end{proof}
\noindent We now restate and prove Lemma \ref{lem:properties}:
	\begin{manuallemma}{1'}
	\label{lem:properties2}
	Let $\w\in \reals^d\setminus \{\bm{0}\}$ and $0< r\leq R$. Further, let $\K$ be a closed convex set such that $\cB(r) \subseteq \K \subseteq \cB(R)$. Then, the following properties hold:
	\begin{enumerate}[label=(\alph*)]
		\item $\sigma_{\K^\circ}(\w)= \gamma_{\K}(\w)$ and $(\K^\circ)^\circ =\K$.
		\item $\sigma_{\K}(\alpha \w) = \alpha \sigma_{\K}(\w)$ and $\partial \sigma_{\K}(\alpha \w)= \partial \sigma_{\K}(\w) = \argmax_{\bu \in \K} \inner{\bu}{\w}$, for all $\alpha\geq 0$.  
		\item $r \|\w\|\leq \sigma_{\K}(\w) \leq R\|\w\|$, $\|\w\|/R\leq \gamma_{\K}(\w) \leq \|\w\|/r$, and $\cB(1/R)\subseteq \K^\circ \subseteq \cB(1/r)$.
		\item $\inner{\w}{\bu}\leq \sigma_{\K}(\w) \cdot \gamma_{\K}(\bu)$, for all $\bu\in \reals^d$.  (Cauchy Schwarz)
		\item $\sigma_{\K}(\w+\bu)\leq\sigma_{\K}(\w)+\sigma_{\K}(\bu)$, for all $\bu\in \reals^d$.  (Sub-additivity)
	\end{enumerate}
\end{manuallemma}
\begin{proof}[{\bf Proof}] Points $(a)$, $(b)$, and $(e)$ follow from standard results in convex analysis, see e.g.~\cite[Lemma 2]{molinaro2020} for point (a) and \cite{hiriart2004} for points $(b)$ and $(e)$. Point \emph{(d)} follows from \cite[Equation 2.3 \& Proposition 2.3]{friedlander2014gauge} and Point \emph{(a)}. We now show point (c). The first inequality in Point $(c)$ follows by the fact that 
	\begin{align}
		R \|\w\|	\stackrel{(i)}{\geq}  \sup_{\bu \in \K}\inner{\bu}{\w} = 	  \sigma_{\K}(\w) \geq \inf_{\bu \in \K} \inner{\bu}{\w} \stackrel{(ii)}{\geq} r\|\w\|, \nn
	\end{align}
	where $(i)$ and $(ii)$ follow by the assumption that $\cB(r)\subseteq \K \subseteq \cB(R)$. We now show that $\cB(1/R)\subseteq \K^\circ \subseteq \cB(1/r)$. For any $\x\in \cB(1/R)$, we have $\inner{\x}{\w}\leq 1$ for all $\w \in \K$, since $\K \subseteq \cB(R)$. By definition of the polar set, this implies that $\cB(1/R)\subseteq \K^\circ$. Now, let $\x\in \K^\circ$. This implies that $\inner{\x}{\w}\leq 1$ for all $\w\in \K$. For $\w= r \x/\|\x\|$ (which is guaranteed to be in $\K$ since $\cB(r)\subseteq \K$) this inequality implies that $\|\x\|\leq 1/r$, and so $\K^\circ \subseteq \cB(1/r)$. Finally, $\|\w\|/R\leq \gamma_{\K}(\w) \leq \|\w\|/r$ follows by point $(a)$ and the facts that $\cB(1/R)\subseteq \K^\circ \subseteq \cB(1/r)$ and that $(\cB(r) \subseteq \K \subseteq \cB(R) \implies  r \|\w\|\leq \sigma_{\K}(\w) \leq R\|\w\|)$, which we just showed.
\end{proof}
	
	\section{Adaptive OCO Algorithms}
	\label{app:OCO}
	We now present two algorithms for online convex optimization that we will build on to derive our results.
	\begin{algorithm}[H]
		\caption{FTRL-proximal on $\cB(R)$ \cite[Algorithm 2 \& Section 3.3]{mcmahan2017survey}}
		\label{alg:FTRL-proximal}
		\begin{algorithmic}[1] 
			\Require Radius $R$.
			\vspace{0.2cm} 
			\State Initialize ${\w}_1=\bm{0}$, $\G_0=\bm{0}$, and $V_0=0$.
			\State Initialize $\eta_0=\infty$ and $\Sigma_{0}=0$, 
			\For{$t=1,2,\dots$}
			\State Play ${\w}_t$ and observe ${\bnabla}_t\in \partial \ell_t({\w}_t)$.
			\State Set $\eta_t = \sqrt{2} R/\sqrt{V_t}$, $\sigma_t = 1/\eta_t -1/\eta_{t-1},$ and $\Sigma_t = \Sigma_{t-1}+\sigma_t$.  \algcomment{With the convention $1/\infty\coloneqq 0$}
			\State Set $\G_t=\G_{t-1}+{\bnabla}_t$ and $V_t = V_{t-1}+\|\bnabla_t\|^2$.
			\State Set $\w_{t+1}= \Pi_{\cB(R)}(\wtilde \w_{t+1})$, where $\wtilde \w_{t+1} \coloneqq (-\G_{t} + \sum_{s=1}^t \sigma_s {\w}_s)/ \Sigma_t$. 
			\Statex \algcomment{The vector ${\w}_{t+1}$ above satisfies ${\w}_{t+1}\in \argmin_{\w\in \cB(R)} \inner{\G_t}{\w} + \sum_{s=1}^t \sigma_s^2 \|\w-{\w}_s\|^2/2.$}
			\EndFor
		\end{algorithmic}
	\end{algorithm}
\noindent 
The algorithm requires $R$ as input, but does not require an upper bound on the norm of the input loss vectors $(\bnabla_t)$ (i.e.~the algorithm adapts to the norm of the loss vectors). We now state the guarantee of \ftrl{} which follows from \cite[Theorem 2 \& Section 3.3]{mcmahan2017survey} with the choice of learning rate $\eta_t \coloneqq   \sqrt{2} R/\sqrt{V_t}$, where $V_t \coloneqq  \sum_{s=1}^t \|{\bnabla}_s\|^2$:
	\begin{proposition}[FTRL-proximal's regret]
		\label{prop:ftrl-proximal}
		For any adversarial sequence of convex losses $(\ell_t)$ on $\cB(R)$, the iterates $(\w_t)$ of \ftrl{} with parameter $R>0$ in response to $(\ell_t)$, satisfy for all $T\geq 1$ and $\w\in \cB(R)$,
		\begin{align}
			\sum_{t=1}^T (\ell_t(\w_t)- \ell_t(\w))	\leq 	\sum_{t=1}^T \inner{\bnabla_t}{\w_t - \w} \leq 2 R\sqrt{2V_T},\nn
		\end{align}  
		where $\bnabla_t$, $t\geq 1$, is any subgradient in $\partial \ell_t(\w_t)$ and $V_T \coloneqq\sum_{t=1}^T \|{\bnabla}_t\|^2$.
	\end{proposition}
	\begin{algorithm}[H]
		\caption{\freegrad{} \cite{mhammedi2020} with the unconstrained-to-constrained reduction due to \cite{cutkosky2020parameter}.}
		\label{alg:freegrad}
		\begin{algorithmic}[1] 
			\Require Parameters $\epsilon>0$ and $R>0$.
			\vspace{0.2cm} 
			\State Initialize $\w_1=\s_1=\bm{0}$, $B_0 = \epsilon$, $\G_0=\bm{0}$, and $Q_0=\epsilon^2$. \label{line:one}
			\For{$t=1,2,\dots$}
			\State Play $\w_t$ and observe $\bnabla_t\in \partial \ell_t(\w_t)$. 
			\State Set $B_t = B_{t-1} \vee\|\bnabla_t\|$ and $\bar\bnabla_t \coloneqq \bnabla_t \cdot  B_{t-1}/B_t$.
			\State Set $\tilde \bnabla_t = \bar \bnabla_t - \mathbb{I}_{\inner{\bnabla_t}{\w_t}<0} \cdot  \inner{\bar \bnabla_t}{\s_t} \s_t$. \algcomment{$\s_t$ is defined on Lines \ref{line:one} and \ref{line:eight}}
			\State Set $\G_t=\G_{t-1}+\tilde\bnabla_t$ and $Q_t = Q_{t-1}+\|\tilde\bnabla_t\|^2$. 
			\State  \label{eq:output} Set $\displaystyle  \x_{t+1} \coloneqq -\G_{t} \cdot    \frac{ (2Q_{t} +B_{t} \|\G_{t}\|)\cdot \epsilon^2 }{2(Q_{t}+ B_t\|\G_{t} \|)^2 \ \sqrt{ Q_{t}}} \cdot \exp\left(\frac{\|\G_{t}\|^2}{2 Q_{t} + 2 B_t \|\G_{t}\|} \right)$.
			\State Set $\w_{t+1}=\Pi_{\cB(R)}(\x_{t+1})$ and $\s_{t+1} =  \x_{t+1}/\|\x_{t+1}\| \cdot  \mathbb{I}_{\|\x_{t+1}\|>R}$. \algcomment{We use the convention that $0/0=0$.} \label{line:eight}
			\EndFor
		\end{algorithmic}
	\end{algorithm}
\noindent Another algorithm that will be instrumental to developing our projection-free algorithms is \freegrad{} (see Algorithm \ref{alg:freegrad}).
An important property of \freegrad{} that will be useful to us when developing a projection-free algorithm for strongly convex losses (see Section \ref{sec:strong}) is that its regret scales directly with the norm $\|\w\|$ of the comparator, as opposed to the worst-case $R$. 
We note that \freegrad{} internally clips the sequence of observed sub-gradients. In our application of \freegrad{}, it will be useful to state the guarantee of its iterates $(\w_t)$ in response to the sequence of clipped subgradients $(\bar \bnabla_t)$: 
	\begin{proposition}[FreeGrad's clipped linearized regret]
		\label{prop:freegrad}
		For any adversarial sequence of convex losses $(\ell_t)$ on $\reals^d$, the iterates $(\w_t)$ of \freegrad{} with parameter $\epsilon,R>0$, satisfy for all $T\geq 1$ and $\w\in \cB(R)$,	
		\begin{align}
			\sum_{t=1}^T \inner{\bar \bnabla_t}{\w_t - \w} \leq  2 \|\w\| \sqrt{\bar V_T\ln_+ \left(\frac{ 2\|\w\| \bar V_T }{\epsilon^2} \right)}     +  4 B_T \|\w\|  \ln \left(  \frac{4B_T  \|\w\| \sqrt{\bar V_T}  }{ \epsilon^2} \right)+\epsilon , \nn
		\end{align}
		where $\bnabla_t\in \partial \ell_t(\w_t)$ (any sub-grad.),  $B_T \coloneqq \epsilon \vee \max_{t\in[T]} \|\bnabla_t\|$, $\bar \bnabla_t  \coloneqq \bnabla_t \cdot B_{t-1}/B_t$, and $\bar V_T \coloneqq \epsilon^2+ \sum_{t=1}^T \|\bar \bnabla_t\|^2$.
	\end{proposition}
\noindent Technically, the original version of \freegrad{} applies to unbounded OCO, whereas the version of \freegrad{} in Algorithm \ref{alg:freegrad} generates outputs in $\cB(R)$ using the same constrained-to-unconstrained reduction due to \cite{cutkosky2020parameter}. We constrain the output of \freegrad{} to ensure that the iterates $(\w_t)$ of Algorithm \ref{alg:projectionfreewrapper} in the setting of Section \ref{sec:strong} are bounded, which in turn ensures that the approximation error of $\O_{\K^\circ}$ in Algorithm \ref{alg:projectionfreewrapper} is not too large. 
The proof of Proposition \ref{prop:freegrad} follows by \cite[Theorem 6]{mhammedi2020} and \cite[Theorem 2]{cutkosky2020parameter}.

	We close this section by mentioning that the unbounded version of \freegrad{}, where $\tilde \bnabla_t = \bnabla_t$ and $\w_t = \x_t$,  is an FTRL instance with the specific sequence of regularizers $(\phi^{*}_t)$, where $\phi^*_t$ is the Fenchel dual of
	\begin{align}
		\phi_t(\x) \coloneqq \frac{1}{\sqrt{V_{t-1}}}\exp \left( \frac{\|\x\|^2}{2 V_{t-1} + 2 \|\x\|} \right), \quad V_t \coloneqq \epsilon^2 + \sum_{s=1}^t \|\bnabla_s\|^2.\nn
\end{align}
In particular, the iterate $\x_{t+1}$ on Line \ref{eq:output} of Algorithm \ref{alg:freegrad} is given by 
\begin{align}
	\x_{t+1} \in \argmin_{\x\in \reals^d}  \left\{\inner{\G_{t}}{\x} + \phi^*_t(\x) \right\}= \left\{\nabla  \phi_t(-\G_t) \right\}.\nn
	\end{align}
We recall that the Fenchel dual $\phi^*_t$ of $\phi_t$ is defined as $\phi_t^*(\x) = \sup_{\bu \in \reals^d} \left\{ \inner{\bu}{\x} - \phi_t(\bu) \right\}$.

	\section{Linear Optimization on $\K^\circ$ using a Membership Oracle for $\K$}
	\label{app:opt}
	In this section, we restate and prove a slight extension of \cite[Lemma 9 \& 10]{lee2018efficient}, which we need in the proof of Proposition \ref{prop:subgradient} at the end of this section. Our extension involves showing that the approximate subgradient in \cite[Lemma 10]{lee2018efficient} has bounded norm in high probability, which we need in the proof of Proposition \ref{prop:subgradient}. We also make the limiting argument used in the proof of \cite[Lemma 9 \& 10]{lee2018efficient} more precise; in particular, for any convex function $f:\reals^d \rightarrow \reals$ (not necessarily differentiable) and $\varepsilon>0$, we explicitly construct a twice differentiable function $\tilde f_{\varepsilon}$ such that $\|f-\tilde f_{\varepsilon}\|_{\infty}\leq \varepsilon$. This will then allow us to make a limiting argument precise via Fatou's lemma to get the final result we want (see proof of Lemma \ref{lem:separate_conv_funcnew}).

	Throughout this section, we let $U_{\infty}(\bu,\nu)$ denote the uniform distribution over $\cB_{\infty}(\bu,\nu)$, for any $\bu\in \reals^d$ and $\nu >0$. The next lemma is taken from \cite[Lemma 9]{lee2018efficient} with only minor notation adjustments.
	\begin{lemma}
		\label{lem:convex_almost_flat} For any $\w\in \reals^d$, $0<\nu_{2}\leq \nu_{1}$, and twice
		differentiable convex function $h$ defined on $B_{\infty}(\w,\nu_{1}+\nu_{2})$
		with $\norm{\nabla h(\z)}_{\infty}\leq L$ for any $z\in \cB_{\infty}(\w,\nu_{1}+\nu_{2})$
		we have 
		\[
		\E_{\bu\sim U_{\infty}(\w,\nu_{1})}\E_{\z\sim U_{\infty}(\bu,\nu_{2})}\norm{\nabla h(\z)-\g(\bu)}_{1}\leq \frac{\nu_{2} d^{3/2}L}{\nu_{1}} ,
		\]
		where $\g(\bu) \coloneqq \E_{\z\sim U_{\infty}(\bu,\nu_{2})}[\nabla h(\z)]$ and $U_{\infty}(\bu,\nu)$ denotes the uniform distribution over $\cB_{\infty}(\bu,\nu)$.
	\end{lemma}

	\begin{algorithm}
		\caption{Approximate Subgradient Oracle \cite{lee2018efficient}.}
		\label{alg:optimizeold}
		\begin{algorithmic}[1]
			\Require Inputs $\nu_1>0$, $L>0$, $\varepsilon>0$, and $\w\in \reals^d$. 
			\Statex~~~~~~~~~~ Function $\tilde f :\reals^d \rightarrow \reals$. 
			\vspace{0.2cm} 
			\State Set $\nu_2=\sqrt{\frac{\varepsilon \nu_1 }{d^{1/2}L }}.$ 
			\State Sample $\bu \in \cB_{\infty}(\w,\nu_1)$ and $\z\in \cB_{\infty}(\bu,\nu_2)$ independently and uniformly at random.
			\For{$i=1,2,\dots,d$}
			\State Let $\w'_i$ and $\w_i$ be the end point of the interval $\cB_{\infty}(\bu,\nu_2) \cap \{\z+\lambda \e_i : \lambda \in \reals\}$.
			\State Set $\tilde s_i =\frac{1}{2\nu_2} (\tilde f( \w'_i) -\tilde f(\w_i) )$.
			\EndFor
			\State Set $\tilde \s = (\tilde s_1, \dots, \tilde s_d)^\top$.
			\State Return $\tilde \s$. 
		\end{algorithmic}
	\end{algorithm}
	\noindent We now restate and slightly extend \cite[Lemma 10]{lee2018efficient} for twice differentiable functions. We then extend the result to non-differentiable functions using Fatou's lemma---see Lemma~\ref{lem:separate_conv_funcnew}.
	\begin{lemma}
		\label{lem:separate_conv_func} Let $L,\nu_{1}>0$, $\w \in \reals^d$, and $h:\reals^d \rightarrow \reals$ be a twice differentiable convex
		function such that $\norm{\nabla h(\x)}_{\infty}\leq L$, for any $\x\in B_{\infty}(\w,2\nu_{1})$. Also, let $\varepsilon\in(0 , \nu_{1}\sqrt{d}L]$ and $\tilde h:\reals^d \rightarrow \reals$ be such that $\|\tilde h - h\|_{\infty} \leq\varepsilon'$ for some $\varepsilon'>0$. Then, the variables $\nu_2$, $\bu$, $\z$, $\tilde \s$, and $(\w_i,\w_i')_{i\in[d]}$ generated during a run of Alg.~\ref{alg:optimizeold}  with input $(\tilde h,\nu_{1}, L,\varepsilon,\w)$ satisfy
		\begin{gather}
			\forall \v\in\reals^d, \ h(\v)\geq h(\w)+\left\langle \tilde \s,\v-\w\right\rangle - \| \nabla h(\z)- \tilde \s\|_1  \cdot \norm{\v-\w}_{\infty}-4d \nu_{1}L, \nn 
			\intertext{Furthermore, for $\tilde s_i \coloneqq (h(\w_i')- h(\w_i))/(2\nu_2)$, for all $i\in[d]$, we have }
			\| \nabla h(\z)- \tilde \s\|_1 \leq  \| \nabla h(\z)- \tilde \s\|_1 + d \varepsilon'/\nu_2 \quad \text{and} \quad \E[\| \nabla h(\z)- \tilde \s\|_1]\leq 2 d^{5/4}\sqrt{{L \varepsilon }/{\nu_1}}.\nn
		\end{gather}
	\end{lemma}
	\begin{proof}[{\bf Proof}]
		Let $\bu,\w_i$, and $\w_i'$ be the random vectors generated during the call to Algorithm \ref{alg:optimizeold} in the lemma's statement. Further, let $\tilde \s \in \reals^d$ be such that $\tilde s_i \coloneqq (h(\w_i') - h(\w_i))/(2\nu_2)$, for $i\in[d]$. Applying the convexity of $h$ yields, for any $\v \in \reals^d$,
		\begin{align*}
			h(\v) & \geq h(\z)+\left\langle \nabla h(\z),\v-\z\right\rangle \\
			& =h(\z)+\left\langle \tilde \s,\v-\w\right\rangle +\left\langle \nabla h(\z)-\tilde \s,\v-\w\right\rangle +\left\langle \nabla h(\z),\w-\z\right\rangle \\
			& \geq h(\w)+\left\langle \tilde \s,\v-\w\right\rangle -\norm{\nabla h(\z)-\tilde \s}_{1}\norm{\v-\w}_{\infty}-\norm{\nabla h(\z)}_{\infty}\norm{\w-\z}_{1}.
		\end{align*}
		Now, $\norm{\nabla h(\z)}_{\infty}\leq L$ and $\norm{\w-\z}_{1}\leq d\cdot\norm{\w-\z}_{\infty}\leq 2 d(\nu_{1}+\nu_{2})$
		by definition of $\z$ in Algorithm \ref{alg:optimizeold}. Furthermore, since $\varepsilon\leq \nu_1 \sqrt{d} L$, we have $\nu_{2}=\sqrt{\frac{\varepsilon \nu_{1}}{{d}^{1/2}L}}\leq \nu_{1}$. Plugging these facts in the inequality of the previous display implies 
		\begin{align}
			h(\v)  & \geq h(\w)+\left\langle \tilde \s,\v-\w\right\rangle -\norm{\nabla h(\z)-\tilde \s}_{1}\norm{\v-\w}_{\infty}-4 d \nu_1 L.\nn
		\end{align}
		This shows the first inequality we are after. Now, by the definition of $\tilde \s$ in Algorithm \ref{alg:optimizeold} and the fact that $\tilde h$ satisfies $\|\tilde h- h\|_{\infty}\leq \varepsilon'$, we have 
		$$ \norm{\nabla h(\z)-\tilde \s}_{1}  \leq    \norm{\nabla h(\z)-\tilde \s}_{1} + d\varepsilon'/\nu_2.
		$$
		It remains to bound $\E[\|\nabla h(\z) -\tilde \s\|]$. For this, let $\g(\bu)\coloneqq \E_{\z\sim U_{\infty}(\bu, \nu_2)}[\nabla h(\z)]$ and note that
		\begin{align*}
			\E_{\z}\left[ \left|\tilde s_{i}-[\g(\bu)]_{i}\right|\right] & =\E_{\z}\left[\left|\frac{h(\w'_{i})-h(\w_{i})}{2\nu _{2}}-[\g(\bu)]_{i}\right| \right], \\
			& \leq\E_{\z} \left[ \frac{1}{2\nu_{2}}\int\left|\frac{dh}{dw_{i}}(\z+\lambda \e_{i})-[\g(\bu)]_{i}\right|d\lambda \right], \\
			& =\E_{\z}\left[\left|\frac{dh}{dw_{i}}(\z)-[\g(\bu)]_{i}\right|\right],
		\end{align*}
		where we used that both $\z+\lambda \e_{i}$ and $\z$ are uniform distribution
		on $\cB_{\infty}(\bu ,\nu_{2})$ in the last line. Hence, we have
		\begin{align}
			\E_{\z}\left[ \norm{\tilde \s-\nabla h(\z)}_{1} \right]\leq\E_{z}\left[\norm{\nabla h(\z)-\g(\bu)}_{1}+\E_{\z}\norm{\tilde \s-\g(\y)}_{1}\right]&\leq2\E_{\z}\left[\norm{\nabla h(\z)-\g(\y)}_{1}\right]\leq 2d^{5/4}\sqrt{\frac{\varepsilon L}{\nu_1}},\nn
		\end{align}
		where the last inequality follows by Lemma \ref{lem:convex_almost_flat} and the fact that $\nu_{2}=\sqrt{\frac{\varepsilon \nu_{1}}{d^{1/2}L}}\leq \nu_{1}$.
	\end{proof}

	\begin{lemma}
		\label{lem:separate_conv_funcnew} Let $L,\nu_{1}>0$, $\w \in \reals^d$, and $f:\reals^d \rightarrow \reals$ be convex (not necessarily differentiable)
		function such that $\sup_{\g \in \partial  f(\x)}\norm{\g}\leq L$, for any $\x\in B_{\infty}(\w,2\nu_{1})$. Also, let $\varepsilon\in(0 , \nu_{1}\sqrt{d}L]$ and $\tilde f:\reals^d \rightarrow \reals$ be such that $\|\tilde f - f\|_{\infty} \leq \varepsilon$. Then, the output $\tilde \s$ of Algorithm \ref{alg:optimizeold}  with input $(\tilde f,\nu_{1}, L,\varepsilon,\w)$ satisfies
		\begin{gather}
			\forall \v\in\reals^d, \ f(\v)\geq f(\w)+\left\langle \tilde \s,\v-\w\right\rangle - X \cdot \norm{\v-\w}_{\infty}-4d \nu_{1}L;\nn \\
		\| \tilde \s\|_{\infty} \leq L+d^{1/4} \sqrt{L \varepsilon/\nu_1};\ \ \|\tilde \s\|\leq X+ L \leq \|\tilde \s\|_1 + (d+1)L;  \quad \text{and} \ \  \|\tilde \s\|^2 \leq   \left(4 L + d^{1/4} \sqrt{L \varepsilon/\nu_1}\right) X + L^2,  \nn
		\end{gather}
		where $X \geq 0$ is a non-negative random variable satisfying $\E[X ]\leq 3  d^{5/4}\sqrt{{L \varepsilon }/{\nu_1}}$.
	\end{lemma}
	\begin{proof}[{\bf Proof}]
		Let $\nu_2$, $\bu$, $\z$, $\tilde \s$, and $(\w_i,\w_i')_{i\in[d]}$ be the variables generated during the call to Algorithm \ref{alg:optimizeold} in the lemma's statement. Further, let $\sigma >0$ and  $h_\sigma:\reals^d\rightarrow \reals$ be such that 
		\begin{align}
			h_{\sigma}(\w) \coloneqq \E [f(\w+\x)], \quad \text{where} \quad \x \sim \cN(\bm{0},\sigma^2  I_d). \nn
		\end{align}
		It is known that $h_{\sigma}$ is twice differentiable \cite{bhatnagar2007adaptive,nesterov2017random,abernethy2014online}, and satisfies \cite[Lemma 9]{duchi2012randomized},
		\begin{align}
			\forall \w\in \reals^d, \quad 
			\|\nabla h_{\sigma}(\w)\| \leq L, \quad \text{and} \quad f(\w) \leq h_{\sigma}(\w)  \leq f(\w)+L \sigma \sqrt{d}. \label{eq:sand}
		\end{align}
		Thus, we have $\|\tilde f  - h_{\sigma}\|_{\infty}\leq \varepsilon_{\sigma} \coloneqq \varepsilon + L\sigma \sqrt{d}$, and so by Lemma \ref{lem:separate_conv_func}
		\begin{gather}
			h_{\sigma}(\v)\geq h_{\sigma}(\w)+\left\langle \tilde \s,\v-\w\right\rangle - \| \nabla h_{\sigma}(\z)- \tilde \s\|_1  \cdot \norm{\v-\w}_{\infty}-4d \nu_{1}L, \quad \forall \v\in \reals^d, \nn
			\intertext{and for $\tilde \s_{\sigma} \coloneqq (h_{\sigma}(\w_1')-h_{\sigma}(\w_1) ,\dots,h_{\sigma}(\w_d')-h_{\sigma}(\w_d) )^\top/(2\nu_2)$, }
			\|\nabla h_{\sigma}(\z) - \tilde \s\|_1 \leq \|\nabla h_{\sigma}(\z) - \tilde \s_{\sigma}\|_1+ d \varepsilon_{\sigma}/\nu_2 ,\quad \text{and} \quad \E[\|\nabla h_{\sigma}(\z) - \tilde \s_{\sigma}\|_1]\leq 2 d^{5/4} \sqrt{L \varepsilon /\nu_1}, \label{eq:thesecond}
		\end{gather}
		where the expectation is with respect to $\z\sim U_{\infty}(\bu,\nu_2)$ and $\bu \sim U_{\infty}(\w,\nu_1)$. Combining this with \eqref{eq:sand} leads to 
		\begin{align}
			f(\v)\geq f(\w)+\left\langle \tilde \s,\v-\w\right\rangle - \| \nabla h_{\sigma}(\z)- \tilde \s\|_1  \cdot \norm{\v-\w}_{\infty}-4d \nu_{1}L- L \sigma \sqrt{d},\quad \forall \v\in \reals^d. \label{eq:thethird}
		\end{align}
		Now, define $X_n \coloneqq \|\nabla h_{1/n}(\z) - \tilde \s\|_1$, $n\geq 1$ and note that
		\begin{align}
			0\leq X_n\leq  d L + \|\tilde \s\|_1 <+\infty, \quad \forall n. \nn
		\end{align}
		Therefore, we have $X \coloneqq \liminf_{n\to \infty}X_n \leq \|\tilde \s\| +dL<+\infty$. Furthermore, by Fatou's lemma and \eqref{eq:thesecond}, we have 
		\begin{align}
			\E[X] =\E[\liminf_{n\to \infty} X_n] \leq \liminf_{n\to \infty} \E[X_n]\leq \liminf_{n\to \infty} \E[\|\nabla h_{1/n}(\z) - \tilde \s_{1/n}\|_1] + d \varepsilon/\nu_2\leq 3 d^{5/4} \sqrt{L \varepsilon /\nu_1},\nn
		\end{align} 
		where the last inequality also follows by \eqref{eq:thesecond} and the fact that $\nu_{2}=\sqrt{\frac{\varepsilon \nu_{1}}{d^{1/2}L}}$. Combining this with \eqref{eq:thethird} leads to the first inequality of the lemma. Now, we have, for any $i\in[d]$,
		\begin{align}
			\tilde s_i = \frac{\tilde f(\w_i')- \tilde f (\w_i)}{2\nu_2} \leq  \frac{f(\w_i')-  f (\w_i)}{2\nu_2} +  \frac{\varepsilon}{\nu_2} \leq  L +   \frac{\varepsilon}{\nu_2} = L + d^{1/4} \sqrt{\frac{L \varepsilon}{\nu_1}}, \nn 
		\end{align}
		where the last inequality follows by the definition of $\w'_i$ and $\w_i$, and the fact that $\sup_{\g\in \partial f(\x)} \|\g\|\leq L$ (and so $f$ is $L$-Lipschitz). In particular, this implies that
		\begin{align}
			\forall i \in[d],\quad |\tilde s_i - [\nabla h_{1/n} (\z)]_i| \leq 2 L + d^{1/4}\sqrt{L\varepsilon/\nu_1}. \label{eq:thirdthird}
		\end{align}
		Using this, we get 
		\begin{align}
			\|\tilde \s\|^2  &\leq \| \tilde \s - \nabla h_{1/n}(\z)\|^2 +2 \| \tilde \s - \nabla h_{1/n}(\z)\| \|\nabla h_{1/n}(\z)\| + \|\nabla h_{1/n}(\z)\|^2,\nn \\ & \leq     \left(2 L + d^{1/4} \sqrt{L \varepsilon/\nu_1}\right) \| \tilde \s - \nabla h_{1/n}(\z)\|_1  + 2 \|\tilde \s - \nabla h_{1/n}(\z) \|_1  L +  L^2, \quad \text{(by \eqref{eq:thirdthird})} \nn\\
			& \leq \left(4 L + d^{1/4} \sqrt{L \varepsilon/\nu_1}\right) \| \tilde \s - \nabla h_{1/n}(\z)\|_1   +  L^2, \nn \\
			& \leq  \left(4 L + d^{1/4} \sqrt{L \varepsilon/\nu_1}\right)  \liminf_{n\to \infty}  \| \tilde \s - \nabla h_{1/n}(\z)\|_1 + L^2,\nn \\
			& =    \left(4 L + d^{1/4} \sqrt{L \varepsilon/\nu_1}\right) X + L^2.\nn
		\end{align}
	It remains to bound the norm $\|\tilde \s\|$. Similar to how we bounded $\|\tilde \s\|^2$, we have
		\begin{align}
			\|\tilde \s\| &  \leq \| \tilde \s - \nabla h_{1/n}(\z)\| +\|\nabla h_{1/n}(\z)\|,\nn \\ & \leq \| \tilde \s - \nabla h_{1/n}(\z)\|_1 +L,\nn \\
		&  \leq       \liminf_{n\to \infty}  \| \tilde \s - \nabla h_{1/n}(\z)\|_1 + L = X + L.\nn
		\end{align}
		This completes the proof.
	\end{proof}

		\subsection{Proof of Lemma \ref{lem:gauge} (Approximate Gauge Function)}
		\label{sec:gaugeproof}
	\begin{proof}[{\bf Proof of Lemma \ref{lem:gauge}}]
Let $\epsilon = \delta r/(4 \kappa)^2$. We will first show that for all $\w \in \cB(6R/5)$,
		\begin{align}
		[\gamma_{\K}(\w) \geq 9/16\  \text{ or }\ \tilde \gamma \geq 1] \implies	[\mem_{\K}(2\w;\epsilon)=0  \ \text{ and }  \ \|\w\| \geq r/2]. \label{eq:imple}
		\end{align}
		Suppose that $\tilde \gamma \geq 1$. By Lines \ref{line:memset} and \ref{line:memsetnext} of Algorithm \ref{alg:gauge}, this implies that $\mem_{\K}(2\w;\epsilon)=0$. Now suppose that $\gamma_{\K}(\w)\geq 9/16$. Note that $\mem_{\K}(2\w;\epsilon)=1$ only if $2\w\in \cB(\K,\epsilon)$. By the fact that $\cB(r) \subseteq \K$ (Assumption \ref{ass:assum}) and $\delta < 1$, we have that $\K + \cB(\epsilon)= \cB(\K,\epsilon)$. 
		Thus, by a standard result in convex analysis, see e.g.~\cite[Thm~C.3.3.2]{hiriart2004}, we have, for all $\x\in \K$,
		\begin{align}
			\sigma_{\cB(\K,\epsilon)}(\x) = \sigma_{\K}(\x) +\sigma_{\cB(\epsilon)}(\x). \label{eq:dash}
		\end{align}
		Let $\x_* \in \partial \gamma_{\K}(\w)=\argmax_{\x \in \K^\circ} \inner{\x}{\w}$. Since $\x_*\in \K^\circ $, we have $\inner{\x_*}{\y}\leq 1$, for all $\y \in \K$, by definition of $\K^\circ$. Further, since $\w/\gamma_{\K}(\w)\in \K$ (implied by Lemma \ref{lem:projectiononK}) and $\inner{\x_*}{\w} =\gamma_{\K}(\w)$, we have 
		\begin{align}
			\sigma_{\K}(\x_*) = \sup_{\y \in \K}	\inner{\x_*}{\y} =  	\inner{\x_*}{\w/\gamma_{\K}(\w)} = 1. \label{eq:support}
		\end{align}
		Plugging this into \eqref{eq:dash} and using the fact that $\sigma_{\cB(\epsilon)}(\x_*)=  \delta r \|\x_*\|/(8\kappa^2)$, we get
		\begin{align}
			\sigma_{\cB(\K,\epsilon)}(\x_*)=1 + \delta r \|\x_*\|/(8\kappa^2)   \leq 1 + \delta/(8\kappa^2) < 9/8, \label{eq:seperation}
		\end{align}
		where the last inequality follows by the fact that $\delta \in(0,1)$ and $\|\x_*\|\leq 1/r$ (because $\x_* \in \K^\circ$, see Lemma \ref{lem:properties}-(c)). On the other hand, since $\gamma_{\K}(\w)\geq 9/16$, we have, by \eqref{eq:support}
		\begin{align}
			9/8=9/8\inner{\x_*}{\w/\gamma_{\K}(\w)} \leq 	9/8\inner{\x_*}{16\w/9}= \inner{\x_*}{2\w}. \nn
		\end{align}
		This inequality together with \eqref{eq:seperation} implies that the vector $\x_*$ separates $2 \w$ from $\cB(\K,\epsilon)$ and so we have $\mem_{\K}(2 \w; \epsilon)=0$ by definition of the Membership Oracle $\mem_{\K}(\cdot;\epsilon)$. It remains to show that $\|\w\|\geq r/2$. If $\gamma_{\K}(\w)\geq 9/16$, then $\|\w\| \geq 9r/16$ by the fact that $\gamma_{\K}(\w)\leq \|\w\|/r$ (Lemma \ref{lem:properties}-(c)). Also, if $\tilde \gamma \geq 1$, then Lines \ref{line:memset} and \ref{line:memsetnext} of Algorithm \ref{alg:gauge} imply that $\|\w\| \geq r/2$. So far, we have shown that \eqref{eq:imple} holds. By definition of $\alpha$ and $\beta$ in Algorithm \ref{alg:gauge}, \eqref{eq:imple} implies that if $\gamma_{\K}(\w) \geq 9/16$ or $\tilde \gamma \geq 1$, then
		\begin{align}
		\mem_{\K}(\alpha\w; \epsilon) =1 \ \  \text{and}  \ \  	\mem_{\K}(\beta\w; \epsilon) =0. \label{eq:always}
		\end{align}
		Next, we will show that if either $\gamma_{\K}(\w)\geq 9/16$ or $\tilde \gamma\geq 1$, then for any $\mu>0$,
		\begin{align}
			\mem_{\K}(\mu\w; \epsilon)=0 \implies \mu \geq \frac{1}{\gamma_{\K}(\w)} -\frac{\delta }{8\kappa^2}  \quad  \text{and}  \quad \mem_{\K}(\mu\w; \epsilon)=1 \implies \mu \leq \frac{1}{\gamma_{\K}(\w)} +\frac{\delta }{8\kappa^2}. \label{eq:implications}
		\end{align}
	Suppose that $\mem_{\K}(\mu\w; \epsilon)=0$. By \eqref{eq:mem}, this implies that $\mu \w \not\in \cB(\K, -\epsilon)$. On the other hand, since $\w/\gamma_{\K}(\w)$ is on the boundary of $\K$ (by definition of the Gauge function and Lemma \ref{lem:projectiononK}), the vector $\v \coloneqq \w/\gamma_{\K}(\w)- \epsilon \w/\|\w\|$ is in $\cB(\K,-\epsilon)$. Since $\v$ and $\mu \w$ are aligned, the fact that $\mu \w \not\in \cB(\K, -\epsilon)$ and $\K$ is convex implies that
		\begin{align}
			\|\mu \w\| \geq \|\w/\gamma_{\K}(\w) - \epsilon  \w /\|\w\| \| = \|\w\|/\gamma_{\K}(\w) - r \delta /(16\kappa^2) \geq   \|\w\|/\gamma_{\K}(\w) - \|\w\| \delta/(8\kappa^2),\nn
		\end{align}
	where the last inequality follows by the fact that $\|\w\|\geq r/2$ (see \eqref{eq:imple}). Dividing by $\|\w\|$ (this is non-zero by \eqref{eq:imple}) on both sides shows the first implication in \eqref{eq:implications}. Similarly, if $\mem_{\K}(\mu\w; \epsilon)=1$. Again by \eqref{eq:mem}, this implies that $\mu\w \in \cB(\K,\epsilon)$. Combining this with the fact that $\w/\gamma_{\K}(\w)+\epsilon  \w/\|\w\|$ is on the boundary of $\cB(\K,\epsilon)$ (since $\w/\gamma_{\K}(\w)$ is on the boundary of $\K$) implies that 
		\begin{align}
			\|\mu\w\| \leq  \|\w/\gamma_{\K}(\w) + \epsilon \w/\|\w\|\|= \|\w\|/\gamma_{\K}(\w) + \delta r/ (16\kappa^2) \leq  \|\w\|/\gamma_{\K}(\w) + \delta \|\w\|/ (8\kappa^2),\nn
		\end{align}
		where the last inequality follows by the fact that $\|\w\|\geq r/2$ (see \eqref{eq:imple}). This shows the second implication in \eqref{eq:implications}. Let $\tilde \alpha$ and $\tilde \beta$ be the values of $\alpha$ and $\beta$ in Algorithm \ref{alg:gauge} after the while-loop is over. By combining \eqref{eq:always}, which we recall holds when $\gamma_{\K}(\w) \geq 9/16$ or $\tilde \gamma \geq 1$, and \eqref{eq:implications}, we get 		
		\begin{gather} \tilde \gamma \geq  \gamma_{\K}(\w)  ;\ \ \text{and} \ \   \frac{1}{\gamma_{\K}(\w)}- \frac{1}{\tilde \gamma} = \frac{1}{\gamma_{\K}(\w)}- \left(\tilde \alpha - \frac{\delta }{8\kappa^2}\right) \leq \tilde  \beta - \tilde\alpha + \frac{\delta}{8\kappa}  +\frac{\delta}{8\kappa^2}\leq \frac{3\delta}{8\kappa^2 },\label{eq:bad}
		\end{gather}
		where the last inequality follows by the fact that the while-loop in Algorithm \ref{alg:gauge} terminates when $\beta -\alpha \leq \delta/(8\kappa^2)$. Multiplying both sides of \eqref{eq:bad} by $\tilde \gamma \cdot\gamma_{\K}(\w)$, we get 
		\begin{align}
			\tilde \gamma - \gamma_{\K}(\w) \leq \tilde \gamma \cdot \gamma_{\K}(\w) \cdot {3 \delta}/({8 \kappa^2}). \label{eq:poststarte}
		\end{align}
		Using \eqref{eq:bad}, we get that $\tilde \gamma \leq (1/\gamma_{\K}(\w) -3\delta/(8\kappa^2))^{-1}$. Plugging this into \eqref{eq:poststarte}, we get 
		\begin{align}
			\tilde \gamma - \gamma_{\K}(\w) \leq  \frac{\gamma_{\K}(\w) \cdot  {3\delta}/{(8\kappa^2)}}{{1}/{\gamma_{\K}(\w)} - 3 {\delta}/({8\kappa^2})} \stackrel{(a)}{\leq}  \frac{6\kappa/5 \cdot  {3\delta}/({8\kappa^2})}{5/(6 \kappa) - {3\delta}/({8\kappa^2})}\leq  \delta, \nn
		\end{align}
		where $(a)$ follows by the fact that $\gamma_{\K}(\w)\leq  {\|\w\|}/{r}\leq 6\kappa/5$, for all $\w\in \cB(6R/5)$, and the last inequality follows by the fact that $\delta \in(0,1)$ and $\kappa \geq 1$.
		
		We now consider the computational complexity. Note that at every iteration of the while-loop in Line \ref{line:cond} of Algorithm \ref{alg:gauge}, the difference $\beta-\alpha$ halves. Thus, since $\beta-\alpha$ is initially equal to 2, the algorithm terminates (i.e.~when $\beta - \alpha\leq \delta/(8\kappa^2)$) after at most $\ceil{\log_2((4\kappa)^2/\delta)}+1$ steps.
	\end{proof}

		\subsection{Proof of Proposition \ref{prop:subgradient} (Approximate LO Oracle on $\K^\circ$)}
		\label{sec:subgradientproof}
	\begin{proof}[{\bf Proof of Proposition \ref{prop:subgradient}}]
		Let $\varepsilon$ and $\tilde \gamma$ be as in Algorithm \ref{alg:optimize}. First, suppose that $\tilde \gamma\geq 1$. We will show \eqref{eq:bisection} using the result of Lemma \ref{lem:separate_conv_funcnew} with $(f,\tilde f)=(\gamma_{\K},\gau_{\K}(\cdot;\varepsilon))$ and $(\nu_1,\nu_2,\w)$ as in Algorithm \ref{alg:optimize}. First, note that by our choice of $\nu_1$ and $\varepsilon$, the technical condition $\varepsilon \leq  \nu_1 \sqrt{d} L$ under which Lemma \ref{lem:separate_conv_funcnew} holds translates to $\delta \leq 10 d^{3/2}\kappa$. This holds since $\delta \in(0,1)$ and $\kappa,d\geq 1$. We also need to check the technical condition $\|f-\tilde f\|_{\infty}\leq \varepsilon$ of  Lemma \ref{lem:separate_conv_funcnew}. We note that in Algorithm \ref{alg:optimize} we only ever evaluate $\tilde f = \gau_{\K}(\cdot ;\varepsilon)$ at the points $(\w_i,\w_i')_{i\in[d]}$, and so we only need to check the condition
		\begin{align}
			|\gamma_{\K}(\w_i)- \gau_{\K}(\w_i;\varepsilon)| \vee 	|\gamma_{\K}(\w'_i)- \gau_{\K}(\w'_i;\varepsilon)| \leq \varepsilon,\quad \forall i\in[d].\label{eq:prereq}
		\end{align}
		The condition would follow from Lemma \ref{lem:gauge} if $\w_i$ and $\w_i'$ satisfy $\gamma_{\K}(\w_i)\wedge \gamma_{\K}(\w_i') \geq 9/16$ and $\w_i,\w_i'\in \cB(6R/5)$, for all $i\in[d]$. Let $i\in[d]$ and $\v \in\{\w_i,\w_i'\}$. By definition of $\w_i$, $\w_i'$, and $\bu$ in Algorithm \ref{alg:optimize}, we have 
		\begin{align}
			\|\w- \v\| & \leq \sqrt{d} \|\w -\bu  \|_{\infty} + \sqrt{d} \| \bu - \v\|_{\infty}\leq \sqrt{d} \nu_1 + \sqrt{d} \nu_2\leq    \frac{11 r \delta}{100},\label{eq:sink}
		\end{align}
		where the last inequality follows by our choice of $\nu_1$ and $\nu_2$. Combining \eqref{eq:sink} with the fact that $\w\in \cB(R)$ and triangular inequality, we get $\|\v\| \leq \|\w\| + 11 r\delta/100 \leq 6 R/5$, and so $\v\in \cB(6R/5)$.
		 On the other hand, by the fact that $\gamma_{\K}=\sigma_{\K^\circ}$ and the sub-additivity of the support function (Lemma \ref{lem:properties}-(e)), we have 
		\begin{align}
			\gamma_{\K}(\v) = \sigma_{\K^\circ}(\v) \geq \sigma_{\K^\circ}(\w) - \sigma_{\K^\circ}(\w - \v)= \gamma_{\K}(\w) - \gamma_{\K}(\w- \v) \stackrel{(a)}{\geq} \gamma_{\K}(\w) - \frac{11 \delta}{100} \geq \frac{9}{16},\nn 
		\end{align}
		where $(a)$ follows by the fact that $\gamma_{\K}(\w - \v)\leq \|\w - \v\|/r$ (Lemma \ref{lem:properties}-(c)) and \eqref{eq:sink}, and the last inequality follows by the fact that $\delta \leq 1/3$ and $\gamma_{\K}(\w)\geq 2/3$; the latter is because $\tilde \gamma\geq 1$ and $\gamma_{\K}(\w)\geq \tilde \gamma -\delta$ (Lemma \ref{lem:gauge}). Therefore, by Lemma \ref{lem:gauge}, \eqref{eq:prereq} holds, and so does the result of Lemma \ref{lem:separate_conv_funcnew} with $(f,\tilde f)=(\gamma_{\K},\gau_{\K}(\cdot;\varepsilon))$ and $(\nu_1,\nu_2,\w)$ as in Algorithm \ref{alg:optimize}. We will now use this fact to show \eqref{eq:bisection}.
		
By Lemma \ref{lem:properties} (points (a) and (b)), we have $\partial \gamma_{\K}(\x)\subseteq\K^\circ$ for any $\x\in\reals^d$, and so we get
		\begin{align}
			\sup_{\x\in \reals^d}\| \partial \gamma_{\K}(\x)\| \leq \sup_{\y\in \K^\circ } \|\y \| \leq \frac{1}{r},\nn
		\end{align}
	where the last inequality follows by Lemma \ref{lem:properties}-(c). This implies that the Lipschitz constant $L$ in Lemma \ref{lem:separate_conv_funcnew} can be set to $1/r$. Furthermore, by our choice of $\varepsilon$ and $\nu_1$ in Algorithm \ref{alg:optimize}, we have $\varepsilon =  \nu_1^3 /(R^2 r\sqrt{d})$ and so Lemma \ref{lem:separate_conv_funcnew} implies that there exists a non-negative random variable $X$ satisfying $\E[X] \leq 3  \nu_1 d/(rR)$, and for all $\bu\in \reals^d$ 
		\begin{align}
			\gamma_{\K}(\bu) &\geq \gamma_{\K}(\w) + \inner{\tilde\s}{\bu - \w} -X  \|\w -\bu\|_{\infty} -4 d \nu_1/r ,\nn
		\end{align}
		where $\tilde \s$ is as in Algorithm \ref{alg:optimize}. Thus, with $\Delta \coloneqq 2 R  X + 4 d \nu_1/r\geq 0$, we get 
		\begin{align}
			\gamma_{\K}(\bu) &\geq \gamma_{\K}(\w) + \inner{\tilde\s}{\bu - \w} - \Delta \cdot \max(1,\|\bu\|/R), \quad \text{and}\quad \E[\Delta]\leq 10 d\nu_1/r,  \label{eq:premium}
		\end{align}
	where we used the fact that $\w\in \cB(R)$. Further, Lemma \ref{lem:separate_conv_funcnew} implies $\|\tilde \s\|_{\infty} \leq  1/r + d^{1/4} \sqrt{L\varepsilon/\nu_1}$, $\|\tilde \s\| \leq X +1/r$, and $\|\tilde \s\|^2 \leq \left(4 /r + d^{1/4} \sqrt{L\varepsilon/\nu_1}\right) X +1/r^2$, and so
		\begin{align}
		\|	\tilde  \s\|_{\infty} \leq d^{5/4} \sqrt{L\varepsilon/\nu_1}+ 1/r; \quad \|\tilde \s\| \leq \Delta/R + 1/r;\quad \text{and} \quad 	\|\tilde \s\|^2 \leq \left(2 /r + d^{1/4} \sqrt{L\varepsilon/\nu_1}\right) \Delta/R +1/r^2. \label{eq:dijon}
		\end{align}
		Now, since $\varepsilon = \nu_1^3/(R^2 r \sqrt{d})$ and $\nu_1 = \delta r /(10 d)$, we get that $10 d\nu_1/r =\delta$ and $d^{1/4} \sqrt{L\varepsilon/\nu_1} = \delta/(10 R d)$. Plugging these into \eqref{eq:premium} and \eqref{eq:dijon}, we get \eqref{eq:bisection} for the case when $\tilde \gamma \geq 1$.

		It remain to show that $\|\tilde \s\|_{\infty}< +\infty $ almost surely and bound $\Delta$. We do not assume that $\tilde \gamma\geq 1$ anymore. Let $i\in[n]$ and $(\w_i,\w'_i)$ be as in Algorithm \ref{alg:optimize}. Let $\v \in\{\w_i,\w_i'\}$ and suppose that $\gau_{\K}(\w;\varepsilon)\geq 1$, then by Lemma \ref{lem:gauge} and Lemma \ref{lem:properties}-(c), we have 
		\begin{align}
			1\leq 	\gau_{\K}(\v;\varepsilon) \leq \gamma_{\K}(\v)+\varepsilon \leq  \|\v\|/r + \varepsilon \leq \frac{R+ \sqrt{d}\nu_1 + \sqrt{d}\nu_2}{r} +\varepsilon < +\infty.\nn
		\end{align} 
		Alternatively, we have $0 \leq\gau_{\K}(\v;\varepsilon)<1$. Therefore, for all $i\in[d]$, we have
		\begin{align}
					|\tilde s_i| = \frac{|\gau_{\K}(\w_i';\varepsilon)- \gau_{\K}(\w_i;\varepsilon)|}{2\nu_2} \leq \frac{R + \sqrt{d}\nu_1 + \sqrt{d}\nu_2}{\nu_2 r} +\frac{\varepsilon}{\nu_2} &= \frac{100 d^{5/2} \kappa ^2}{\delta^2 r}+\frac{10 d^{2} \kappa }{\delta 
			r}+\frac{\delta }{10 d R}+\frac{\sqrt{d}}{r} , \nn \\
		& \leq \frac{112 d^{5/2} \kappa^2}{r\delta^2}<+\infty.  \label{eq:wings}
		\end{align}
	Finally, by Lemma \ref{lem:separate_conv_funcnew}, we have $X\leq \|\tilde\s\|_1 +d/r$ and so combining this with \eqref{eq:wings}, we get 
	\begin{align}
		\Delta = 2 R X + 4 d \nu_1/r \leq \frac{224 d^{4} \kappa^{3}}{\delta^2} + 2d\kappa+ \frac{4\delta}{10}\leq \frac{15^2 d^{4} \kappa^3}{\delta^2}, \nn
	\end{align}
where the last inequality follows by the fact that $\delta \leq 1/3$. This completes the proof. 
	\end{proof}
	\subsection{Proof of Lemma \ref{lem:celeb} (Efficient Stochastic LO Oracle on $\K^\circ$)}
	\label{sec:celebproof}
\begin{proof}[{\bf Proof of Lemma \ref{lem:celeb}}]
	Let $I\in[d]$ be the random variable in Algorithm \ref{alg:optimizenew} generated during the call to $\onedopt$ in the lemma's statement. Note that the approximate Gauge function $\gau_{\K}$ (Alg.~\ref{alg:gauge}) is deterministic and so  $\hat \gamma=\tilde \gamma$. On the other hand, by definition of $\tilde \s$ and $\hat \s$ in Algorithms \ref{alg:optimize} and \ref{alg:optimizenew}, respectively, we have $$\E[\hat \s]=\E[ \hat s_I \cdot \e_I]=  \sum_{i=1}^d d^{-1}   \E[\hat s_i \mid I=i ]\cdot  \e_i \stackrel{(*)}{=}  \sum_{i=1}^d d^{-1}\E[d\tilde s_i]\cdot \e_i= \tilde \s,$$
		 where $(*)$ follows by the fact that, conditioned on $I=i$, $\hat s _i/d$ has the same distribution as $\tilde s_i$. Similarly,
		\begin{align}
			\E[\|\hat  \s\|^2] = \E[\|  \hat s_I \e_I\|^2]  =  \sum_{i=1}^d d^{-1}\E[\hat s_i^2 \e_i\mid I=i]=\sum_{i=1}^d d^{-1} \E[d^2  \tilde  s_i^2\e_i]= d\E[\|\tilde \s\|^2]\leq d \cdot \left(\frac{1}{r} + \frac{\delta}{R}\right)^2, \nn  
		\end{align}
		where the last inequality follows by the fact that $\|\tilde \s\|^2\leq (2/r+\delta/R)\Delta/R + 1/r^2$ (Proposition \ref{prop:subgradient}), where $\Delta\geq 0$ is a random variable satisfying $\E[\Delta]\leq \delta$. Finally, the fact that $\|\hat \s\|<+\infty$ a.s.~follows from the fact that $\|\tilde \s\|_{\infty}<+\infty$ a.s. (Prop.~\ref{prop:subgradient}), and for any $i\in[d]$, conditioned on $I=i$, $\hat s_i/d$ has the same distribution as $\tilde s_i$.
	\end{proof}

		\section{General Regret Reduction}
	\label{sec:genred}
	
	\begin{proposition}
		\label{prop:generalred}
		Let $\delta \in(0,1)$, $B>0$, $T\geq 1$, and $\A$ be the sub-routine of Algorithm \ref{alg:projectionfreewrapper}. Suppose there exists a function $\scR^{\A}: \K \times \cup_{t=1}^\infty \reals^{d\times t}\rightarrow \reals_{\geq 0}$ such that for any adversarial sequence of convex losses $(f_t)$ on $\K$, the iterates $(\w_t)$ of $\A$ in response to $(f_t)$, and the subgradients $\bnabla_t \in \partial f_t(\w_t)$, $t\in [T]$, satisfy 
		\begin{gather} \sum_{t=1}^T \inner{ \bnabla_t}{\w_t- \x} \leq \scR^\A(\x, \bnabla_{1:T}), \forall \x\in \K,\nn 
			\intertext{as long as $\|\bnabla_t \|\leq (1+\frac{\delta}{t}+\kappa)B$, for all $t\in [T] $. Then, for any adversarial sequence of $B$-Lipschitz convex losses $(\ell_t)$, the iterates $(\x_t)$ of Alg.~\ref{alg:projectionfreewrapper} with tolerance sequence $\delta_t \coloneqq \frac{\delta}{t^3}$, and the subgradients $\g_t\in\partial \ell_t(\x_t)$, satisfy}
			\P\left[\begin{array}{c} \forall \x\in \K,\  \sum_{t=1}^T \inner{\g_t}{\x_t - \x}\leq \scR^\A(\x,\tilde\g_{1:T})+ \left(4 +\frac{5\delta\ln T}{\rho}  \right)R B \\  \text{and} \quad \forall t\in[T], \ \|\tilde \g_t\|\leq (1+\delta/t+\kappa)B  \end{array} \right] \geq 1 - 2\rho,\nn
		\end{gather}
		for all $\rho\in(0,1)$ and $(\tilde \g_t)$ as in Algorithm \ref{alg:projectionfreewrapper}.
	\end{proposition}
	\begin{proof}[{\bf Proof}]
		By Lemma \ref{lem:mainreduction}, we have, for all $\w \in \K$ and $t\geq 1$, 
		\begin{align}
			\inner{\g_t}{\x_t - \x}  \leq \inner{\tilde \g_t}{\w_t-\x} + (\delta_t +\Delta_t)R \|\g_t\|,  \label{eq:secondonetwo}
		\end{align}
		where $\Delta_t\geq 0$ is a non-negative random variable satisfying $\E_{t-1}[\Delta_t]\leq \delta_t$. Thus, by the law of total expectation and Markov's inequality, we have 
		\begin{align}
			\P[\Delta_t \geq t^2 \delta_t/\rho ]\leq \rho/t^2.\nn
		\end{align}
		Let $\cE_T$ be the event that $\{\Delta_t \leq t^2 \delta_t/\rho, \forall t\in[T]\}$. By a union bound and the fact that $\sum_{t=1}^{\infty} 1/t^2 \leq 2$, we get that $\P[\cE_T] \geq 1 -2 \rho$. For the rest of this proof, we will condition on the event $\cE_T$. We have, by our choice of $(\delta_t)$, 
		\begin{align}
			\sum_{t=1}^T \Delta_t \leq  \sum_{t=1}^T t^2 \delta_t/\rho =  \sum_{t=1}^T \delta/(\rho t)\leq \frac{\delta \ln T} {\rho}. \label{eq:keye}
		\end{align}
		Also, note that by Lemma \ref{lem:mainreduction} (and the fact that $\cE_T$ holds), we have, 
		\begin{align}
			\|\tilde \g_t\| \leq (1+\Delta_t + \kappa) \|\g_t\|\leq  (1+\delta/t + \kappa) \|\g_t\| \leq (1+\delta/t + \kappa) B,\label{eq:bound}
		\end{align}
		where the last inequality follows by the fact that $\ell_t$ is $B$-Lipschitz. Now, by summing \eqref{eq:secondonetwo} for $t=1,\dots,T$, we obtain, for all $\x\in \K$ 
		\begin{align}
			\sum_{t=1}^T	\inner{\g_t}{\x_t - \x} &\leq \sum_{t=1}^T \inner{\tilde{\g}_t}{\w_t-\x}+\sum_{t=1}^T(2\delta_t +\Delta_t)R \|\g_t\|,\nn \\ 
			& \leq \scR^{\A}(\x,\tilde \g_{1:t}) + \sum_{t=1}^T(2\delta_t +\Delta_t)R \|\g_t\|,\label{eq:regretftrltwo}  \\
			&  \leq  \scR^{\A}(\x,\tilde \g_{1:t}) + (4 +5\delta/\rho \ln (T) )R B, \nn 
		\end{align}
		where \eqref{eq:regretftrltwo} follows by \eqref{eq:bound} and the assumption made about the regret of Algorithm $\A$, and the last inequality follows by \eqref{eq:keye} and the fact that $\ell_t$ is $B$-Lipschitz. 
	\end{proof}
	
	\section{Algorithm Wrapper for Scale-Invariance}
	\label{sec:scaleinvariant}
	In this appendix, we extend the technique of \cite{mhammedi2019,mhammedi2020} by presenting a general reduction (Algorithm \ref{alg:red}) that makes a large class of Online Learning algorithms scale-invariant and gets rid of problematic ratios in their regret bounds. In particular, we will consider all Online Learning algorithms whose regret bound can be expressed as a map $\scR:\K \times \cup_{t\geq 1}  \reals^{d\times t}  \times  \reals \rightarrow \reals_{\geq 0}$ such that for any comparator $\x$ and sequence of observed subgradients $(\bnabla_t)$, the regret of the algorithm after $t\geq 1$ rounds is bounded from above by $\scR(\x,\bnabla_{1:T},L_t/\epsilon)$, where $\epsilon>0$ is a parameter of the algorithm and $L_t\coloneqq \epsilon \vee \max_{t\in[T]}\|\bnabla_t\|$.  We note that so far we have not restricted the class of eligible algorithms by much. We will further require the following monotonicity property, which is satisfied by most popular Online Learning algorithms:
	\begin{align}
		\scR(\bu,\g_{1:s},p)\leq  \scR(\bu,\g_{1:t},q),\ \ \text{for all } p \leq q , s\leq t,\bu \in \K \text{ and } (\g_t)\subset \reals^d. \label{eq:monot}
	\end{align}
	With this, we now state our reduction result for scale-invariance:
	\begin{proposition}
		\label{prop:red}
		Let $\A$ be the sub-routine of Algorithm \ref{alg:red}. Suppose there exists $\scR^{\A}$ satisfying \eqref{eq:monot} and such that for any adversarial sequence of convex losses $(f_t)$ on $\K$, the iterates $(\x_t)$ of $\A$ in response to $(f_t)$ satisfy 
		\begin{gather}\x_1 = \bm{0}  \quad \text{and}\quad  \sum_{s=1}^t \inner{ \bnabla_s}{\x_s- \x} \leq \scR^\A(\x, \bnabla_{1:t}, L_t/\epsilon), \quad \forall t\geq1, \forall \x\in \K,\label{eq:round1}
			\intertext{where $\epsilon>0$ is a parameter of $\A$; $ \bnabla_s \in \partial f_s(\x_s)$, for all $s\in [t]$; and $L_t=\epsilon \vee \max_{s\in [t]}\|\bnabla_s\|$. Then, for any adversarial sequence of convex losses $(\ell_t)$, the iterates $(\y_t)$ of Algorithm \ref{alg:red} in response to $(\ell_t)$ satisfy}
			\sum_{t=1}^T \inner{\g_t}{\y_t - \x}\leq 2\scR^\A(\x,\g_{1:T},T)+4RB_T,\quad  \forall T\geq1, \forall \x\in \K,  \nn 
		\end{gather}
		where $\g_t$, $t\in[T]$, is any subgradients in $\partial \ell_t(\y_t)$ and $B_T=\epsilon \vee \max_{t\in [T]}\|\g_t\|$.
	\end{proposition}
	
	\begin{algorithm}[tbp]
		\caption{Scale-Invariant Wrapper via Restarts.}
		\label{alg:red}
		\begin{algorithmic}[1]
			\Require OCO Algorithm $\A$ on $\K$ taking parameter $\epsilon>0$.
			\vspace{0.2cm} 
			\State Initialize $\tau=1$, $\x^{\tau}_1=\bm{0}$, and $S_0=B_0=0$.
			\For{$t=1,2,\dots $}
			\State Play $\y_t=\x_t^{\tau}$ and observe $\g_t \in \partial \ell_t(\y_t)$.
			\State Set $B_t = B_{t-1}\vee \|\g_t\|$ and $S_t = S_{t-1}+\|\g_t\|/B_t$.
			\If{$B_t=0$} \algcomment{Continue play the zero vector until the first non-zero subgradients}
			\State Set $\y_{t+1}=\y_t$.
			\State Continue.
			\ElsIf{$B_{t-1}=0$ {\bf or} $B_{t}/B_{\tau} \geq  S_t$} \label{eq:restart}  \algcomment{The latter is checking if $B_t/B_\tau \geq \sum_{s=\tau}^t \|\g_s\|/B_s$}
			\State Set $\tau=t$ and $S_{\tau}=0$.
			\State Initialize $\A$ with parameter $\epsilon=B_\tau$. \algcomment{$\tau$ is the new `round 1'}
			\EndIf
			\State Send linear loss $\w \mapsto \inner{\g_t}{\w}$ to $\A$.
			\State Set $\x^{\tau}_{t+1}$ to $\A$'s $(t+2-\tau)$th output given history $((\x^\tau_{i},\w \mapsto \inner{\g_i}{\w}))_{\tau \leq i\leq t}$.
			\EndFor
		\end{algorithmic}
	\end{algorithm}
	
	\begin{proof}[{\bf Proof of Proposition \ref{prop:red}}]
		We say that $\tau=t$ is the start of an epoch of Algorithm \ref{alg:red} if the condition in Line \ref{eq:restart} is satisfied. We use the convention that the ``last epoch'' starts at $t=T+1$. Let $\tau<t$ be the start of two consecutive epochs. Then, by the condition on Line \ref{eq:restart} of Algorithm \ref{alg:red}, we have 
		\begin{align}
			B_{t-1}/B_{\tau} \leq   \sum_{s=\tau}^{t-1} \| \g_s\|/B_s \leq t-1, \quad \text{and}\quad B_{t}/B_{\tau} >   \sum_{s=\tau}^{t} \| \g_s\|/B_s. \label{eq:restartcond}
		\end{align}
		Recall also that at round $s=\tau$, sub-routine $\A$ is initialized with $\epsilon = B_{\tau}$. Thus, by our assumption in \eqref{eq:round1}, we have $\x^{\tau}_\tau =\bm{0}$ and 
		\begin{align}
			\forall \x\in \K, \quad \sum_{s=\tau}^{t-1} \inner{\g_t}{\y_t-\x} =\inner{\y_\tau}{\g_\tau}+ \sum_{s=\tau}^{t-1} \inner{\g_t}{\x^{\tau}_t-\x} &\stackrel{\eqref{eq:round1}}{\leq}\inner{\y_\tau}{\g_\tau} + \scR^\A(\x,\g_{\tau:t-1}, B_{t-1}/\epsilon),\nn \\
			& \leq R B_t + \scR^\A(\x,\g_{\tau:t-1},t-1), \nn\\
			& \leq R B_T + \scR^\A(\x,\g_{1:T},T),\label{eq:within}
		\end{align}   
		where the last inequality follows by the fact that $\scR^\A \in \cM$. If there are two epochs or less, summing \eqref{eq:within} across the epochs leads to the desired result. Now suppose that there are more than two epochs, and let $\tau$ [resp.~$t$] be the start of the ante-penultimate [resp.~penultimate] epoch (recall that the last epoch starts at $T+1$ by convention). By \eqref{eq:within}, the regret across these two epochs is bounded as 
		\begin{align}
			\sum_{s=\tau}^{T} \inner{\g_t}{\y_t-\x}  \leq 2 R B_T + 2 \scR^\A(\x,\g_{1:T},T).
			\label{eq:lasttwo}
		\end{align}
		We will now bound the regret across the earlier epochs. We have 
		\begin{align}
			\sum_{s=1}^{\tau-1}  \inner{\g_t}{\y_t-\x} \leq 2 R\sum_{s=1}^{\tau-1}  \|\g_s\| \leq 2 R B_{\tau }\sum_{s=1}^{\tau-1}  \frac{\|\g_s\|}{B_s}\leq 2 R B_{\tau }\sum_{s=1}^{t}  \frac{\|\g_s\|}{B_s} \stackrel{\eqref{eq:restartcond}}{<} 2 R B_t \leq 2 RB_T. \label{eq:earlier}
		\end{align}
		Combining \eqref{eq:lasttwo} and \eqref{eq:earlier} leads to the desired result.
	\end{proof}

	\section{Proofs of the Regret Bounds and Convergence Rates}
	\label{sec:proofs}	
	\subsection{Proof of Lemma \ref{lem:mainreduction} (Instantaneous Regret Bound)}
	\label{sec:mainreductionproof}
	To avoid expensive projections our main algorithm (Alg.~\ref{alg:projectionfreewrapper}) makes use of surrogate losses of the form $\tilde \ell(\w) \coloneqq\inner{\g}{\w} + b S_{\K}(\w)$, $\w\in \reals^d$ and $b\geq 0$. The choice of such a surrogate loss function is inspired by existing constrained-to-unconstrained reductions in OCO due to \cite{cutkosky2018black,cutkosky2020parameter}. We will use the approximate optimization Oracle $\O_{\K^\circ}$ from Section \ref{sec:gauge} to compute approximate subgradients of such surrogate losses. In particular, $\O_{\K^\circ}$ (Algorithm \ref{alg:optimize}) guarantees the following:
	\begin{lemma}
	\label{lem:subgradient}
	Let $\g\in \reals^d$, $b\geq  0$, and $\delta \in(0,1/3)$. Let $S_{\K}$ be the Gauge distance function in \eqref{eq:theS} and define $\tilde \ell(\w) \coloneqq\inner{\g}{\w} + b S_{\K}(\w)$. Then, for $\w\in \cB(R)$, the output $(\tilde  \gamma, \tilde \s)$ of Alg.~\ref{alg:optimize} with input $(\w,\delta)$ satisfies \eqref{eq:bisection} and
	\begin{gather}
		\forall \bu \in \cB(R), \ \ \tilde \ell(\bu)\geq \tilde \ell(\w) +  \inner{\g + b  \bnu}{\bu - \w} -b\cdot (\Delta+\delta)\cdot \mathbb{I}_{\{\tilde \gamma \geq 1\}}, \quad \text{where \ $\bnu \coloneqq \mathbb{I}_{\{ \tilde \gamma \geq 1 \}}\tilde \s$,} \nn
	\end{gather}
 and $\Delta\in[0, {15^2 d^{4} \kappa^3}{\delta^{-2}}]$ is the same random variable satisfying \eqref{eq:bisection}; in particular, $\E[\Delta]\leq \delta$. 
\end{lemma}
	\begin{proof}[{\bf Proof}]
	First suppose that $\tilde  \gamma <1$. If $\gamma_{\K}(\w)\geq 9/16$, then by Lemma \ref{lem:gauge}, we have $\gamma_{\K}(\w)\leq \tilde \gamma<1$. Alternatively, $\gamma_{\K}(\w)<9/16<1$. Therefore, by Lemma \ref{lem:projectiononK}, $\partial S_{\K}(\w)=\{\bm{0}\}$, and so $\g \in \partial \tilde{\ell}(\w)$. Since the function $\tilde \ell$ is convex, it follows that 
	\begin{align}
\text{[case where $\tilde \gamma<1$]}	\quad	\forall \bu \in \cB(R), \ \ \tilde \ell(\bu) \geq \tilde \ell(\w) + \inner{\g}{\bu - \w}.  \label{eq:firstcase}
	\end{align}
	Now, suppose that $\tilde\gamma\geq 1$. In this case, by Lemma \ref{lem:gauge}, we have $\gamma_{\K}(\w) \leq \tilde  \gamma \leq \gamma_{\K}(\w)+\delta$. Combining this with the fact that $S_{\K}(\w)=0 \vee (\gamma_{\K}(\w) -1)$ (Lemma \ref{lem:projectiononK}), we get  
	\begin{align}
 \gamma_{\K}(\w) -1 \geq \tilde \gamma -1 - \delta \stackrel{(a)}{\geq}    0 \vee (\gamma_{\K}(\w)-1) -\delta =	S_{\K}(\w) - \delta,\label{eq:game} 
	\end{align}
	where $(a)$ follows from the fact that $\tilde \gamma\geq 1$ (by assumption) and $\tilde \gamma \geq \gamma_{\K}(\w)$. Using this, we get 
	\begin{align}
\text{[case where $\tilde \gamma\geq 1$]}\quad		\forall \bu \in \cB(R), \ \ \tilde \ell(\bu) &=  \inner{\g}{\bu}+  0 \vee (b\gamma_{\K}(\bu)-b),\nn \\
		& \geq \inner{\g}{\bu}+ b \gamma_{\K}(\bu)-b, \nn \\ 
		& \geq \inner{\g}{\bu} + b\gamma_{\K}(\w)-b+  \inner{b \tilde \s}{\bu - \w} - b \Delta ,  \label{eq:case}\\
		& \geq  \inner{\g}{\bu}+ bS_{\K}(\w) + \inner{b \tilde  \s}{\bu - \w} - b \Delta - b\delta, \label{eq:theSS}\\
		& = \tilde \ell(\w) + \inner{\g + b \tilde \s}{\bu - \w} - b \Delta -b \delta, \nn \\
		& =\tilde \ell(\w) + \inner{\tilde \g + b \bnu}{\bu - \w} - b \Delta - b\delta, \label{eq:secondcase}
	\end{align}
where \eqref{eq:case} follows by the fact that $\tilde \gamma\geq 1$ and Proposition \ref{prop:subgradient}; \eqref{eq:theSS} follows by \eqref{eq:game}; and \eqref{eq:secondcase} follows by the fact that $\tilde \s = \bnu$ since $\tilde \gamma\geq 1$. Combining \eqref{eq:firstcase} and \eqref{eq:secondcase} implies the desired result.
\end{proof}

\noindent Lemma \ref{lem:subgradient} shows that for any $t\geq 1$, the vector $\tilde \g_t$ in Algorithm \ref{alg:projectionfreewrapper} is an approximate subgradient of the surrogate loss $\tilde \ell_t(\cdot)\coloneqq \inner{\g_t}{\cdot} - \mathbb{I}_{\inner{\g_t}{\w_t}<0} \inner{\g_t}{\x_t} S_{\K}(\cdot)$, where $\g_t\in \partial \ell_t(\x_t)$ and $\x_t$ is the $t$th iterate of Alg.~\ref{alg:projectionfreewrapper}. We now use this to prove Lemma \ref{lem:mainreduction}. We actually state and prove a slight generalization of Lemma \ref{lem:mainreduction}, which will be useful when considering the stochastic optimization setting of Section \ref{sec:stochastic}. In this generalization, we assume that for some $R'\in[r,R]$, 
\begin{align}
\cB(r)\subseteq \K\subseteq \cB(R')\subseteq \cB(R). \label{eq:sanding} 
\end{align}
\begin{lemma}
		\label{lem:mainreductiongen}
		Let $\kappa\coloneqq R'/r$ and $\A$ be any OCO algorithm on $\cB(R)$, for $r,R,R'>0$ as in \eqref{eq:sanding}. Further, for $t\ge1$, let $\w_t, \tilde{\g}_t, \x_t$, and $\g_t\in \partial \ell_t(\x_t)$ be as in Alg.~\ref{alg:projectionfreewrapper} with $\cO_{\K^\circ}\equiv \O_{\K^\circ}$ and any tolerance sequence $(\delta_s)\subset (0,1/3)$. Then, $(\x_t)\subset \K$, and for all $t\geq 1$, there exists a variable $\Delta_t \in \left[0,{15^2 d^{4} (R/r)^3}{\delta^{-2}_t}\right]$ s.t.~$\E_{t-1}[\Delta_t] \leq \delta_t$, $\|\tilde{\g}_t\|\leq   (1+\Delta_t+\kappa) \|\g_t\|$, and
		\begin{align}
			\forall \x\in \K,\quad 
			\inner{\g_t}{\x_t-\x} \leq \tilde\ell_t(\w_t) - \tilde \ell_t(\x)+\delta_t R \|\g_t\| \leq \inner{\tilde{\g}_t}{\w_t-\x} + (2\delta_t +\Delta_t) R \|\g_t\|. \label{eq:individualnew} 
		\end{align}	
		\end{lemma}
	\begin{proof}[{\bf Proof of Lemma \ref{lem:mainreduction}}]
		Let $\gamma_t$, $\bnu_t$, $\w_t, \tilde{\g}_t$, and $\x_t$ be as in Algorithm \ref{alg:projectionfreewrapper}. We will first show that 
		\begin{align}
			-R\|\g_t\| \leq \mathbb{I}_{\inner{\g_t}{\w_t}<0} \cdot \inner{\g_t}{{\x}_t} \leq 0
			\label{eq:firstgoal}
		\end{align}
		First suppose that $\gamma_t < 1$. In this case, $\x_t =\w_t$, and so 
		\begin{align}
			-R \|\g_t\|\leq \mathbb{I}_{\inner{\g_t}{\w_t}<0} \cdot \inner{\g_t}{{\x}_t} =\mathbb{I}_{\inner{\g_t}{\w_t}<0} \cdot \inner{\g_t}{{\w}_t} \leq  0. \nn
		\end{align}
		Now suppose that $\gamma_t \geq 1$, then $\mathbb{I}_{\inner{\g_t}{\w_t}<0} \cdot \inner{\g_t}{{\x}_t}= \mathbb{I}_{\inner{\g_t}{\w_t}<0}  \cdot  \inner{\g_t}{\w_t/\gamma_t}$ and so 
		\begin{align}
0 	\geq		\mathbb{I}_{\inner{\g_t}{\w_t}<0} \cdot \inner{\g_t}{{\x}_t} \geq  \mathbb{I}_{\inner{\g_t}{\w_t}<0}   \inner{\g_t}{\w_t} \geq - R \|\g_t\|,\nn
		\end{align}
		where the last inequality follows by the fact $\w_t\in \cB(R)$. This shows \eqref{eq:firstgoal}. Now recall the definition of the surrogate function $\tilde \ell_t:\cB(R) \rightarrow \reals$:
		\begin{align}
			\tilde \ell_t(\w) \coloneqq \inner{\g_t}{\w} - \mathbb{I}_{\inner{\g_t}{\w_t} <0} \cdot \inner{\g_t}{{\x}_t} \cdot S_{\K}(\w), \nn
		\end{align}
		where $S_{\K}$ is as in \eqref{eq:theS}. By \eqref{eq:firstgoal} and Lemma \ref{lem:subgradient}, there exists a r.v.~$\Delta_t\in [0,\tfrac{15^2 d^{4} \kappa^3}{\delta^2}]$ such that $\E_{t-1}[\Delta_t]\leq \delta_t$ and 
		\begin{align}
			\forall \x\in \K,\quad \tilde \ell_t(\w_t) - \tilde \ell_t(\x) \leq \inner{\tilde\g_t}{\w_t -\x} + R \|\g_t\|\cdot( \Delta_t+\delta_t). \label{eq:convexity2}
		\end{align}
		It remains to show that $\inner{\g_t}{\x_t -\x}\leq \tilde \ell_t(\w_t) - \tilde \ell_t(\x)+ \delta_t R \|\g_t\|$, for all $t\geq 1$ and $\x \in \K$. First, note that for all $\x \in \K$, we have 	$S_{\K}(\x)=0$, and so
		\begin{align}
			\tilde \ell_t(\x) = \inner{\g_t}{\x}, \quad \forall \x \in \K. \label{eq:firstequality}
		\end{align}
		We will now compare $\inner{\g_t}{\x_t}$ to $\tilde \ell_t(\w_t)$ by considering cases. Suppose that $\gamma_t <1$. In this case, we have $\x_t = \w_t$ and so 
		\begin{align}
			\inner{\g_t}{\x_t} = \inner{\g_t}{\w_t} = \tilde \ell_t(\w_t).\quad \text{[\text{case where $\gamma_t < 1$}]} \label{eq:first}
		\end{align}
 Now suppose that $\gamma_t \geq 1$ and $\inner{\g_t}{\w_t}\geq 0$. In this case, we have $\x_t = \w_t/\gamma_t$, and so
		\begin{align}
			\inner{\g_t}{\x_t}=\inner{\g_t}{\w_t/\gamma_t} \leq \inner{\g_t}{\w_t} = \tilde \ell_t(\w_t). \quad \text{[\text{case where $\gamma_t \geq 1, \inner{\g_t}{\w_t}\geq 0$}]} \label{eq:firstsecond}
		\end{align}
		Now suppose that $\gamma_t \geq 1$, $\inner{\g_t}{\w_t}<0$, and $\w_t\in \K$. We note that this implies that $\gamma_{\K}(\w_t)\leq1$ and so $S_{\K}(\w_t)=0$ (Lemma \ref{lem:projectiononK}). Thus, $\tilde \ell_{t}(\w_t)= \inner{\g_t}{\w_t}$. On the other hand, by Lemma \ref{lem:gauge} we have $\gamma_t \leq \gamma_{\K}(\w_t)+ \delta_t \leq 1 + \delta_t$, and so since $\x_t= \w_t/\gamma_t$, we have 
		\begin{align}
			\inner{\g_t}{\x_t} \leq \inner{\g_t}{\w_t}/(1+\delta_t) = \inner{\g_t}{\w_t} - \delta_t \inner{\g_t}{\w_t}/(1+\delta_t). \nn 
			\end{align}
		Thus, since $\tilde \ell_t(\w_t)=\inner{\g_t}{\w_t}$, the previous display implies that 
		\begin{align}
				\inner{\g_t}{\x_t} \leq \tilde{\ell}_t(\w_t)+ \delta_t R \|\g_t\|. \quad \text{[\text{case where $\gamma_t \geq 1, \inner{\g_t}{\w_t}<0, \w_t \in \K$}]} \label{eq:case2}
		\end{align}
	We now consider the last case where $\gamma_t \geq 1$, $\inner{\g_t}{\w_t}<0$, and $\w_t\not\in \K$. We note that this implies that $\gamma_{\K}(\w_t)\geq 1$. By the fact that $\gamma_t \leq \gamma_{\K}(\w_t) +\delta_t$ (Lemma \ref{lem:gauge}), we have 
		\begin{align}
			\inner{\g_t}{\x_t} \leq  \frac{  \inner{\g_t}{\w_t}}{\gamma_{\K}(\w_t) + \delta_t} =  \frac{ \inner{\g_t}{\w_t/\gamma_{\K}(\w_t)}}{1 + \delta_t/\gamma_{\K}(\w_t)}&= \inner{\g_t}{\w_t/\gamma_{\K}(\w_t)} -\frac{\delta_t/\gamma_{\K}(\w_t)}{1+\delta_t/\gamma_{\K}(\w_t)} \inner{\g_t}{\w_t/\gamma_{\K}(\w_t)},\label{eq:seec} 
		\end{align}
		Thus, since $S_{\K}(\w_t)=\gamma_{\K}(\w_t)-1$ (by Lemma \ref{lem:projectiononK} and $\gamma_{\K}(\w_t)\geq 1$), we get 
		\begin{align}
			\inner{\g_t}{\x_t}\cdot S_{\K}(\w_t)  \leq  \inner{\g_t}{\w_t}- \frac{\delta_t \inner{\g_t}{\w_t} }{1+\delta_t/\gamma_{\K}(\w_t)} - \inner{\g_t}{\w_t/\gamma_{\K}(\w_t)}  +  \frac{\delta_t/\gamma_{\K}(\w_t)}{1+\delta_t/\gamma_{\K}(\w_t)} \inner{\g_t}{\w_t/\gamma_{\K}(\w_t)}. \nn
		\end{align}
		Adding this together with \eqref{eq:seec} and using the fact that $\|\w_t\|\leq R$, we get 
		\begin{gather}	
			\inner{\g_t}{\x_t} + \inner{\g_t}{\x_t} S_{\K}(\w_t)\leq \inner{\g_t}{\w_t} +\delta_t R \|\g_t\|,\nn
			\intertext{and so after rearranging and using that $\tilde{\ell}_t(\w_t)= \inner{\g_t}{\w_t} -  \inner{\g_t}{\x_t} S_{\K}(\w_t)$, we get}
			\inner{\g_t}{\x_t}  \leq  \tilde{\ell}_t(\w_t)+ \delta_t R \|\g_t\|. \quad \text{[\text{case where $\gamma_t \geq 1, \inner{\g_t}{\w_t}<0, \w_t \not\in \K$}]}\label{eq:second}
		\end{gather}
		By combining, \eqref{eq:first}, \eqref{eq:firstsecond}, \eqref{eq:case2}, and \eqref{eq:second}, we obtain:
		\begin{align}
				\inner{\g_t}{\x_t}  \leq \tilde{\ell}_t(\w_t)+ \delta_t R \|\g_t\|.\nn
			\end{align}
	Combining this with \eqref{eq:convexity2} and \eqref{eq:firstequality}, we get
		\begin{align}
			\inner{\g_t}{\x_t -\x}  \leq \tilde \ell_t(\w_t) - \tilde \ell_t(\x) +\delta_t R \|\g_t\|  \leq \inner{\tilde{\g}_t}{\w_t -\x} + (2\delta_t +\Delta_t) R \|\g_t\|, \quad \forall \x \in \K. \nn
		\end{align}
		This shows \eqref{eq:individualnew}. It remains to bound $\|\tilde{ \boldsymbol{g}}_t\|$ in terms of $\|\g_t\|$ and show that $\x_t \in \K$. When $\gamma_t <1$, we have $\tilde \g_t =\g_t$ and $\x_t=\w_t$, and so $\|\tilde\g_t\|=\|\g_t\|$. Furthermore, we also have that $\x_t \in \K$. In fact, if $\gamma_{\K}(\w_t)\leq 9/16$, then by definition of the Gauge function we have $\w_t \in \K$ and the same holds for $\x_t$ (since $\x_t=\w_t$). On the other hand, if $\gamma_\K(\w_t)\geq 9/16$, then by Lemma \ref{lem:gauge}, we have $\gamma_{\K}(\w_t)\leq \gamma_t<1$ ($\gamma_t<1$ is the case we are currently considering), and so $\x_t=\w_t\in \K$. 
		 
	Now suppose that $\gamma_t \geq 1$. In this case, we have $\x_t= \w_t/\gamma_t$ and $\tilde\g_t=\g_t - \mathbb{I}_{\inner{\g_t}{\w_t}<0} \inner{\g_t}{\x_t} \bnu_t$. The latter fact together with the positive homogeneity of the Gauge function (Lemma \ref{lem:properties}-(a,b)) imply that $\gamma_\K(\x_t)=\gamma_\K(\w_t/{\gamma_{t}})=\gamma_\K(\w_t)/{\gamma_{t}}\leq \gamma_\K(\w_t)/{\gamma_{\K}(\w_t)}=1$ (since $\gamma_{\K}(\w_t)\leq \gamma_t$ by Lemma \ref{lem:gauge}), and so $\x_t \in \K$. Using that $\x_t \in \K$ (which implies that $\|\x_t\|\leq R'$) and that $\tilde\g_t=\g_t - \mathbb{I}_{\inner{\g_t}{\w_t}<0} \inner{\g_t}{\x_t} \bnu_t$, we get
		\begin{align}
  	\|\tilde{\g}_t\|  =  \|\g_t- \mathbb{I}_{\inner{\g_t}{\w_t} <0}  \inner{\g_t}{{\x}_t}  \bnu_t\|\leq \|\g_t\|(1 + R' \|\bnu_t\|)  
			 \leq  (1+ \Delta_t +\kappa )\|\g_t\|, \nn
		\end{align}
		where the last inequality follows the fact that $\|\bnu_t\|  \leq \Delta_t/R+1/r$ (by \eqref{eq:bisection} which is implied by Lemma \ref{lem:subgradient}).
	\end{proof}
	
	\begin{proof}[{\bf Proof of Lemma \ref{lem:mainreduction}}]
	Follows from Lemma \ref{lem:mainreductiongen} with $R'=R$.
	\end{proof}

\subsection{Proof of Theorem \ref{thm:genthmeff} (Regret Bound in Expectation)}
\label{sec:genthmeffproof}
\begin{proof}[{\bf Proof}]
	Let $(\tilde \gamma_t, \tilde \s_t)= \O_{\K^\circ}(\w_t;\delta_t)$ and $(\hat \gamma_t, \hat \s_t)= \onedopt(\w_t;\delta_t)$. Further, let $\tilde  \bnu_t\coloneqq \mathbb{I}_{\{\tilde \gamma_t\geq 1  \}} \tilde \s_t$, and $\gamma_t$ and $\bnu_t$ be as in Algorithm \ref{alg:projectionfreewrapper}. Note that by Proposition \ref{prop:subgradient} and Lemma \ref{lem:celeb}, we have $\|\tilde \s_t\|\vee \|\hat \s_t\| < +\infty$ almost surely and so the expectations $\E_{t-1}[ \mathbb{I}_{\{\tilde \gamma_t\geq 1  \}} \tilde \s_t]$ and  $\E_{t-1}[\mathbb{I}_{\{\hat \gamma_t\geq 1  \}} \hat \s_t]$ are well defined. We also note that $\gamma_t=\hat \gamma_t= \tilde \gamma_t$, and $\gamma_t$, $\w_t$, and $\x_t$ are all deterministic functions of the past ($\w_t$ it the output of \ftrl, which is a deterministic function of the past).
	
	By Lemma \ref{lem:mainreduction}, there exists a random variable $\Delta_t\geq 0$ satisfying $\E_{t-1}[\Delta_t]$ and such that for all $t\geq 1$, \begin{align} \forall \x\in \K, \quad 
		\inner{\g_t}{\x_t - \x}  \leq \inner{\g_t - \mathbb{I}_{\inner{\g_t}{\w_t}<0} \inner{\g_t}{\x_t} \tilde \bnu_t}{\w_t-\x} + (2\delta_t +\Delta_t)R \|\g_t\|,  
		\label{eq:six}\end{align}
	Now by Lemma \ref{lem:celeb}, and the law of total expectation, we have, for all $\x\in \K$,
	\begin{align}
		\E[	\inner{\g_t}{\x_t - \x}]& =\E[ \E_{t-1}[ \inner{\g_t}{\x_t - \x} ]], \nn \\ & \leq  \E[ \E_{t-1}[\inner{\g_t - \mathbb{I}_{\inner{\g_t}{\w_t}<0} \inner{\g_t}{\x_t} \tilde \bnu_t}{\w_t-\x}+(2\delta_t +\Delta_t)R \|\g_t\| ]],  \ \ \text{(by \eqref{eq:six})} \nn \\ 
		& \leq  \E[ \E_{t-1}[\inner{\g_t - \mathbb{I}_{\inner{\g_t}{\w_t}<0} \inner{\g_t}{\x_t} \tilde\bnu_t}{\w_t-\x} ] + 3\delta_t R \|\g_t\|],\nn \\ 
		&  = \E[  \mathbb{I}_{\{ \tilde \gamma_t <1\}} \cdot   \E_{t-1}[\inner{\g_t }{\w_t-\x} ]] \nn \\ & \quad  +  \E[ \mathbb{I}_{\{\tilde \gamma_t \geq 1\}} \cdot    \E_{t-1}[\inner{\g_t - \mathbb{I}_{\inner{\g_t}{\w_t}<0} \inner{\g_t}{\x_t} \tilde  \s_t}{\w_t-\x}]+ 3\delta_t R \|\g_t\|],\nn  \\ 
		& = \E[ \mathbb{I}_{\{ \tilde \gamma_t<1\}} \cdot     \E_{t-1}[\inner{\g_t }{\w_t-\x}]] \nn \\ 
		& \quad  +  \E[  \mathbb{I}_{\{ \hat  \gamma_t \geq 1 \}} \cdot  \E_{t-1}[\inner{\g_t - \mathbb{I}_{\inner{\g_t}{\w_t}<0} \inner{\g_t}{\x_t} \hat \s_t}{\w_t-\x} ]+ 3\delta_t R \|\g_t\|], \label{eq:attract}  \\ & =  \E[ \E_{t-1}[\inner{\g_t - \mathbb{I}_{\inner{\g_t}{\w_t}<0} \inner{\g_t}{\x_t} \bnu_t}{\w_t-\x} ] + 3\delta_t R \|\g_t\|], \nn \\ 
		& = \E[\inner{\tilde \g_t}{\w_t -\x} + 3 \delta_t R \|\g_t\|],\label{eq:noone}
	\end{align}
	where \eqref{eq:attract} follows by the facts that $\tilde \gamma_t = \hat \gamma_t$ and $\E_{t-1}[\hat \s_t]= \E_{t-1}[\tilde  \s_t]$, when $\tilde \gamma_t \geq 1$ (Lemma \ref{lem:celeb}).
	Summing \eqref{eq:noone} for $t=1,\dots,T$, we obtain, 
	\begin{align}
		\forall \x\in \K	,\quad 	\sum_{t=1}^T\E[	\inner{{\g}_t}{\x_t - \x}] & \leq \E\left[\sum_{t=1}^T \inner{\tilde{\g}_t}{\w_t-\x}+3\sum_{t=1}^T\delta_t R \|\g_t\|\right], \nn \\
		& \leq \E \left[2 R \sqrt{2\sum_{t=1}^T \| {\tilde{\g}}_t\|^2}  + 3\sum_{t=1}^T\delta_t R \|\g_t\|\right],\label{eq:regretftrllw}  \\
		& \leq2 R \sqrt{\E \left[2\sum_{t=1}^T \E_{t-1} \left[ \| {\tilde{\g}}_t\|^2\right]\right]}  +\E\left[ 3\sum_{t=1}^T\delta_t R \|\g_t\|\right],\label{eq:regretftrlnew}   
	\end{align}
	where the last step follows by Jensen's inequality and \eqref{eq:regretftrllw} follows by our choice of the subroutine $\A$ ($\equiv$\ftrl) and the regret bound of \ftrl{} in Proposition \ref{prop:ftrl-proximal}.
	
	 It remains to bound $\|\tilde \g_t\|$ in terms of $\|\g_t\|$ and show that $\x_t \in \K$. First, if $\gamma_t <1$, then $(\tilde \g_t,\x_t)=(\g_t,\w_t)$, and so $\|\tilde \g_t\|\leq \|\g_t\|(1+R\|\s_t\|)$. Furthermore, we also have that $\x_t \in \K$. In fact, if $\gamma_{\K}(\w_t)\leq 9/16$, then by definition of the Gauge function we have $\w_t \in \K$ and the same holds for $\x_t$ (since $\x_t=\w_t$). On the other hand, if $\gamma_\K(\w_t)\geq 9/16$, then by Lemma \ref{lem:gauge}, we have $\gamma_{\K}(\w_t)\leq \gamma_t<1$, and so $\x_t=\w_t\in \K$. 
	 
	 Now suppose that $\gamma_t \geq 1$. In this case, we have $\x_t = \w_t/ \gamma_t$ and $\bnu_t =  \s_t$. Using this, the triangular inequality, and Cauchy Schwarz, we get 
	\begin{align}
		\|\tilde{\g}_t\|   =  \|\inner{\g_t}{\w} - \mathbb{I}_{\inner{\g_t}{\w_t} <0} \cdot \inner{\g_t}{{\x}_t} \cdot  \s_t\|= 	\|\g_t\| + \|\g_t\| \| \w_t\|/\gamma_t \cdot \|{\s}_t\|
		& \leq  \|\g_t\| (1 + R\| \s_t\|), \nn 
	\end{align}
	where the last inequality follows by the fact that $\gamma_t \geq 1$. Therefore, using Lemma \ref{lem:celeb}, we get $\|\tilde \g_t\|\leq \|\g_t\|(1+\Delta_t + \kappa)$ and
	\begin{align}
		\E_{t-1}[	\|\tilde{\g}_t\|^2 ]\leq 2\|\g_t\|^2 (1 + R^2 \E_{t-1}[\|{\s}_t\|^2])=2\|\g_t\|^2 (1 + R^2 \E_{t-1}[\|\hat{\s}_t\|^2] ) \leq 2 \|\g_t\|^2 (1+ d\cdot  (\kappa+\delta_t)^2).\nn
	\end{align}
	Plugging this into \eqref{eq:regretftrlnew} leads to, for all $\x\in \K$,
	\begin{align}
		\sum_{t=1}^T\E[	\inner{{\g}_t}{\x_t - \x}] &\leq 4 R \sqrt{\sum_{t=1}^T (1+d \cdot (\kappa+\delta_t)^2) \cdot  \E \left[ \| {{\g}}_t\|^2\right]}  + 3\E\left[\sum_{t=1}^T\delta_t R \cdot  \|\g_t\|\right] ,\nn \\
		& \leq 4 R \sqrt{\sum_{t=1}^T (1+d\cdot (\kappa+\delta)^2) \cdot \E \left[ \| {{\g}}_t\|^2\right]}  + 6 \delta R \cdot \E\left[\max_{t\in[T]}\|\g_t\|\right],\nn
	\end{align}
where in the last inequality we used that $\sum_{t=1}^{+\infty}1/t^2 \leq 2$. Furthermore, when $\gamma_t\geq1$, we have $	\gamma_\K(\x_t)=\gamma_\K(\w_t/{\gamma_{t}})=\gamma_\K(\w_t)/{\gamma_{t}}\leq \gamma_\K(\w_t)/{\gamma_{\K}(\w_t)}=1$ (since $\gamma_{\K}(\w_t)\leq \gamma_t$ by Lemma \ref{lem:gauge}), and so $\x_t \in \K$. 
\end{proof}

\subsection{Proof of Theorem \ref{thm:genthmprob} (Regret Bound in High Probability)}
\label{sec:genthmprobproof}
To prove Theorem \ref{thm:genthmprob}, we need the following extension of Lemma \ref{lem:celeb}:
\begin{lemma}
	\label{lem:celebext}
	Let $\delta\in (0,1)$, $\w\in \cB(R)$, and $\kappa \coloneqq R/r$, with $r,R>0$ as in \eqref{eq:lowradius}. Further, let $(\tilde \gamma ,  \tilde \s)=\O_{\K^\circ}(\w;\delta)$ and $( \hat\gamma,   \hat\s)=\onedopt(\w;\delta)$ (Alg.~\ref{alg:optimizenew}). Then, $\hat \gamma =\tilde \gamma$, $\|\hat \s\|<+\infty$ a.s., and if $\hat \gamma \geq 1$ it follows that
	\begin{gather}
		\E[  \hat \s] =  \E[\tilde \s],  \quad     \E[\|  \hat \s\|^2]\leq  d\cdot \left({1}/{r}+ {\delta}/{R} \right)^2,\label{eq:firtrow} \\
		\|\hat \s\|_{\infty} \leq d \cdot (1/r+ \delta/R), \quad  \|\hat \s\| \leq d \cdot (1/r+ \delta/R), \quad  \text{and} \quad \E[\|\hat \s\|^4] \leq d^3\cdot \left({1}/{r}+ {\delta}/{R} \right)^4. \label{eq:secondrow}
	\end{gather}
\end{lemma}
\begin{proof}[{\bf Proof}]
	Lemma \ref{lem:celeb} implies \eqref{eq:firtrow}. We now show \eqref{eq:secondrow}. Proposition \ref{prop:subgradient} implies that $\tilde s_i\leq \delta/R + 1/r$, for all $i\in[d]$. Let $I\in[d]$ be the random variable in Algorithm \ref{alg:optimizenew} generated during the call to $\onedopt$ in the lemma's statement. Since, conditioned on $I=i$, $\hat s _i/d$ has the same distribution as $\tilde s_i$ and $\hat s_i =0$ when $I\neq i$, we get that $\hat s_i\leq d\cdot (1/r+\delta/R)$. This implies the first inequality in \eqref{eq:secondrow}. Using this, we get 
	\begin{align}
		\E[\|\hat  \s\|^4] = \E[\|  \hat s_I \e_I\|^4] = \E[\hat s_I^4 ] \leq d^2 \cdot \left(\frac{1}{r}+\frac{\delta}{R}\right)^2 \cdot   \E[\hat s_I^2 ]= d^2 \cdot \left(\frac{1}{r}+\frac{\delta}{R}\right)^2 \cdot \E[\|\hat \s\|^2]\leq  d^3 \cdot \left(\frac{1}{r}+\frac{\delta}{R}\right)^4,\nn
	\end{align}
	where the last inequality follows the right-most inequality in \eqref{eq:firtrow}. This shows the right-most inequality in \eqref{eq:secondrow}. Finally, we have 
	\begin{align}
		\|\hat \s\| =  |\hat s_I| \leq \max_{i\in[d]} |\hat s_i|\leq  d \cdot (1/r + \delta/R),\nn
	\end{align}
	where the last inequality follows by the first inequality in \eqref{eq:secondrow}, which we already showed. Now we do not assume that $\hat \gamma \geq 1$ anymore. Finally, the fact that $\|\hat \s\|< +\infty$ follows by Lemma \ref{lem:celeb}.
\end{proof}

\begin{proof}[{\bf Proof of Theorem \ref{thm:genthmprob}}]
	The fact that $(\x_t)\subset \K$ follows from Lemma \ref{thm:genthmeff}. Let $(\tilde \gamma_t, \tilde \s_t)= \O_{\K^\circ}(\w_t;\delta_t)$ and $(\hat \gamma_t, \hat \s_t)= \onedopt(\w_t;\delta_t)$. 
	Further, let $\tilde  \bnu_t\coloneqq \mathbb{I}_{\{\tilde \gamma_t\geq 1  \}} \tilde \s_t$, and $\gamma_t$ and $\bnu_t$ be as in Algorithm \ref{alg:projectionfreewrapper}. Note that by Proposition \ref{prop:subgradient} and Lemma \ref{lem:celeb}, we have $\|\tilde \s_t\|\vee \|\hat \s_t\| < +\infty$ almost surely and so the expectations $\E_{t-1}[ \mathbb{I}_{\{\tilde \gamma_t\geq 1  \}} \tilde \s_t]$ and  $\E_{t-1}[\mathbb{I}_{\{\hat \gamma_t\geq 1  \}} \hat \s_t]$ are well defined. We also note that $\gamma_t=\hat \gamma_t= \tilde \gamma_t$, and $\gamma_t$, $\w_t$, and $\x_t$ are all deterministic functions of the past ($\w_t$ it the output of \ftrl, which is a deterministic function of the past).
	
	By Lemma \ref{lem:mainreduction}, there exists a random variable $\Delta_t\geq 0$ satisfying $\E_{t-1}[\Delta_t]$ and such that for all $t\geq 1$, \begin{align} \forall \x\in \K, \quad 
		\inner{\g_t}{\x_t - \x} & \leq \inner{\g_t - \mathbb{I}_{\inner{\g_t}{\w_t}<0} \inner{\g_t}{\x_t} \tilde \bnu_t}{\w_t-\x} + (2\delta_t +\Delta_t)R \|\g_t\|,  \nn \\
		& =  \inner{\tilde \g_t}{\w_t-\x} + \mathbb{I}_{\inner{\g_t}{\w_t}<0}\inner{\g_t}{\x_t}  \inner{ \bnu_t-  \tilde \bnu_t}{\w_t - \x}+ (2\delta_t +\Delta_t)R \|\g_t\|.
		\label{eq:sixten}\end{align}
	Fix $\x\in \K$ and let $X_t\coloneqq \mathbb{I}_{\inner{\g_t}{\w_t}<0}\inner{\g_t}{\x_t} \cdot  \inner{\bnu_t-  \tilde \bnu_t}{\w_t - \x}$. We start by bounding $|X_t|$ from above. By the fact that $\x_t \in \K$ and the $B$-Lipschitzness of $\ell_t$, we have  
	\begin{align}
		\quad |X_t| \leq 2 \|\g_t\| R  \| \bnu_t -  \tilde\bnu_t\|\leq 2 R B \cdot( \| \bnu_t\|+ \| \tilde \bnu_t\|)\leq 4 d R B \cdot (\kappa + \delta),\label{eq:firstlast}
	\end{align}
	where the last inequality follows by the fact that $\bnu_t$ [resp.~$\tilde \bnu_t$] has the same distribution as $\mathbb{I}_{\{\hat \gamma_t \geq 1 \}} \hat \s_t$ [resp.~$\mathbb{I}_{\{\tilde \gamma_t \geq 1\}} \tilde \s_t$], and $\|\hat \s_t \cdot \mathbb{I}_{\{\hat \gamma_t \geq 1\}}\|\leq d \cdot(\delta/R+1/r)$ by Lemma \ref{lem:celebext} [resp.~$\|\tilde \s_t \cdot \mathbb{I}_{\{\tilde \gamma_t \geq 1\}}\|\leq d \cdot(\delta/R+1/r)$ by Proposition \ref{prop:subgradient}]. We now show that $\E_{t-1}[X_t]=0$. Using the definition of $X_t$, we have 
	\begin{align}
		\E_{t-1}[X_t] & = \E_{t-1}[  \mathbb{I}_{\{ \tilde  \gamma_t <1 \}} \cdot  \mathbb{I}_{\inner{\g_t}{\w_t}<0}\inner{\g_t}{\x_t} \cdot  \inner{ \bnu_t-  \tilde \bnu_t}{\w_t - \x}  ] \nn \\ & \quad +  \E_{t-1}[ \mathbb{I}_{\{ \tilde \gamma_t \geq 1  \}} \cdot  \mathbb{I}_{\inner{\g_t}{\w_t}<0}\inner{\g_t}{\x_t} \cdot  \inner{ \bnu_t- \tilde\bnu_t}{\w_t - \x}],\nn \\
		& = \E_{t-1}[ \mathbb{I}_{\{\hat \gamma_t\geq 1 \}} \cdot  \mathbb{I}_{\inner{\g_t}{\w_t}<0}\inner{\g_t}{\x_t} \cdot  \inner{ \hat \s_t-  \tilde \s_t}{\w_t - \x} ] =0, \label{eq:freed}
	\end{align}
	where the equalities in \eqref{eq:freed} follow by the facts that $\tilde \gamma_t = \hat \gamma_t = \gamma_t$ (by Lemma~\ref{lem:celebext}); $\bnu_t = \hat \s_t=\bm{0}$ if $\gamma_t \geq 1$; and that $\E_{t-1}[ \mathbb{I}_{\{\tilde \gamma_t\geq 1 \}} \cdot ( \hat \s_t-  \tilde \s_t) ]=0$ (by Lemma \ref{lem:celebext}). Using that $\x_t$ is a deterministic function of the past, we get
	\begin{align}
		\E_{t-1}[X_t^2 ]& \leq   \|\x_t\|^2 \|\g_t\|^2   \E_{t-1}[\| \bnu_t-  \tilde \bnu_t\|^2\cdot \|\w_t - \x\|^2], \nn \\ & 
		\leq 8 R^4 \|\g_t\|^2 \E_{t-1}[ (\| \bnu_t\|^2+  \|\tilde \bnu_t\|^2)], \quad (\x_t\in \K)
		\nn \\ 
		& \leq  8 R^4 \|\g_t\|^2 \E_{t-1}[ \mathbb{I}_{\{\hat \gamma_t \geq 1  \}} \cdot  (\| \hat \s_t\|^2+  \|\tilde \s_t\|^2)],\nn \\ &
		\leq   8 R^2 B^2 d (\kappa+\delta)^2. \label{eq:freed2}
	\end{align}
where the last inequality follows by Lemma \ref{lem:celebext} and Proposition \ref{prop:subgradient}. Thus, by \eqref{eq:sixten}, \eqref{eq:firstlast}, \eqref{eq:freed}, \eqref{eq:freed2}, and Freedman's inequality (Theorem \ref{thm:freedman}), we have, for any $\rho \in(0,1)$, with probability at least $1-\rho$, 
	\begin{align}
		\sum_{t=1}^T	\inner{\g_t}{\x_t - \x}& =  \sum_{t=1}^T	\inner{\tilde \g_t}{\w_t - \x} + \sum_{t=1}^T (2\delta_t + \Delta_t) R\|\g_t\| + 4R B (\kappa+\delta) \sqrt{2 d T \ln \rho^{-1}} + 4 d R B \cdot (\kappa+\delta) , \nn \\ & \leq  2 R\sqrt{2\sum_{t=1}^T \|\tilde \g_t\|^2 } + \sum_{t=1}^T (2\delta_t + \Delta_t) R\|\g_t\| + 4R B (\kappa+\delta) \sqrt{2 d T \ln \rho^{-1}} + 4 d RB \cdot (\kappa+\delta), \label{eq:milestone}
	\end{align}
where the last inequality follows by the regret bound \ftrl{} in Proposition \ref{prop:ftrl-proximal}. We now bound the first two terms on the RHS of \eqref{eq:milestone}. We start with $\sum_{t=1}^T \|\tilde \g_t\|^2$. With $Y_t \coloneqq  \|\tilde \g_t\|^2$, we have 
	\begin{align}
		\E_{t-1}[Y_t] & \leq  \E_{t-1}[ 2\|\g_t\|^2 + 2\|\g_t\|^2 \|\x_t\|^2 \|\bnu_t\|^2 ],\nn \\
		& \leq 2B^2 R^2  \E_{t-1}[\mathbb{I}_{\{ \hat \gamma_t <1 \}} \cdot   \|\bnu_t\|^2 ] + 2B^2 R^2  \E_{t-1}[\mathbb{I}_{\{ \hat \gamma_t \geq 1\}} \cdot   \|\hat\s_t\|^2 ] + 2\|\g_t\|^2,\quad (\x_t \in \K)\nn \\
		& \leq  2 d B^2 (\kappa +\delta)^2 + 2 B^2, \label{eq:exp}
	\end{align} 
	where the last inequality follows by Lemma \ref{lem:celebext}. Similarly, we also have 
	\begin{align}
		\E_{t-1}[Y_t^2] & \leq  \E_{t-1}[ 8\|\g_t\|^4 + 8\|\g_t\|^4 \|\x_t\|^4 \|\bnu_t\|^4 ],\nn \\
		& =  8 B^4 R^4 \E_{t-1}[  \|\bnu_t\|^4 ]+ 8 B^4,\nn  \quad  (\x_t \in \K)\\
		& =  8 B^4 R^4 \E_{t-1}[\mathbb{I}_{\{ \hat \gamma_t <1\}} \cdot    \|\bnu_t\|^4 ] +8 B^4 R^4 \E_{t-1}[\mathbb{I}_{\{ \hat \gamma_t \geq 1\}} \cdot   \|\hat \s_t\|^4 ]+ 8 B^4,\nn \\
		& \leq  8 d^3  B^4 (\kappa+\delta)^4 + 8  B^4,\label{eq:mom}
	\end{align} 
	where the last inequality follows by Lemma \ref{lem:celebext}. Also, we have 
	\begin{align}
		|Y_t| \leq \|\g_t\| + \|\g_t\| \|\x_t\| \|\bnu_t\|\leq  B \cdot( 1+ R \| \bnu_t\| )\leq B \cdot (1 +d \cdot(\kappa +\delta)),  \label{eq:up}
	\end{align} 
	where the last inequality follows by the fact that $\bnu_t$ has the same distribution as $\mathbb{I}_{\{\hat \gamma_t \geq 1\}} \hat \s_t$, and $\|\hat \s_t\|\leq d \cdot(\delta/R+1/r)$ by Lemma \ref{lem:celebext}. By combining \eqref{eq:exp}, \eqref{eq:mom}, and \eqref{eq:up}, and applying Freedman's inequality (Theorem \ref{thm:freedman}), we get for all $\rho>0$, with probability at least $1-\rho$,
	\begin{align}
		\sum_{t=1}^T \|\tilde \g_t\|^2& \leq\sum_{t=1}^T \E_{t-1}[\|\tilde \g_t\|^2] + 4 \sqrt{ 2B^4( 1 + d^{3} (\kappa+\delta)^4 )T \ln(1/\rho)} + B^2 (1+d (\kappa+\delta))^2 \ln (1/\delta )  ,\nn \\ & \leq    4 B^2 d (\kappa+\delta)^2 T + 8 d^{3/2}B^2 (\kappa+\delta)^2\sqrt{ T \ln (1/\rho)}+ B^2 (1+d (\kappa+\delta))^2 \ln (1/\delta ), \nn \\
		& \leq  12 B^2 d (\kappa+\delta)^2 T+ 4B^2 d^2 (\kappa+\delta)^2 \ln (1/\delta ), \label{eq:secondprobevent}
	\end{align}
	where the last inequality follows by the fact that $d\ln(1/\rho)\leq T$.
	Finally, by Markov's inequality and the fact that $\ell_t$ is $B$-Lipschitz, we have, for all $\rho$,
	\begin{align}
		\P\left[\sum_{t=1}^T \Delta_t \|\g_t\| \geq   2 B \delta \sqrt{T} \rho^{-1} \right] \leq \frac{\rho \sum_{t=1}^T\E[ \Delta_t \|\g_t\|]}{2 \delta B\sqrt{T}} \leq \frac{\rho \sum_{t=1}^T\E[ \Delta_t ]}{2\delta \sqrt{T}} \leq \frac{\rho \sum_{t=1}^T\delta_t}{2\delta } \leq \rho, \label{eq:markov}
	\end{align}
	where the last inequality follows the fact that $\delta_t = \delta/t^{-1/2}$ and $\sum_{t=1}^T 1/t^{-1/2}\leq 2\sqrt{T}$, for all $T\geq 1$. By combining \eqref{eq:milestone}, \eqref{eq:secondprobevent}, and \eqref{eq:markov}, and a union bound, we get, with probability at least $1-3 \rho$,
	\begin{align}
		\sum_{t=1}^T	\inner{\g_t}{\x_t - \x} & \leq   4 R B(\kappa+\delta)\sqrt{6 d T +  2 d  \ln \rho^{-1} } +6  R B\delta\sqrt{T} /\rho  + 4R B (\kappa+\delta) \sqrt{2 d T \ln \rho^{-1}} + 4 d R B \cdot (\kappa+\delta),\nn \\
		& \leq 8 R B(\kappa+\delta)\sqrt{ d  T (3+\ln \rho^{-1})  + d \ln \rho^{-1}}    + 2 d B R \cdot (2\kappa+\delta \cdot(2 +3\sqrt{T}/\rho)),\nn
	\end{align}
where the last inequality follows by the fact that $\sqrt{a}+\sqrt{b}\leq \sqrt{2 a + 2 b}$. This completes the proof.
\end{proof}

	\subsection{Proof of Theorem \ref{thm:mainstrong} (The Strongly Convex Case)}
	\label{sec:mainstrongproof}
	Before proving Theorem \ref{thm:mainstrong}, we present a result that may be viewed as an extension of Lemma \ref{lem:mainreduction}:
	\begin{lemma}
		\label{lem:tool}
		Let $\w_t, \tilde{\g}_t$, and $\x_t$ be as in Algorithm \ref{alg:projectionfreewrapper} with $\cO_{\K^\circ}\equiv \O_{\K^\circ}$, any tolerance sequence $(\delta_s)\subset (0,1/3)$, and any OCO subroutine $\A$ defined on $\cB(R)$ ($R$ as in \eqref{eq:lowradius}). Then, for all $t\geq 1$ and all $\x\in\K$, we have 
		\begin{align}
			\inner{\tilde \g_t}{\x_t}\leq \inner{\tilde\g_t}{\w_t} + (2 \delta_t + \Delta_t) R \|\g_t\|,\nn
		\end{align}
		where $\Delta_t\in[0,{15^2 d^{4} \kappa^3}{\delta^{-2}_t}]$ is the random variable in Lemma \ref{lem:mainreduction}, which satisfies $\E_{t-1}[\Delta_t]\leq \delta_t$.
	\end{lemma}
	\begin{proof}[{\bf Proof}]
		Instantiating \eqref{eq:individual} in Lemma \ref{lem:mainreduction} with $\x=\x_t$, we get 
		\begin{align}
			\inner{\tilde \g_t}{\w_t}\geq \tilde\ell_t(\w_t) - \tilde \ell_t(\x_t) + \inner{\tilde \g_t}{\x_t}- (\delta_t+\Delta_t) R\|\g_t\|.\label{eq:keypre}
		\end{align}
		Now, to get the desired result, we need to show that $\tilde\ell_t(\w_t) - \tilde \ell_t(\x_t)\geq - \delta_t R \|\g_t\|$. When $\gamma_t<1$, we have $\x_t = \w_t$ and so \begin{align}\tilde \ell_t(\x_t)=\tilde \ell_t(\w_t). \quad \text{[case where $\gamma_t<1$]} \label{eq:case1}
			\end{align} 
		Now, suppose that $\gamma_t \geq 1$ and $\inner{\g_t}{\w_t}\geq 0$. By the fact that $\x_t = \w_t/\gamma_t$, we get
		\begin{align}
			\tilde \ell_t(\x_t) =	\inner{\g_t}{\x_t} \leq \inner{\g_t}{\w_t} = \tilde \ell_t(\w_t).\quad \text{[case where $\gamma_t\geq 1, \inner{\g_t}{\w_t}\geq 0$]}  \label{eq:firstsecond2}
		\end{align}
		Now suppose that $\gamma_t \geq 1$, $\inner{\g_t}{\w_t}<0$, and $\w_t\in \K$. We note that this implies that $\gamma_{\K}(\w_t)\leq1$ and so $S_{\K}(\w_t)=0$ (Lemma \ref{lem:projectiononK}). Thus, $\tilde \ell_{t}(\w_t)= \inner{\g_t}{\w_t}$. On the other hand, by Lemma \ref{lem:gauge} we have $\gamma_t \leq \gamma_{\K}(\w_t)+ \delta_t\leq 1 + \delta_t$, and so since $\x_t= \w_t/\gamma_t$, we have 
	\begin{align}
		\inner{\g_t}{\x_t} \leq \inner{\g_t}{\w_t}/(1+\delta_t) = \inner{\g_t}{\w_t} - \delta_t \inner{\g_t}{\w_t}/(1+\delta_t). \nn 
	\end{align}
	Thus, since $\tilde \ell_t(\w_t)=\inner{\g_t}{\w_t}$ and $\tilde \ell_t(\x_t) = \inner{\g_t}{\x_t}$ (because $\x_t\in \K$), the previous display implies that 
	\begin{align}
	\tilde \ell_t(\x_t)	=\inner{\g_t}{\x_t} \leq \tilde{\ell}_t(\w_t)+ \delta_t R \|\g_t\|. \quad \text{[\text{case where $\gamma \geq 1, \inner{\g_t}{\w_t}<0, \w_t \in \K$}]} \label{eq:case22}
	\end{align}
	We now consider the last case where $\gamma_t \geq 1$, $\inner{\g_t}{\w_t}<0$, and $\w_t\not\in \K$. We note that this implies that $\gamma_{\K}(\w_t)\geq 1$. By the fact that $\gamma_t \leq \gamma_{\K}(\w_t) +\delta_t$ (Lemma \ref{lem:gauge}), we have 
	\begin{align}
		\inner{\g_t}{\x_t} \leq  \frac{  \inner{\g_t}{\w_t}}{\gamma_{\K}(\w_t) + \delta_t} =  \frac{ \inner{\g_t}{\w_t/\gamma_{\K}(\w_t)}}{1 + \delta_t/\gamma_{\K}(\w_t)}&= \inner{\g_t}{\w_t/\gamma_{\K}(\w_t)} -\frac{\delta_t/\gamma_{\K}(\w_t)}{1+\delta_t/\gamma_{\K}(\w_t)} \inner{\g_t}{\w_t/\gamma_{\K}(\w_t)},\label{eq:seec22} 
	\end{align}
	Thus, since $S_{\K}(\w_t)=\gamma_{\K}(\w_t)-1$ (by Lemma \ref{lem:projectiononK} and $\gamma_{\K}(\w_t)\geq 1$), we get 
	\begin{align}
		\inner{\g_t}{\x_t}\cdot S_{\K}(\w_t)  \leq  \inner{\g_t}{\w_t}- \frac{\delta_t \inner{\g_t}{\w_t} }{1+\delta_t/\gamma_{\K}(\w_t)} - \inner{\g_t}{\w_t/\gamma_{\K}(\w_t)}  +  \frac{\delta_t/\gamma_{\K}(\w_t)}{1+\delta_t/\gamma_{\K}(\w_t)} \inner{\g_t}{\w_t/\gamma_{\K}(\w_t)}. \nn
	\end{align}
	Adding this together with \eqref{eq:seec22} and using the fact that $\|\w_t\|\leq R$, we get 
	\begin{gather}	
		\inner{\g_t}{\x_t} + \inner{\g_t}{\x_t} S_{\K}(\w_t)\leq \inner{\g_t}{\w_t} +\delta_t R \|\g_t\|,\nn
		\intertext{and so after rearranging and using that $\tilde\ell_t(\x_t)=\inner{\x_t}{\g_t}$ (since $\x_t\in\K$) and $\tilde{\ell}_t(\w_t)= \inner{\g_t}{\w_t} -  \inner{\g_t}{\x_t} S_{\K}(\w_t)$, we get}
		\inner{\g_t}{\x_t}  \leq  \tilde{\ell}_t(\w_t)+ \delta_t R \|\g_t\|. \quad \text{[\text{case where $\gamma \geq 1, \inner{\g_t}{\w_t}<0, \w_t \not\in \K$}]}\label{eq:second22}
	\end{gather}
	By combining, \eqref{eq:case1}, \eqref{eq:firstsecond2}, \eqref{eq:case22}, and \eqref{eq:second22}, we obtain:
	\begin{align}
	\tilde\ell_t(\x_t)  \leq \tilde{\ell}_t(\w_t)+ \delta_t R \|\g_t\|.\nn
	\end{align}
		Combining this with \eqref{eq:keypre}, we get 
		\begin{align}
			\inner{\tilde \g_t}{\x_t}\leq \inner{\tilde \g_t}{\w_t} + (2\delta_t + \Delta_t) R \|\g_t\|.\nn
		\end{align}
	\end{proof}

\noindent	We now present the versions of Lemmas \ref{lem:mainreduction} and \ref{lem:tool} when $\cO_{\K^\circ} \equiv \onedopt$:
	\begin{lemma}
		\label{lem:weso}
		Let $\w_t, \tilde{\g}_t$, and $\x_t$ be as in Algorithm \ref{alg:projectionfreewrapper} with $\cO_{\K^\circ}\equiv \onedopt$; any tolerance sequence $(\delta_s)\subset (0,1/3)$; any tolerance sequence $(\delta_s)\subset (0,1/3)$; and any OCO subroutine $\A$ defined on $\cB(R)$ ($R$ as in \eqref{eq:lowradius}). Then, for all $t\geq 1$, $\x_t\in \K$ and there exists a family of random variables $\{ \xi_t(\x) : \x\in \K\}$ such that for all $\x\in \K$, $\E_{t-1}[\xi_t(\x)]=0$,  $|\xi_t(\x)|<4d(\kappa+\delta_t)R\|\g_t\|$ almost surely, and
		\begin{align}
	 \inner{\g_t}{\x_t-\x}  &\leq \inner{\tilde\g_t}{\w_t-\x} + (2 \delta_t + \Delta_t) R \|\g_t\| +  \xi_t(\x);\label{eq:good}\\
	 	\inner{\tilde \g_t}{\x_t} & \leq \inner{\tilde\g_t}{\w_t} + (2 \delta_t + \Delta_t) R \|\g_t\| +  \xi_t(\x_t), \label{eq:badgood} 
				\end{align}
 where $\Delta_t\in [0,{15^2d^{4} \kappa^3}{\delta^{-2}_t}]$ is a random variable satisfying $\E_{t-1}[\Delta_t]\leq \delta_t$. Furthermore, $\|\tilde \g_t\| \leq (1 +d\kappa +d\Delta_t)\|\g_t\|$ and $\E_{t-1}[\|\tilde \g_t\|^2]\leq 2(1+d (\kappa+\delta_t)^2)\|\g_t\|^2$. 
	\end{lemma}
	\begin{proof}[{\bf Proof}]
		Let $t\geq 1$; $\x\in \K$; $(\tilde \gamma_t, \tilde \s_t)= \O_{\K^\circ}(\w_t;\delta_t)$; and $(\hat \gamma_t, \hat \s_t)= \onedopt(\w_t;\delta_t)$. Further, let $\tilde  \bnu_t\coloneqq \mathbb{I}_{\{\tilde \gamma_t\geq 1  \}} \tilde \s_t$ and $\xi_t(\x) \coloneqq \mathbb{I}_{\inner{\g_t}{\w_t}<0}\inner{\g_t}{\x_t}  \inner{ \bnu_t-  \tilde \bnu_t}{\w_t - \x}$. By Lemma \ref{lem:mainreduction}, there exists a random variable $\Delta_t\geq 0$ satisfying $\E_{t-1}[\Delta_t]\leq \delta_t$ and such that for all $t\geq 1$ and $\x\in \K$, 
		\begin{align}
			\inner{\g_t}{\x_t - \x} & \leq \inner{\g_t - \mathbb{I}_{\inner{\g_t}{\w_t}<0} \inner{\g_t}{\x_t} \tilde \bnu_t}{\w_t-\x} + (2\delta_t +\Delta_t)R \|\g_t\|,\nn \\
			& \leq  \inner{\tilde \g_t}{\w_t - \x} +(2\delta_t +\Delta_t)R \|\g_t\|+\xi_t(\x).
			\nn
		\end{align}
	By Proposition \ref{prop:subgradient} and Lemma \ref{lem:celebext}, we have $\|\tilde \s_t\|\vee \|\hat \s_t\|<d\cdot (1/R+\delta/r)$ whenever $\hat\gamma_t =\tilde \gamma_t \geq 1$. This implies that $\|\tilde \bnu_t\|\vee \|\hat \bnu_t\|<d\cdot (1/R+\delta/r)$, and so $|\xi_t(\x)|<4 d(\kappa+\delta) R \|\g_t\|$. We now show that the conditional expectation of $\xi_t(\x)$ is zero. By Lemma \ref{lem:celeb}, and the law of total expectation, we have 
		\begin{align}
			\E_{t-1}[\xi_t(\x)]&  \leq   \E_{t-1}[\mathbb{I}_{\inner{\g_t}{\w_t}<0}\inner{\g_t}{\x_t}  \inner{ \bnu_t-  \tilde \bnu_t}{\w_t - \x}],  \nn \\ 
			&  =   \mathbb{I}_{\{\tilde \gamma_t \geq 1\}} \cdot    \mathbb{I}_{\inner{\g_t}{\w_t}<0}\inner{\g_t}{\x_t}  \inner{\E_{t-1}[ \bnu_t-  \tilde \bnu_t]}{\w_t - \x},\label{eq:penul} \\ 
			&  =  \mathbb{I}_{\{\tilde \gamma_t \geq 1\}} \cdot    \mathbb{I}_{\inner{\g_t}{\w_t}<0}\inner{\g_t}{\x_t}  \inner{\E_{t-1}[ \s_t-  \tilde \s_t]}{\w_t - \x}, \nn\\ 
			& =  0. \label{eq:noone10}
		\end{align}
		where \eqref{eq:penul} follows by the fact that $\w_t,\gamma_t,\tilde \gamma_t$, and $\g_t$ are deterministic function of the past, and \eqref{eq:noone10} follows by the fact $\tilde \gamma_t = \hat \gamma_t$; and $\E_{t-1}[\hat \s_t]= \E_{t-1}[\tilde  \s_t]$, when $\tilde \gamma_t \geq 1$ (Lemma \ref{lem:celeb}). This shows \eqref{eq:good}. We follow similar steps to show \eqref{eq:badgood}, but we use the result of Lemma \ref{lem:tool} instead of Lemma \ref{lem:mainreduction}.  By Lemma \ref{lem:tool}, we have 
		\begin{gather}
		 \inner{\g_t - \mathbb{I}_{\inner{\g_t}{\w_t}<0} \inner{\g_t}{\x_t} \tilde \bnu_t}{\x_t}  \leq \inner{\g_t - \mathbb{I}_{\inner{\g_t}{\w_t}<0} \inner{\g_t}{\x_t} \tilde \bnu_t}{\w_t} + (2\delta_t +\Delta_t)R \|\g_t\|,\nn
		 \shortintertext{and so by rearranging, we get}
		 	\inner{\tilde \g_t}{\x_t}\leq \inner{\tilde \g_t}{\w_t} + (2\delta_t + \Delta_t) R \|\g_t\| + \xi_t(\x_t).\nn
			\end{gather}		
		This shows \eqref{eq:badgood}. We now bound $\|\tilde \g_t\|$. When $\tilde \gamma_t <1$, we have $\g_t=\tilde \g_t$ and so $\|\tilde \g_t\| \leq (1 + d\kappa + d\Delta_t) \|\g_t\|$ holds trivially. Now suppose that $\tilde \gamma_t \geq 1$. In this case, $\bnu_t$ has the same distribution as $\mathbb{I}_{\{\hat \gamma_t \geq 1\}}\hat \s_t$, and so
		\begin{align}
			\|\tilde{\g}_t\|  =  \|\inner{\g_t}{\w} - \mathbb{I}_{\inner{\g_t}{\w_t} <0}  \inner{\g_t}{{\x}_t}  \bnu_t\|= 	\|\g_t\| + \|\g_t\| \frac{\| \w_t\|}{\gamma_t}  \| \bnu_t\| \leq \|\g_t\|(1 + R \| \bnu_t\|)  
			\leq  (1+ d\Delta_t +d\kappa )\|\g_t\|, \nn
		\end{align}
		where the last inequality follows from the fact that $\mathbb{I}_{\{\hat \gamma_t \geq 1\}}\|\hat \s_t\|\leq d\cdot (\Delta_t/R+1/r)$, by Lemma \ref{lem:celebext}. From the penultimate inequality in the above display and Lemma \ref{lem:celeb}, we also have 
		\begin{align}
			\E_{t-1}[	\|\tilde{\g}_t\|^2 ]\leq 2\|\g_t\|^2 (1 + R^2 \E_{t-1}[\|\hat{\s}_t\|^2] ) \leq 2 \|\g_t\|^2 (1+ d\cdot  (\kappa+\delta_t)^2).\nn
		\end{align}
	Finally, the fact that $\x_t \subset \K$ follows from Theorem \ref{thm:genthmeff}. In fact, looking at the proof of Theorem \ref{thm:genthmeff} it is clear that the fact that $(\x_t) \subset \K$ only uses the fact that the subroutine $\A$ outputs iterates in $(\x_t)$ in $\cB(R)$. 
	\end{proof}
	
\noindent We now present the proof of Theorem \ref{thm:mainstrong}. We note that the proof follows from that of \cite[Theorem 7]{cutkosky2018black} with small modifications to account for the differences between their constrained-to-unconstrained reduction and our projection-free reduction (Algorithm \ref{alg:projectionfreewrapper}), as well as the differences described in Section \ref{sec:strong} to achieve scale-invariance.
\begin{proof}[{\bf Proof of Theorem \ref{thm:mainstrong}}]
	First of all, note that $(\x_t)\subset \K$ follows directly from Lemma \ref{lem:weso}. Next, we will instantiate Theorem \ref{thm:genthmstrong} below with $\delta_t = \delta/t^2$ for all $t\geq1$ and some $\delta \in(0,1/3)$. By Lemma \ref{lem:weso} and the fact that the losses are $B$-Lipschitz, the variables $\zeta_T, Z_T, M_T,$ and $N_T$ in Theorem \ref{thm:genthmstrong} below satisfy
	\begin{itemize}
		\item $\E[\zeta_T] \leq 1 + d\kappa + 2d\delta_T$, where we used the fact that $\E_{t-1}[\Delta_t]\leq \delta_t$ and $\sum_{t=1}^{\infty}1/t^2\leq 2$.
		\item $\zeta_T \leq 16^2 d^{5}\kappa^3/\delta^2_T$, where we used the fact that $\Delta_t\leq 15^2 d^{4}\kappa^3/\delta^2_t$, $\forall t$.
		\item $Z_T\leq \epsilon^2 +B^2 (16^4 d^{10}\kappa^6/\delta^4_T) (1+T) \leq B^2 (16^4 d^{10}\kappa^6/\delta^4_T) (2+T)$.
		\item $M_T \leq 2 \sqrt{\ln_+ \left( { 2 (16^6 d^{15}\kappa^9/\delta^6_T (2T+1)^2  R B^2 }/{\epsilon^2}\right)}.$
		\item $\E[N_T] \leq \E[4 (1 + d\kappa + 2d\delta) B  \ln_+ \left( { 4 (16^6 d^{15}\kappa^9/\delta^6_T )(2T+1)^2  R B^2}/{\epsilon^2}\right) +   \epsilon/R +  (1+2 d\kappa+ 8d\delta) B].$ This follows by the facts that $\E[\zeta_T] \leq 1 + d\kappa + 2d\delta$ and $\zeta_T \leq 16^2 d^{5}\kappa^3/\delta^2_T$.
		\item $\E_{t-1}[\|\tilde \g_t \|^2]\leq 2 (1+ d (\kappa+\delta_t)^2)\|\g_t\|^2$ by Lemma \ref{lem:weso}.
	\end{itemize}
	Thus, for $\nu \coloneqq  1/(R\wedge 1) + \kappa + 2\delta$, there exists $W_T = O\left( \ln (e+\kappa d R  T B/(\delta \epsilon)) \right)$ such that 
	\begin{align}
	\E\left[	\sum_{t=1}^T \inner{\g_t}{\x_t- \x} \right]&\leq W_T \sqrt{\E\left[\sum_{t=1}^T  \|\tilde \g_t\|^2 \|\x_t-\x\|^2\right]} +d\nu  R B  W_T^2, \nn \\
		& =W_T\sqrt{\E\left[\sum_{t=1}^T \E_{t-1} [\|\tilde \g_t\|^2 \|\x_t-\x\|^2]\right]} +d\nu  R B  W_T^2,\nn \\
		& \leq  W_T \sqrt{\E\left[2 d\nu^2\sum_{t=1}^T \|\g_t\|^2 \|\x_t-\x\|^2\right]} +d\nu  R B  W_T^2,\nn \\ 
			& \leq  U_T \sqrt{\E\left[\sum_{t=1}^T \|\g_t\|^2 \|\x_t-\x\|^2\right]} +  R B  U_T^2/\nu,\label{eq:regret}
	\end{align}
where $U_T = O\left( \nu d^{1/2}\ln (e+\kappa d R  T B/(\delta \epsilon)) \right)$. Thus, for general convex functions, we have, in expectation
	\begin{align}
		\sum_{t=1}^T (\ell_t(\x_t) -\ell_t(\x)) \leq U_T R B \sqrt{T} +  R B U_T^2/\nu.\nn
	\end{align}
	For $\mu$-strongly convex functions $(\ell_t)$, we have
	\begin{align}
	\E\left[\sum_{t=1}^T (\ell_t(\x_t)-\ell_t(\x))\right] &\leq \E\left[\sum_{t=1}^T \inner{\g_t}{\x_t - \x}\right] - \E\left[ \frac{\mu}{2}\sum_{t=1}^T \|\x_t - \x   \|^2\right], \nn \\
		& \leq  U_T \sqrt{\E\left[\sum_{t=1}^T \|\g_t\|^2 \|\x_t-\x\|^2\right]} +  R B  U_T^2/\nu -\E\left[\frac{\mu}{2} \sum_{t=1}^T \|\x_t - \x \|^2\right], \label{eq:134} \\
		& \leq \inf_{\eta>0}  \left\{ \E\left[ \sum_{t=1}^T \|\g_t\|^2 \|\x_t - \x\|^2/\eta\right]    + \eta U_T^2/4 \right\}  - \E\left[\mu \sum_{t=1}^T \|\x_t - \x   \|^2/2\right]+\nu  R B  U_T^2/\nu,\nn \\
		& \leq  \frac{ \mu}{2} \E\left[\sum_{t=1}^T \frac{\|\g_t\|^2}{B^2} \|\x_t - \x\|^2\right]  + \frac{B^2 U_T^2}{2 \mu} -  \E\left[\frac{ \mu}{2} \sum_{t=1}^T \|\x_t - \x   \|^2\right]+  R B  U_T^2/\nu, \label{eq:setup}\\
		& \leq  B^2 U_T^2/(2 \mu)+   R B  U_T^2/\nu,\nn
	\end{align}
where \eqref{eq:134} follows by \eqref{eq:regret} and \eqref{eq:setup} follows by setting $\eta = 2 B^2/\mu$.
\end{proof}
	\begin{theorem}
		\label{thm:genthmstrong}
		Let $\delta \in (0,1/3)$ and $\kappa\coloneqq R/r$, with $r$ and $R$ as in \eqref{eq:lowradius}. Suppose that Algorithm \ref{alg:projectionfreewrapper} is run with $\cO_{\K^\circ}\equiv \onedopt$; $\delta_t = \delta/t^2$, $\forall t \geq 1$; and sub-routine $\A$ set to Alg.~\ref{alg:projectionfreewrapperstong} with parameter $\epsilon>0$. Then, for any adversarial sequence of convex losses $(\ell_t)$ on $\K$, the iterates $(\x_t)$ of Alg.~\ref{alg:projectionfreewrapper} in response to $(\ell_t)$ satisfy,
		\begin{align}
		\forall T\geq 1,\forall \x\in \K,\ \	\sum_{t=1}^T \inner{\g_t}{\x_t- \x} &\leq M_T \sqrt{2   \left[ (\epsilon^2+  \zeta_T B_T^2) \|\x\|^2 +  \sum_{t=1}^T \|\tilde \g_t \|^2 \|\x_t-\x\|^2\right] \cdot \ln \frac{Z_T}{ \epsilon^2}} \nn \\
			& \qquad +2R N_T \cdot\left(1+ \ln \frac{Z_T}{\epsilon^2}\right) + \Xi_T(\x), \nn
		\end{align}
	where $\g_t\in \partial \ell_t(\x_t)$, $\forall t$, $B_T \coloneqq \epsilon \vee \max_{t\in[T]} \|\g_t\|$, and $\zeta_T, Z_T, M_T, \Xi_T(\x)$ and $N_T$ are such that:
		\begin{itemize}
			\item $\zeta_T \coloneqq 1 +d\kappa +d\max_{t\in[T]} \Delta_t$, with $\Delta_t\geq 0$ a non-negative random variable satisfying $\E_{t-1}[\Delta_t]\leq \delta_t$.
			\item $Z_T\leq  \epsilon^2 + \tilde B_T^2 + \sum_{t=1}^T \|\tilde {\g}_t\|^2\leq   \epsilon^2+  \zeta_T^2 B^2_T +\zeta_T^2 \sum_{t=1}^T \|\g_t\|^2$.
			\item $M_T \coloneqq 2  \sqrt{\ln_+\left({2 \zeta_t^2(R + 2\zeta_t R T) Q_T }/{\epsilon^2}\right)}$, and $Q_t \coloneqq \epsilon^2+ \sum_{i=1}^t \|\g_i\|^2$.
			\item $\Xi_T(\x)\coloneqq \sum_{t=1}^T \xi_t(\x)$, where $\{ \xi_t(\x) : \x\in \K\}$ is a family of random variables satisfying for all $\x\in \K$, $\E_{t-1}[\xi_t(\x)]=0$. 
			\item $N_T \coloneqq 4 \zeta_t B_T \ln_+ \left( {4 \zeta_t^2 B_T(R + 2 \zeta_t R  T)\sqrt{Q_T} }/{\epsilon^2} \right) +B_T +\frac{\epsilon}{R} + \sum_{t=1}^T (2\delta_t +\Delta_t) \|\g_t\| +2d(\kappa+\delta)B_T.$ 
		\end{itemize}
	\end{theorem}

	\begin{proof}[{\bf Proof}]
		Define $\zeta_t \coloneqq 1+d\kappa+d\max_{s\leq t }\Delta_s$, where $\kappa=R/r$ and $\Delta_t$ is the same random variable as in Lemma~\ref{lem:weso}. Note that $\zeta_t \geq 0$ and $\E[\zeta_t]\leq 1 +d\kappa+d\sum_{s=1}^t\delta_s$. For any $t\geq 1$, consider the random vector $\X_t$ that takes value $ \x_{i}$ for $s \le t$ with probability proportional to $\|\hat{\g}_{i}\|  ^2$, and value $\bm{0}$ with probability proportional to $\epsilon^2 +\tilde B_t^2$, where $\tilde B_t \coloneqq \epsilon \vee \max_{i\in[t]}\|\hat \g_i\|$ and $\hat \g_t\coloneq \tilde \g_t \cdot \tilde B_{t-1}/\tilde B_t$ (see Algorithm \ref{alg:projectionfreewrapperstong}). Moving forward, we define 
		\begin{align}
			z_t \coloneqq \|\hat{\g}_{t}\|^2,\quad \text{and} \quad  z_{0,t}\coloneqq \epsilon^2+\tilde B_t^2, \nn
		\end{align}
		so that $Z_t \coloneqq  z_{0,t}+\sum_{i=1}^t z_i= \epsilon^2 + \tilde B_t^2 + \sum_{i=1}^t \|\hat{\g}_i\|^2$. We make the following definitions/observations:
		\begin{itemize}
			\item We define
			$V_T(\x)
			=z_{0,T} \|\x\|^2 +  \sum_{t=1}^T z_t \| \x_t-\x\|^2
			= Z_T \cdot \E[\|\X_T-\x\|^2]~.$
			\item $Z_t \leq  \zeta_T^2 B^2_T\cdot (T+1) + \epsilon^2$, for all $t\in[T]$, which follows from Lemma \ref{lem:weso}.
			\item $\overline{ \x}_t = \E[\X_t] =   Z_t^{-1}\cdot \sum_{i=1}^t z_i \cdot \x_i = \v_t/Z_t$, where $(\v_i)$ are as in Algorithm \ref{alg:projectionfreewrapperstong}.
			\item $\sigma^2_t := (  z_{0,t}\|\overline{ \x}_t\|^2+\sum_{i=1}^t z_t\cdot   \|  \x_{i}-\overline{ \x}_t\|^2)/Z_t$ so that $\sigma^2_t = \E[\|\X_t-\overline{ \x}_t\|^2]$.
		\end{itemize}
		To prove the theorem, we are going to show for any $\x\in\K$,
		\begin{equation}
			\label{eqn:biasvariancetarget}
			\sum_{t=1}^T \inner{\g_t}{\x_t- \x}
			\le M_T \sqrt{Z_T \cdot \|\x- \x_T\|^2} +  M_T   \sqrt{\sigma_T^2 \cdot Z_T\cdot \ln \frac{Z_T}{ \epsilon^2}}+2R N_T \cdot\left(1+ \ln \frac{Z_T}{\epsilon^2}\right) + \Xi_T(\x),
		\end{equation}
		where $M_T$ and $N_T$ are as in \eqref{eq:M} and \eqref{eq:N} below, respectively, which implies the desired bound by a bias-variance decomposition: $$Z_T\cdot\|\x-\overline{ \x}_T\|^2 + Z_T\cdot \sigma^2_T = Z_T\cdot\E[\|\X_T-\x\|^2]=  V_T(\x).$$
		In particular, combining this with \eqref{eqn:biasvariancetarget} and using the fact that $\sqrt{x}+\sqrt{y} \leq \sqrt{2 (x+y)}$ for all $x,y>0$, we get 
		\begin{align}
			\sum_{t=1}^T \inner{\g_t}{\x_t- \x}
			\le  M_T \sqrt{2 V_T(\x) \cdot \|\x- \x_T\|^2 \cdot \ln \frac{Z_T}{ \epsilon^2}} +2R N_T \cdot\left(1+ \ln \frac{Z_T}{\epsilon^2}\right)+ \Xi_T(\x) .\nn
		\end{align}
		To get to \eqref{eqn:biasvariancetarget}, first observe that for all $\x \in \K$, the linearized regret of Algorithm \ref{alg:projectionfreewrapperstong} satisfies
		\begin{align}
			\sum_{t=1}^T \inner{\g_t}{\x_t- \x}- \sum_{t=1}^T	\inner{\bar{\g}_t}{\x_t - \x} =	  \sum_{t=1}^T	\inner{{\g}_t -\bar{\g}_t}{\x_t - \x} 
			\stackrel{(*)}{\leq} 2 R\sum_{t=1}^T  (B_t- B_{t-1})\leq 2 R B_T, \nn
		\end{align}
		where $(*)$ following by Cauchy Schwarz.
		Using this together with \eqref{eq:good} (multiplied by $\tilde{B}_{t-1}/\tilde B_t$) and the clipped regret of \freegrad{} in Proposition \ref{prop:freegrad}, we have, for all $\x\in \K$,
		\begin{align}
			\sum_{t=1}^T \inner{\g_t}{\x_t- \x} - 2 R B_T&  \leq  	\sum_{t=1}^T \langle  \bar\g_t,  \x_t-\x\rangle, \nn \\
			&\leq \sum_{t=1}^T \langle \hat{\g}_t, \w_t-\x\rangle +\sum_{t=1}^T (2\delta_t + \Delta_t) R \|\g_t\| + \sum_{t=1}^T \xi_t(\x),\nn \\
			&=\sum_{t=1}^T \langle \hat{\g}_t, \bu_t - (\x-\overline{ \x}_T)\rangle + \sum_{t=1}^T \langle \hat{\g}_t, \overline{ \x}_{t-1} - \overline{\x}_T\rangle + \sum_{t=1}^T(\xi_t(\x)+(2\delta_t + \Delta_t) R \|\g_t\|), \nn \\
			&=M_T \sqrt{Z_T \cdot\|\x-\overline{ \x}_T\|^2}  + \sum_{t=1}^T \langle \hat{\g}_t, \overline{ \x}_{t-1} - \overline{ \x}_T\rangle +  2R N_T + \Xi_T(\x) ,\label{eq:new}
		\end{align}
		where $\Xi_T(\x)\coloneqq \sum_{t=1}^T \xi_t(\x)$. Note that the first term on the RHS of \eqref{eq:new} is exactly what we want, so we only have to derive an upper bound on the second one. This is readily done through Lemma~\ref{lem:shortcut} that immediately gives us the stated result.
	\end{proof}

	\begin{lemma}
		\label{lem:shortcut}
		Under the hypotheses of Theorem~\ref{thm:genthmstrong}, we have 
		\begin{gather}
			\sum_{t=1}^T \langle \hat{\g}_t, \overline{ \x}_{t-1} - \overline{ \x}_T\rangle
			\leq  M_T \sigma_T \sqrt{Z_T\ln \frac{Z_T}{ \epsilon^2}}+2R N_T \cdot \ln \frac{Z_T}{ \epsilon^2}, \ \ \text{where} \nn \\
			M_T \coloneqq 2\sqrt{\ln_+\left(\frac{2 \zeta_t^2(R + 2\zeta_t R T) Q_T }{\epsilon^2}\right)}; \quad Q_T \coloneqq \epsilon^2 + \sum_{t=1}^T \|\g_t\|^2; \quad \text{and} \label{eq:M} \\  N_T \coloneqq 4 \zeta_t B_T \ln_+ \left( \frac{4 \zeta_t^2 B_T(R + 2 \zeta_t R  T)\sqrt{Q_T} }{\epsilon^2} \right) +B_T +\frac{\epsilon}{R} + \sum_{t=1}^T (2\delta_t +\Delta_t) \|\g_t\| + 2 d (\kappa+\delta)B_T.\label{eq:N}
		\end{gather}
	\end{lemma}
	\begin{proof}[{\bf Proof}]
		We have that
		\begin{align*}
			\sum_{i=1}^t \langle \hat{\g}_i, \overline{ \x}_{i-1}-\overline{ \x}_t\rangle - \sum_{i=1}^{t-1} \langle \hat{\g}_i, \overline{ \x}_{i-1}-\overline{ \x}_{t-1}\rangle
			&= \left\langle \sum_{i=1}^t \hat{\g}_{i},\overline{ \x}_{t-1}-\overline{ \x}_t\right\rangle~.
		\end{align*}
		The telescoping sum gives us
		\[
		\sum_{t=1}^T \langle \hat{\g}_t, \overline{ \x}_{t-1}-\overline{ \x}_T\rangle 
		=\sum_{t=1}^{T} \left\langle \sum_{i=1}^t \hat{\g}_i,\overline{ \x}_{t-1}-\overline{ \x}_t\right\rangle
		\leq \sum_{t=1}^{T} \left\| \sum_{i=1}^t \hat{\g}_i\right\|   \|\overline{ \x}_{t-1}-\overline{ \x}_t\|~.
		\]
		So in order to bound $\sum_{t=1}^T \langle \hat{\g}_t,\overline{ \x}_{t-1}-\overline{ \x}_T\rangle$, it suffices to bound $\left\| \sum_{i=1}^t \hat{\g}_i\right\|   \|\overline{ \x}_{t-1}-\overline{ \x}_t\|$ by a sufficiently small value.
		First, we will tackle $\left\|\sum_{i=1}^t \hat{\g}_i\right\|$. By Lemma \ref{lem:weso} (in particular \eqref{eq:badgood}) and the fact that $\hat \g_t = \tilde \g_t \cdot B_{t-1}/B_t$, we have, 
		\begin{align}
			\forall x>0,\quad		\sum_{i=1}^t -2R \|\hat{\g}_i\| +\left\|\sum_{i=1}^t\hat{\g}_i\right\|  x
			&\le \sum_{i=1}^t \langle \hat{\g}_i,\x_i-\overline{ \x}_{i-1}\rangle +\left\|\sum_{i=1}^t\hat{\g}_i\right\|  x , \quad (\text{since } \x_i,\overline \x_{i-1}\in \K) \nn  \\
			&\le \sum_{i=1}^t \langle \hat{\g}_i,\w_i-\overline{ \x}_{i-1}\rangle +\left\|\sum_{i=1}^t\hat{\g}_i\right\|  x + \sum_{i=1}^t\xi_i(\x_i)+ (2\delta_i + \Delta_i) R \|\g_i\|),\nn\\
			&= \sum_{i=1}^t \langle \hat{\g}_i,\bu_i\rangle+\left\|\sum_{i=1}^t\hat{\g}_i\right\|  x+\sum_{i=1}^t (2\delta_i + \Delta_i) R \|\g_i\|+ \sum_{i=1}^t\xi_i(\x_i), \nn\\
			&\le  2x\sqrt{\left( \epsilon^2 + \sum_{i=1}^t \|\hat{\g}_i\| ^2\right)\ln_+\left(\frac{2\zeta_t^2 x Q_t}{\epsilon^2}\right)} + \sum_{i=1}^t\xi_i(\x_i)\nn \\ & \quad + \sum_{i=1}^t (2\delta_i+ \Delta_i) R \|\g_i\| + 4 \zeta'_t x\ln\left(\frac{4 \zeta_t^2 B_T x\sqrt{Q_t}}{\epsilon^2}\right) + \epsilon, \label{eq:freegradreg}
		\end{align}
		where $\bu_i$ is the $i$th output of \freegrad{}, $Q_t \coloneqq \epsilon^2+ \sum_{i=1}^t \|\g_i\|^2\leq \epsilon^2  + t B_t^2$, and $\zeta'_t \coloneqq  \max_{s\in[t]} \{ (1+d\kappa+ d \Delta_s)B_s\}$. The passage to \eqref{eq:freegradreg} follows from the regret bound \freegrad{} (see Proposition \ref{prop:freegrad}) and the fact that $\|\hat \g_t\|\leq \zeta_t  \|\g_t\|$ (see Lemma \ref{lem:weso}). Moving forward, we let $G_t \coloneqq \epsilon +\sum_{i=1}^t (2\delta_i +\Delta_i)R \|\g_i\|$. Dividing by $x$ in \eqref{eq:freegradreg} and solving for $\left\|\sum_{i=1}^t\hat{\g}_i\right\| $, we get
		\begin{align*}
			\left\|\sum_{i=1}^t\hat{\g}_i\right\|  
			&\le 2\sqrt{\left(\epsilon^2 +\sum_{i=1}^t \|\hat{\g}_i\| ^2\right)\ln_+\left(\frac{2\zeta_t^2 x Q_t}{\epsilon^2}\right)} + 4 \zeta'_t  \ln\left(\frac{4 \zeta_t^2 B_t x\sqrt{Q_t}}{\epsilon^2}\right) +  \frac{G_t}{x}  +\frac{2R }{x} \sum_{i=1}^t \|\hat{\g}_i\| +\frac{1}{x}\sum_{i=1}^t \xi_i(\x_i)~.
		\end{align*}
		Set $x=R +2R\sum_{i=1}^t (\|\hat{\g}_i\| \vee \|\g_i\|)/B_T$ and using the facts that $\|\hat \g_t\|\leq \zeta_t \|\g_t\|$ and $|\xi_t(\x)|\leq 4 d (\kappa+\delta_t) R \|\g_t\|$, for all $\x\in \K$ (see Lemma \ref{lem:weso}), we conclude that:
		\begin{gather}
			\left\|\sum_{i=1}^t\hat{\g}_i\right\|  \le M_t\sqrt{ \epsilon^2 +\sum_{i=1}^t\|\hat{\g}_i\| ^2 } +N_t, \quad \text{where}    \nn \\
			M_t \coloneqq 2\sqrt{\ln_+\left(\frac{2 \zeta_t^2(R + 2\zeta_t R t) Q_t }{\epsilon^2}\right)},\nn \\	\text{and} \ \  N_t \coloneqq 4  \zeta'_t \cdot \ln_+ \left( \frac{4 \zeta_t^2 B_t(R + 2 \zeta_t R  t)\sqrt{Q_t} }{\epsilon^2} \right) +B_t + \frac{G_t}{R}  + 2 d (\kappa +\delta)B_t. \nn
		\end{gather}
		We recall that $Q_t \coloneqq \epsilon^2+ \sum_{i=1}^t \|\g_i\|^2\leq \epsilon^2  + t B_t^2$. With this in hand, we have
		\begin{equation}
			\sum_{t=1}^T \langle \hat{\g}_t, \overline{ \x}_{t-1}-\overline{ \x}_T\rangle 
			\le \sum_{t=1}^{T} \left\|\sum_{i=1}^t \hat{\g}_i\right\|  \|\overline{ \x}_{t-1}-\overline{ \x}_t\|
			\le M_T\sum_{t=1}^{T}\sqrt{\epsilon^2 +\sum_{i=1}^t \|\hat{\g}_i\|^2  }\|\overline{ \x}_{t-1}-\overline{ \x}_t\| + N_T \sum_{t=1}^{T}\|\overline{ \x}_{t-1}-\overline{ \x}_t\|~.\nn
		\end{equation}
		Now, we relate $\|\overline{ \x}_{t-1}-\overline{ \x}_{t}\|$ to $\| \x_t-\overline{ \x}_t\|$:
		\[
		\overline{ \x}_{t-1}-\overline{ \x}_t 
		= \overline{ \x}_{t-1} - \frac{Z_{t-1}\overline{ \x}_{t-1} + \|\hat{\g}_t\|  ^2  \x_t}{Z_t}
		=\frac{\|\hat{\g}_t\|  ^2}{Z_t}(\overline{ \x}_{t-1} -  \x_t)
		=\frac{\|\hat{\g}_t\|  ^2}{Z_t}(\overline{ \x}_{t} -  \x_t) + \frac{\|\hat{\g}_t\|  ^2}{Z_t}(\overline{ \x}_{t-1}-\overline{ \x}_t),
		\]
		that implies
		\[
		Z_t \cdot (\overline{ \x}_{t-1}-\overline{ \x}_t)
		= \|\hat{\g}_t\|  ^2( \x_t-\overline{ \x}_t)+\|\hat{\g}_t\|  ^2(\overline{ \x}_{t-1}-\overline{ \x}_t),
		\]
		that is
		\begin{equation}
			\label{eq:lemma_shortcut_eq1}
			\overline{ \x}_{t-1}-\overline{ \x}_t
			= \frac{\|\hat{\g}_t\|  ^2}{Z_{t-1}}( \x_t-\overline{ \x}_t)~.
		\end{equation}
		Hence, we have 
		\[
		M_T\sum_{t=1}^{T}\sqrt{\epsilon^2 +\sum_{i=1}^t \|\hat{\g}_i\|^2  }\|\overline{ \x}_{t}-\overline{ \x}_{t-1}\|
		\leq M_T\sum_{t=1}^T \sqrt{Z_t}\frac{\| \hat \g_t\|  ^2}{Z_{t-1}}\| \x_t-\overline{ \x}_t\|,
		\]
		and
		\[
		N_T\sum_{t=1}^{T}\|\overline{ \x}_{t}-\overline{ \x}_{t-1}\|
		\leq N_T\sum_{t=1}^T \frac{\|\hat \g_t\|  ^2}{Z_{t-1}}\| \x_t-\overline{ \x}_t\|
		\leq 2 RN_T \sum_{t=1}^T \frac{\|\hat  \g_t\|  ^2}{Z_{t-1}}~.
		\]
		Using Cauchy–Schwarz inequality, we have
		\begin{align*}
			M_T\sum_{t=1}^T \sqrt{Z_t}\frac{\| \hat \g_t\|  ^2}{Z_{t-1}}\| \x_t-\overline{ \x}_t\|
			&\leq M_T\sqrt{\sum_{t=1}^T \frac{\|\hat{\g}_t\| ^2}{Z_{t-1}}} \sqrt{\sum_{t=1}^T \frac{Z_t}{Z_{t-1}}\|\hat{\g}_t\|^2\| \x_t-\overline{ \x}_t\|^2 }~.
		\end{align*}
		So, putting together the last inequalities, we have
		\[
		\sum_{t=1}^T \langle \hat{\g}_t, \overline{ \x}_{t-1}-\overline{ \x}_T\rangle \leq M_T\sqrt{\sum_{t=1}^T \frac{\|\hat{\g}_t\|  ^2}{Z_{t-1}}} \sqrt{\sum_{t=1}^T \frac{Z_t}{Z_{t-1}}\|\hat{\g}_t\|  ^2\| \x_t-\overline{ \x}_t\|^2 } +2R N_T \sum_{t=1}^T \frac{\| \hat \g_t\|  ^2}{Z_{t-1}}~.
		\]
		We now focus on the the term $\sum_{t=1}^T \frac{\|\hat \g_t\|  ^2}{Z_{t-1}}$ that is easily bounded:
		\begin{align*}
			\sum_{t=1}^T \frac{\| \hat \g_t\|  ^2}{Z_{t-1}} &\leq \sum_{t=1}^T \frac{\| \hat \g_t\|  ^2}{\epsilon^2 + \sum_{i=1}^t \|\hat \g_i\|^2 } \leq \ln \frac{Z_T}{ \epsilon^2},
		\end{align*}
		where the first inequality follows by the fact that $\tilde B_{t-1}\geq \|\hat \g_t\|$, and in the last inequality we used the inequality $$\sum_{t=1}^T \frac{a_t}{\sum_{i=0}^t a_i} \leq \ln\left(\frac{\sum_{t=0}^T a_t}{a_0}\right),$$ 
		for all $a_t\geq0$. To bound the term $\sum_{t=1}^T \frac{Z_t}{Z_{t-1}}\|\tilde{\g}_t\|  ^2\| \x_t-\overline{ \x}_t\|^2$ from above, observe that
		\begin{align*}
			\sigma_T^2 Z_T 
			&=(\epsilon^2 +\tilde B_T^2)\|\overline{ \x}_T\|^2+\sum_{t=1}^T \|\hat{\g}_t\|  ^2\| \x_t-\overline{ \x}_T\|^2,\\ 
			&\geq (\epsilon^2 +\tilde B_{T-1}^2)\|\overline{ \x}_T\|^2+\sum_{t=1}^{T-1} \|\hat{\g}_t\|  ^2\| \x_t-\overline{ \x}_T\|^2 + \|\hat{\g}_T\|  ^2\| \x_T-\overline{ \x}_T\|^2, \\
			&= Z_{T-1}\cdot (\sigma_{T-1}^2+\|\overline{ \x}_T-\overline{ \x}_{T-1}\|^2) +\|\hat{\g}_T\|  ^2\| \x_T-\overline{ \x}_T\|^2, \\
			&= Z_{T-1}\sigma_{T-1}^2+\|\hat{\g}_T\|  ^2\left(1+\frac{\|\hat{\g}_T\|  ^2}{Z_{T-1}}\right)\| \x_T-\overline{ \x}_T\|^2, \\
			&= Z_{T-1}\sigma_{T-1}^2+\|\hat{\g}_T\|  ^2 \frac{Z_T}{Z_{T-1}} \| \x_T-\overline{ \x}_T\|^2,
		\end{align*}
		where the third equality comes from bias-variance decomposition and the fourth one comes from~\eqref{eq:lemma_shortcut_eq1}. Hence, we have
		\[
		\sum_{t=1}^T \frac{Z_t}{Z_{t-1}} \|\hat{\g}_t\|  ^2\| \x_t-\overline{ \x}_t\|^2 
		= \sum_{t=1}^T (\sigma_t^2 Z_t - \sigma_{t-1}^2 Z_{t-1} ) \leq \sigma^2_T Z_T~.
		\]
		Putting all together, we have the stated bound.
	\end{proof}

	\subsection{Proof of Theorem \ref{thm:genthmsmooth} (The Smooth Stochastic Case)}
	\label{sec:genthmsmoothproof}
	We will need to use the result of Lemma \ref{lem:mainreductiongen} (in particular \eqref{eq:individualnew}) to show the desired convergence rate. In order for the result of Lemma \ref{lem:mainreductiongen} to be valid in the setting of Theorem \ref{thm:genthmsmooth}, we need to show that the iterates $(\w_t)$ in Algorithm \ref{alg:projectionfreewrappersmooth} are in $\cB(R)$, which is what we do next:
	\begin{lemma}
		In the setting of Theorem \ref{thm:genthmsmooth}, we have $(\w_t)\subset \cB(R)$.
		\end{lemma}
		\begin{proof}[{\bf Proof}]
 Let $\gamma_t$ be as in Algorithm \ref{alg:projectionfreewrappersmooth}. For $t=1$, we have $\w_1=\bm{0}$. For $t>1$, we have $\w_{t}=(1-\mu_{t})(\x_{t-1} - \eta_{t-1} \g_{t-1}) +\mu_{t} \bu_{t}$ and $\|\bu_t\|\leq R'$ (since it is the output of \ftrl{} with parameter $R'$), and so
		\begin{align}
			\|\w_{t}\| \leq (1-\mu_t)\|\x_{t-1}\| +  \mu_t R'+ |\eta_{t-1}| \|\g_{t-1}\| &= (1-\mu_t) \|\x_{t-1}\| + \mu_t R' + \frac{\nu R'  \|\g_{t-1}\|}{\sqrt{Z_{t-1}}}, \nn \\ &\leq (1-\mu_t) \|\x_{t-1}\| + (\mu_t+\nu) R'. \nn
		\end{align}
		Thus, $\w_t\in  \cB(R)$ if $\x_t\in \K$ (which implies $\|\x_t\|\leq R'$). We will now show that $\x_{t}\in \K$. We consider two cases. Suppose that $\gamma_t <1$. In this case, $\x_t=\w_t$. If $\gamma_{\K}(\w_t)\leq 9/16$, then by definition of the Gauge function we have $\w_t \in \K$ and the same holds for $\x_t$ (since $\x_t=\w_t$). On the other hand, if $\gamma_\K(\w_t)\geq 9/16$, then by Lemma \ref{lem:gauge}, we have $\gamma_{\K}(\w_t)\leq \gamma_t<1$, and so $\x_t=\w_t\in \K$. 
		
		Now suppose that $\gamma_t \geq 1$. In this case, we have $\x_t= \w_t/\gamma_t$ and so $	\gamma_\K(\x_t)=\gamma_\K(\w_t/{\gamma_{t}})=\gamma_\K(\w_t)/{\gamma_{t}}\leq \gamma_\K(\w_t)/{\gamma_{\K}(\w_t)}=1$ (since $\gamma_{\K}(\w_t)\leq \gamma_t$ by Lemma \ref{lem:gauge}), and so $\x_t \in \K$.
	\end{proof}
\noindent Now that we know Lemma \ref{lem:mainreductiongen} holds in the setting of Theorem \ref{thm:genthmsmooth}, we state a version of the latter for the case where $\cO_{\K^\circ} =\onedopt$:
\begin{lemma}
	\label{lem:newbound}
	Consider the same setting of Lemma \ref{lem:mainreductiongen} except with $\cO_{\K^\circ} \equiv \onedopt$. Then, for any $t\geq 1$,
	\begin{align}
		\x_t\in \K;\ \ \forall \x\in \K, \ \  \inner{\g_t}{\x_t-\x} \leq \E_{t-1} [\inner{\tilde{\g}_t}{\w_t-\x}] + 3 \delta_t R \|\g_t\|; \ \ \text{and} \ \ \|\tilde \g_t\|\leq (1+d\kappa+d\Delta_t) \|\g_t\|,   \nn
 	\end{align}
 	where we recall that $\kappa=R'/r$ and $\cB(r)\subseteq \K\subseteq \cB(R')\subseteq \cB(R)$ (this is the setting of Lemma \ref{lem:mainreductiongen}).
\end{lemma}
\begin{proof}[{\bf Proof}]
		Let $(\tilde \gamma_t, \tilde \s_t)= \O_{\K^\circ}(\w_t;\delta_t)$ and $(\hat \gamma_t, \hat \s_t)= \onedopt(\w_t;\delta_t)$. Further, let $\tilde  \bnu_t\coloneqq \mathbb{I}_{\{\tilde \gamma_t\geq 1  \}} \tilde \s_t$, and $\gamma_t$ and $\bnu_t$ be as in Algorithm \ref{alg:projectionfreewrapper}. Note that by Proposition \ref{prop:subgradient} and Lemma \ref{lem:celeb}, we have $\|\tilde \s_t\|\vee \|\hat \s_t\| < +\infty$ almost surely and so the expectations $\E_{t-1}[ \mathbb{I}_{\{\tilde \gamma_t\geq 1  \}} \tilde \s_t]$ and  $\E_{t-1}[\mathbb{I}_{\{\hat \gamma_t\geq 1  \}} \hat \s_t]$ are well defined. By Lemma \ref{lem:mainreductiongen}, there exists a random variable $\Delta_t\geq 0$ satisfying $\E_{t-1}[\Delta_t]$ and such that for all $t\geq 1$ and $\x\in \K$, \begin{align}
		\inner{\g_t}{\x_t - \x}  \leq \inner{\g_t - \mathbb{I}_{\inner{\g_t}{\w_t}<0} \inner{\g_t}{\x_t} \tilde \bnu_t}{\w_t-\x} + (2\delta_t +\Delta_t)R \|\g_t\|. 
		\label{eq:six1}\end{align}
	Thus, since $\gamma_t=\hat \gamma_t= \tilde \gamma_t$, and $\gamma_t$, $\w_t$, and $\x_t$ are all deterministic functions of the past ($\w_t$ it the output of \ftrl, which is a deterministic function of the past), we have
	\begin{align}
			\inner{\g_t}{\x_t - \x}& = \E_{t-1}[ \inner{\g_t}{\x_t - \x} ], \nn \\ & \leq   \E_{t-1}[\inner{\g_t - \mathbb{I}_{\inner{\g_t}{\w_t}<0} \inner{\g_t}{\x_t} \tilde \bnu_t}{\w_t-\x}+(2\delta_t +\Delta_t)R \|\g_t\| ],  \ \ \text{(by \eqref{eq:six1})} \nn \\ 
		& \leq  \inner{\g_t - \mathbb{I}_{\inner{\g_t}{\w_t}<0} \inner{\g_t}{\x_t}\E_{t-1}[ \tilde\bnu_t]}{\w_t-\x}  + 3\delta_t R \|\g_t\|,\nn \\ 
		&  =   \mathbb{I}_{\{ \tilde \gamma_t <1\}} \cdot   \inner{\g_t }{\w_t-\x}  \nn \\ & \quad  +  \mathbb{I}_{\{\tilde \gamma_t \geq 1\}} \cdot    \inner{\g_t - \mathbb{I}_{\inner{\g_t}{\w_t}<0} \inner{\g_t}{\x_t}\E_{t-1}[ \tilde  \s_t]}{\w_t-\x}+ 3\delta_t R \|\g_t\|, \nn \\ 
		& =  \mathbb{I}_{\{ \tilde \gamma_t<1\}} \cdot     \inner{\g_t }{\w_t-\x} \nn \\ 
		& \quad  +   \mathbb{I}_{\{ \hat  \gamma_t \geq 1 \}} \cdot  \inner{\g_t - \mathbb{I}_{\inner{\g_t}{\w_t}<0} \inner{\g_t}{\x_t} \E_{t-1}[\hat \s_t]}{\w_t-\x} + 3\delta_t R \|\g_t\|, \label{eq:attract1}  \\ & =   \E_{t-1}[\inner{\g_t - \mathbb{I}_{\inner{\g_t}{\w_t}<0} \inner{\g_t}{\x_t} \bnu_t}{\w_t-\x} ] + 3\delta_t R \|\g_t\|, \nn \\ 
	& =  \E_{t-1} [\inner{\tilde{\g}_t}{\w_t-\x}] + 3 \delta_t R \|\g_t\|,	\nn
	\end{align}
	where \eqref{eq:attract1} follows by the facts that $\tilde \gamma_t = \hat \gamma_t$ and $\E_{t-1}[\hat \s_t]= \E_{t-1}[\tilde  \s_t]$, when $\tilde \gamma_t \geq 1$ (Lemma \ref{lem:celeb}). It remains to bound $\|\tilde \g_t\|$ and show that $\x_t \in \K$. The latter follows from the proof of Lemma \ref{lem:mainreductiongen} since $\gamma_t=\hat\gamma_t = \tilde \gamma_t$. We now bound $\|\tilde \g_t\|$. When $\gamma_t <1$, we have $\g_t=\tilde \g_t$ and so $\|\tilde \g_t\| \leq (1 + d\kappa + d\Delta_t) \|\g_t\|$ holds trivially. Now suppose that $\gamma_t \geq 1$. Then, we have 
	\begin{align}
		\|\tilde{\g}_t\|  =  \|\inner{\g_t}{\w} - \mathbb{I}_{\inner{\g_t}{\w_t} <0}  \inner{\g_t}{{\x}_t}  \bnu_t\| \stackrel{(*)}{\leq} \|\g_t\|(1 + R \|\bnu_t\|)  
	\stackrel{(**)}{\leq}  (1+ d\Delta_t +d\kappa )\|\g_t\|, \nn
\end{align}
where $(*)$ follows from the fact that $\|\x_t\|\le R'$ (since $\x_t \in \K$), and $(**)$ follows the fact that $\|\bnu_t\|\leq d(\Delta_t/R+1/r)$ by Lemma \ref{lem:celebext}.
	\end{proof}
	
\noindent The following lemma will also be useful to us in the proof of Theorem \ref{thm:genthmsmooth}:
	\begin{lemma}
		\label{lem:conv}
		Let $\beta >0$, $t\geq 1$, $\bxi_t \in \reals^d$ be a random vector satisfying \eqref{eq:stochsetting} for $\sigma>0$, and $\ell_t$ be as in \eqref{eq:modifiedloss}. Further, suppose that $\K$ satisfies \eqref{eq:newnewrad}. When $f$ is $\beta$-smooth on $\K$, we have, 
		\begin{align}
			\E[\|\g_t\|^2]\leq  \sigma^2 + 2 R' \beta, \quad \text{for any $\x\in \K$ and $\g_t\in \partial \ell_t(\x)$.} \nn
		\end{align}
	\end{lemma}
	\begin{proof}[{\bf Proof}]
		Let $\x\in \K$ and $\g_t \in \partial \ell_t(\x)$. First, since $f$ is differentiable, $\ell_t$ is also differentiable. Thus, $\g_t = \nabla f(\x) + \bxi_t$, and so 
		\begin{align}
			\E[\|\g_t\|^2] = \E[\|\nabla f(\x)\|^2 + 2 \inner{\nabla f(\x)}{ \bxi_t} + \|\bxi_t\|^2]&  = \|\nabla f(\x)\|^2 + \E[\|\bxi_t\|^2] \leq  2R'\beta  + \sigma^2, \nn
		\end{align}
		where the last inequality follows \eqref{eq:stochsetting} and \eqref{eq:crebro}.
	\end{proof}
	\noindent We now present the proof of Theorem \ref{thm:genthmsmooth}. We note that the proof is very similar to that of \cite[Theorem 4]{cutkosky2019} with modifications to account for the application of our projection-free reduction (Algorithm \ref{alg:projectionfreewrapper}).
	\begin{proof}[{\bf Proof of Theorem \ref{thm:genthmsmooth}}]
		Throughout this proof, we let $\lambda_t =t$, $\forall t$, and $\nu$ and $R'$ be as in \eqref{eq:newnewrad}. Note that $\Lambda_t$ in Algorithm \ref{alg:projectionfreewrappersmooth} satisfies $\Lambda_t =\sum_{s=1}^t \lambda_s$. By a standard convexity argument, we have
		\begin{align}
			\E\left[\sum_{t=1}^T \lambda_t(f(\x_t) -f(\x))\right]&\le \E\left[\sum_{t=1}^T \lambda_t\langle \g_t, \x_t -\x\rangle\right]\nn \\
			& \leq \E\left[\sum_{t=1}^T \lambda_t\langle \tilde\g_t, \w_t -\x\rangle\right] +  \E\left[\sum_{t=1}^T 3\lambda_{t} \delta_t R \|\g_t\| \right],\quad (\text{by Lemma \ref{lem:newbound}})\nn \\ 
			&=\E\left[\sum_{t=1}^T  \lambda_t\langle \tilde \g_t, \bu_t - \x\rangle+\sum_{t=1}^T \lambda_t\langle \tilde \g_t, \w_t - \bu_t\rangle \right] \nn \\ & \quad +   \E\left[ \sum_{t=1}^T3\lambda_{t} \delta_t R \cdot \E[ \|\g_t\| \mid \x_t ] \right],\nn \\
			& \leq \E\left[\sum_{t=1}^T  \lambda_t\langle \tilde \g_t, \bu_t - \x\rangle+\sum_{t=1}^T \lambda_t\langle \tilde \g_t, \w_t - \bu_t\rangle \right]+ 3 R \sqrt{\sigma^2 + 2 \beta R'} \sum_{t=1}^T \lambda_t \delta_t,  \label{eq:combining}
		\end{align}
		where the last inequality follows by Lemma \ref{lem:conv} and Jensen's inequality. By letting $\y_{s}\coloneqq  \x_{s}-\eta_s \g_s$, we have by Line \ref{line:blueend} of Algorithm \ref{alg:projectionfreewrappersmooth} \begin{align}\label{eq:firstone} \lambda_s \cdot (\w_s - \bu_s) = \Lambda_{s-1} \cdot (\y_{s-1}-\w_s), \ \ \text{ for all $s\geq 1$.}
		\end{align}
		Moreover, the first term on the RHS of \eqref{eq:combining} is the regret of the \ftrl{} instance (Algorithm $\A$) within Algorithm \ref{alg:projectionfreewrappersmooth}. Thus, by Proposition \ref{prop:ftrl-proximal} and Lemma \ref{lem:mainreductiongen}, this term is bounded from above by 
		\begin{align}
	\scR^{\Ftrl}_T(\x) \coloneqq 4 (1+d\kappa) R' \sqrt{\sum_{t=1}^T \lambda_t^2 \|\g_t\|^2} + 4 dR'\sum_{t=1}^T\lambda_t  \Delta_t \|\g_t\| , \quad \forall \x\in \K.  
		 \label{eq:regftrl2}
		\end{align} 
		Plugging \eqref{eq:firstone} and \eqref{eq:regftrl2} into \eqref{eq:combining}, yields
		\begin{align}
			\E\left[\sum_{t=1}^T \lambda_t(f(\x_t) -f(\x))\right]   &\le\E\left[\scR^{\Ftrl}_T(\x)\right] + \E\left[\sum_{t=1}^T\Lambda_{t-1} \langle \tilde \g_t, \y_{t-1} - \w_t\rangle\right]+  3 R \sqrt{\sigma^2 + 2 \beta R'} \sum_{t=1}^T \lambda_t \delta_t,\nn \\
			&\le\E\left[\scR^{\Ftrl}_T(\x)\right] + \E\left[\sum_{t=1}^T\Lambda_{t-1} \langle  \g_t, \y_{t-1} - \x_t\rangle\right] \nn \\ & \qquad  +\E\left[\sum_{t=1}^T3\Lambda_{t-1}  \delta_t R \|\g_t\| \right] +3 R \sqrt{\sigma^2 + 2 \beta R'} \sum_{t=1}^T \lambda_t \delta_t,\label{eq:cor} \\
			& =   \E\left[\scR^{\Ftrl}_T(\x)\right] + \E\left[\sum_{t=1}^T\Lambda_{t-1} \langle  \g_t, \y_{t-1} - \x_t\rangle\right]\nn \\ & \qquad +\E\left[\sum_{t=1}^T 3 \delta_t \Lambda_{t-1} R \cdot \E[ \|\g_t\|\mid \x_t] \right]+  3 R \sqrt{\sigma^2 + 2 \beta R'} \sum_{t=1}^T \lambda_t \delta_t,\nn \\
			& \leq \E\left[\scR^{\Ftrl}_T(\x)\right] + \E\left[\sum_{t=1}^T\Lambda_{t-1} \langle  \g_t, \y_{t-1} - \x_t\rangle\right] +3 R\sqrt{\sigma^2  + 2\beta R'}\sum_{t=1}^T( \lambda_t+\Lambda_{t-1})  \delta_t,\nn
		\end{align}
		where \eqref{eq:cor} follows by Lemma \ref{lem:newbound}, and the last inequality follows by Lemma \ref{lem:conv} and Jensen's inequality. Next, we use convexity again to argue $$\E[\langle \g_t, \y_{t-1} -\x_t\rangle ]\le \E[f(\y_{t-1}) - f(\x_t)],$$ and then we subtract $\E[\sum_{t=1}^T \lambda_t f(\x_t)]$ from both sides:
		\begin{align}
			\E[-\Lambda_{T} f(\x) ]&\le \E\left[\scR^{\Ftrl}_T(\x)\right]+\sum_{t=1}^T  \left(\Lambda_{t-1}f(\y_{t-1}) -\Lambda_t f(\x_t)\right)+3 R\sqrt{\sigma^2  + 2\beta R'}\sum_{t=1}^T( \lambda_t+\Lambda_{t-1})  \delta_t.  \nn
		\end{align}
		Now we use smoothness to relate $f(\y_{t})$ to $f(\x_{t})$. Let $\nu \coloneqq 4\sqrt{2} (1+d\kappa)$ so that $\eta_t = \nu R'/\sqrt{Z_t}$, where $Z_t$ is as in Algorithm \ref{alg:projectionfreewrappersmooth}; i.e.~$Z_t \coloneqq \epsilon^2 +\sum_{s=1}^t \Lambda_s \|\g_s\|^2$. With this, we have:
		\begin{align*}
			\E[f(\y_{t})]&\le \E[f(\x_{t}) + \nabla f(\x_{t})(\y_{t}-\x_{t}) + \frac{\beta}{2}\|\x_{t}-\y_{t}\|^2],\\
			&\le  \E\left[f(\x_{t}) - \eta_t\|\g_t\|^2 + \eta_t \langle \bxi_t,\g_t\rangle +\frac{\beta\eta_t^2\|\g_t\|^2}{2}\right].
		\end{align*}
		Then multiply by $\Lambda_t$:
		\begin{align*}
			\E[\Lambda_{t}(f(\y_{t}) - f(\x_{t}))]&\le \E\left[- \frac{\nu R'\Lambda_t \|\g_t\|^2}{\sqrt{\epsilon^2+\sum_{i=1}^{t} \Lambda_i \|\g_i\|^2}} + \frac{\beta\eta_t^2\Lambda_{t}\|\g_t\|^2}{2}+\eta_t \Lambda_t\langle \bxi_t,\g_t\rangle\right].
		\end{align*}
		Next, we make use of the following facts (see e.g.~\cite{cutkosky2019,levy2018online}): for positive numbers $\alpha_0,\dots,\alpha_n$,
		\begin{align*}
			\sqrt{\sum_{i=1}^n \alpha_i}\le \sum_{i=1}^n \frac{\alpha_i}{\sqrt{\sum_{j=1}^i \alpha_{j}}}\le 2\sqrt{\sum_{i=1}^n \alpha_i}
			\quad \text{and} \quad
			\sum_{i=1}^n \frac{\alpha_i}{\alpha_0+\sum_{j=1}^i \alpha_{j}}\le \ln\left(\alpha_0+\sum_{i=1}^n \alpha_i\right) -\ln \alpha_0.
		\end{align*}
		Using this, we obtain
		\begin{align}
			\E\left[\sum_{t=1}^T \Lambda_t (f(\y_{t}) - f(\x_{t}))\right]&\le \E\left[-\nu R' \sqrt{ \epsilon^2+\sum_{t=1}^T \Lambda_t\|\g_t\|^2} +\frac{\nu^2(R')^2\beta}{2} \ln\left(1+\frac{\sum_{t=1}^T \Lambda_{t} \|\g_t\|^2}{\epsilon^2}\right)\right] \nn \\ & \qquad   +\nu R' \epsilon  + \E\left[\sum_{t=1}^T \eta_t \langle \bxi_t, \Lambda_t \g_t\rangle \right],\nn \\
			&\le \E\left[-\nu R' \sqrt{ \epsilon^2+\sum_{t=1}^T \Lambda_t\|\g_t\|^2}\right] +\nu R' \epsilon  + \E\left[\sum_{t=1}^T \eta_t \langle \bxi_t, \Lambda_t \g_t\rangle \right] \nn \\ & \qquad  +\frac{\nu^2(R')^2\beta}{2} \ln\left(1+\frac{\sum_{t=1}^T \Lambda_{t}  \E\left[\|\g_t\|^2\right]}{\epsilon^2}\right) ,\label{eq:triangle} \\
			& \leq \E\left[-\nu R' \sqrt{ \epsilon^2+\sum_{t=1}^T \Lambda_t\|\g_t\|^2} \right] + \nu R'\epsilon + \frac{\nu^2(R')^2\beta}{2} \ln\left(1+\frac{(\sigma^2 +2 R' \beta)\sum_{t=1}^T \Lambda_{t}  }{\epsilon^2}\right) \nn \\
			& \qquad + \E\left[\sum_{t=1}^T \eta_t \langle \bxi_t, \Lambda_t \g_t\rangle \right],\label{eq:thisone}
		\end{align}
		where \eqref{eq:triangle} follows by Jensen's inequality (the log is concave) and the triangular inequality, and \eqref{eq:thisone} follows by Lemma \ref{lem:conv}.
		Using Cauchy-Schwarz, we obtain:
		\begin{align}
			\E\left[\sum_{t=1}^T\eta_t \langle \bxi_t, \Lambda_t\g_t\rangle \right]&\le \E\left[\sqrt{\sum_{t=1}^T \Lambda_t \|\bxi_t\|^2 }\sqrt{\sum_{t=1}^T \eta_t^2 \Lambda_t \|\g_t\|^2}\right],\nn \\
			&\le \E\left[\nu R'  \sqrt{ \sum_{t=1}^T \Lambda_t \|\bxi_t\|^2 \cdot \ln\left(1+ \epsilon^{-2}\sum_{s=1}^T\Lambda_s\|\g_s\|^2\right)}\right],\nn \\
			&\le \E\left[\nu R' \sqrt{ \sum_{t=1}^T \Lambda_t \|\bxi_t\|^2 \cdot \ln\left(1+ 2\epsilon^{-2}\sum_{s=1}^T\Lambda_s(\|\bxi_s\|^2+\|\nabla f(\x_s)\|^2)\right)}\right],\nn\\
			&\le \E\left[\nu R'  \sqrt{ \sum_{t=1}^T \Lambda_t \|\bxi_t\|^2 \cdot \ln\left(1+ 2\epsilon^{-2}\sum_{s=1}^T\Lambda_s(\|\bxi_s\|^2+2 R' \beta)\right)}\right], \quad (\text{by \eqref{eq:crebro}}),\nn \\
			&\le \E\left[\nu R' \sqrt{ \sum_{t=1}^T \Lambda_t \|\bxi_t\|^2 \cdot \ln\left(1+ 2\epsilon^{-2}\sum_{s=1}^T\Lambda_s(\|\bxi_s\|^2+2 R' \beta)\right)}\right],\label{eq:jen} \\
			& \le \nu \sigma  R' \sqrt{\sum_{t=1}^T\Lambda_t   \cdot \ln\left(1+ \frac{ 2(\sigma^2 +2 R' \beta)\sum_{s=1}^T\Lambda_s }{\epsilon^2}\right)},\nn 
		\end{align}
		where \eqref{eq:jen} follows by the concavity of the function $x \mapsto \sqrt{x \cdot\ln (a+b x)}$, for any $a\geq1$, $b>0$, and $x>0$ (see Lemma \ref{lem:concave}) and Jensen's inequality, and the last inequality follows by the fact that $\E[\|\bxi_t\|^2]\le \sigma^2$.
		
		Combining everything, we get
		\begin{align*}
			\E\left[\sum_{t=1}^T -\lambda_t f(\x)\right]&\le\E\left[ \scR^{\Ftrl}_T(\x)+\sum_{t=1}^T \left(\Lambda_{t-1}f(\y_{t-1}) -\Lambda_t f(\y_t)\right)\right]\\
			&\quad +\frac{\nu^2\beta (R')^2}{2}\ln\left(1+\frac{\sum_{t=1}^T \Lambda_{t}(\sigma^2 +2 R' \beta) }{\epsilon^2}\right)-\nu R' \sqrt{\epsilon^2+\sum_{t=1}^T \Lambda_t\|\g_t\|^2} \nn \\ & \quad  + \nu R'\epsilon +\nu R'\sigma \sqrt{\sum_{t=1}^T \Lambda_{t} \ln\left(1+\frac{2(\sigma^2 +2 R' \beta) \sum_{s=1}^T \Lambda_{s} }{\epsilon^2}\right)} + 3 R\sqrt{\sigma^2  + 2\beta R'}\sum_{t=1}^T( \lambda_t+\Lambda_{t-1})  \delta_t.
		\end{align*}
		Now observe that  $t^2>\Lambda_t>\lambda_t^2/2$ and recall $\scR^{\Ftrl}_T(\x)= 4 (1+d\kappa) R'\sqrt{\sum_{t=1}^T \lambda_t^2 \|\g_t\|^2} + 4d R' \sum_{t=1}^T\lambda_t \Delta_t  \|\g_t\|$. Therefore, since $\nu=4\sqrt{2}(1+d\kappa)$ we have:
		\begin{align*}
			\E\left[\scR^{\Ftrl}_T(\x)-\nu R' \sqrt{\epsilon^2+\sum_{t=1}^T \Lambda_t\|\g_t\|^2}\right]& \le \E\left[4(1+d\kappa) R'\sqrt{\sum_{t=1}^T \lambda_t^2 \|\g_t\|^2}- 4 \sqrt{2} (1+d\kappa) R' \sqrt{\sum_{t=1}^T \lambda_t^2\|\g_t\|^2/2}\right]\nn \\ & \qquad  + \E\left[4dR' \sum_{t=1}^T \lambda_t \Delta_t \|\g_t\| \right],\nn \\
			& = \E\left[4dR' \sum_{t=1}^T \lambda_t \Delta_t \cdot \E[\|\g_t\| \mid \x_t ] \right],\nn \\
			& \leq 4d R'\sqrt{\sigma^2 + 2\beta R} \sum_{t=1}^T \lambda_t \delta_t, 
		\end{align*}
		where the last inequality follows by Lemma \ref{lem:conv}. Also, observe that $\sum_{t=1}^T\Lambda_{t}\le\sum_{t=1}^T t^2 \le T^3$. Thus, we telescope the sum to obtain:
		\begin{align}
			\E[\Lambda_{T}(f(\y_T)-f(\x))]&\le \nu R'\epsilon +\frac{\nu^2(R')^2 \beta}{2} \ln\left(1+\frac{(\sigma^2 +2 R' \beta)T^3 }{\epsilon^2}\right)+ \nu R'T^{3/2}\sigma\sqrt{\ln\left(1+\frac{2(\sigma^2 +2 R' \beta)T^3 }{\epsilon^2}\right)}\nn \\
			& \ \qquad  + 3 R\sqrt{\sigma^2  + 2\beta R'}\sum_{t=1}^T(\Lambda_{t-1}+3d\lambda_t)  \delta_t,\nn \\
			&\le \nu R'\epsilon +\frac{\nu^2(R')^2 \beta}{2} \ln\left(1+\frac{(\sigma^2 +2 R' \beta)T^3 }{\epsilon^2}\right)+ \nu R'T^{3/2}\sigma\sqrt{\ln\left(1+\frac{2(\sigma^2 +2 R' \beta)T^3 }{\epsilon^2}\right)}\nn \\
			& \ \qquad  + 3\delta  R (\ln T+6d) \sqrt{\sigma^2  + 2\beta R'}, \label{eq:penultimate}
		\end{align}
		where the last inequality follows by the fact that $\delta_t=\delta/t^3$, $\sum_{t=1}^T 1/t \leq \ln T$; and $\sum_{t=1}^T 1/t^2 \leq 2$.
		Dividing \eqref{eq:penultimate} by $\Lambda_T = \sum_{t=1}^T\lambda_{t}=\tfrac{T(T+1)}{2} >T^2/2$ shows the inequality of theorem. Finally, the fact that $(\x_t)\subset \K$ follows by Lemma \ref{lem:newbound}.
	\end{proof}
	\clearpage
	\bibliography{biblio}
	\bibliographystyle{alpha}
	
\end{document}

%% file: arxiv_style.tex
\usepackage[letterpaper, left=1in, right=1in, top=1in, bottom=1in]{geometry}

\usepackage[dvipsnames]{xcolor}
\usepackage{microtype}

\usepackage{algorithm}
\usepackage{verbatim}
\usepackage[noend]{algpseudocode}
\usepackage{enumitem}

\usepackage{amsthm}
\usepackage{mathtools}
\usepackage{amsmath}
\usepackage{bm}
\usepackage{bbm}
\let\vec\undefined
\usepackage[mathscr]{eucal}
\usepackage{hyperref}

\usepackage{graphicx}
\usepackage{pdfpages}
\usepackage{thm-restate}
\usepackage{url}
\usepackage{comment}

\usepackage{xpatch}

\newcommand{\savehyperref}[2]{\texorpdfstring{\hyperref[#1]{#2}}{#2}}

\renewcommand{\eqref}[1]{\savehyperref{#1}{(\ref*{#1})}}

\theoremstyle{definition} 

\newtheorem{manualtheoreminner}{Lemma}
\newenvironment{manuallemma}[1]{%
	\renewcommand\themanualtheoreminner{#1}%
	\manualtheoreminner
}{\endmanualtheoreminner}

\newtheorem{assumption}{Assumption}

\theoremstyle{plain}
\newtheorem{remark}{Remark}

\newtheorem{theorem}{Theorem}
\newtheorem{definition}{Definition}

\newtheorem{lemma}[theorem]{Lemma}

\newtheorem{corollary}[theorem]{Corollary}
\newtheorem{proposition}[theorem]{Proposition}

\xpatchcmd{\proof}{\itshape}{\normalfont\proofnameformat}{}{}
\newcommand{\proofnameformat}{\bfseries}


%% file: zak.tex
\definecolor{ashgrey}{rgb}{0.7, 0.75, 0.71}

\usepackage{todonotes}
\usepackage{tikz}
\usetikzlibrary{fit,calc}

\colorlet{pink}{red!40}
\colorlet{bluenew}{cyan!60}

\newcommand{\ceil}[1]{\lceil #1 \rceil}

\renewcommand{\v}{\bm{v}}
\newcommand{\e}{\bm{e}}
\newcommand{\z}{\bm{z}}

\newcommand{\bu}{\bm{u}}

\newcommand{\X}{\bm{X}}

\newcommand{\bnabla}{\bm{\nabla}}
\newcommand{\G}{\bm{G}}
\newcommand{\K}{{\mathcal{C}}}
\renewcommand{\O}{{\mathsf{OPT}}}

\newcommand{\mem}{\mathsf{MEM}}
\newcommand{\y}{\bm{y}}
\newcommand{\s}{\bm{s}}
\newcommand{\scR}{\mathscr{R}}
\newcommand{\Cost}{\text{Cost}}

\newcommand{\norm}[1]{\left\| #1 \right\|}

\renewcommand{\mem}{\mathsf{MEM}}
\newcommand{\gau}{\mathsf{GAU}}

\newcommand{\metagrad}{\mathsf{MetaGrad}}

\newcommand{\ftrl}{$\mathsf{FTRL}$-$\mathsf{proximal}$}
\newcommand{\Ftrl}{\mathsf{FTRL}} 

\newcommand{\A}{\mathsf{A}}

\newcommand{\w}{\bm{w}}

\newcommand{\g}{\bm{g}}
\newcommand{\x}{\bm{x}}
\newcommand{\freegrad}{$\mathsf{FreeGrad}$}

\let\P\undefined

\newcommand{\onedopt}{\text{$1\mathsf{D}$-$\O_{\K^\circ}$}}
\DeclareMathOperator{\P}{\mathbb{P}}

\DeclareMathOperator*{\argmin}{arg\,min} 
\DeclareMathOperator*{\argmax}{arg\,max}

\def\ddefloop#1{\ifx\ddefloop#1\else\ddef{#1}\expandafter\ddefloop\fi}
\def\ddef#1{\expandafter\def\csname bb#1\endcsname{\ensuremath{\mathbb{#1}}}}
\ddefloop ABCDEFGHIJKLMNOPQRSTUVWXYZ\ddefloop

\def\ddefloop#1{\ifx\ddefloop#1\else\ddef{#1}\expandafter\ddefloop\fi}
\def\ddef#1{\expandafter\def\csname fr#1\endcsname{\ensuremath{\mathfrak{#1}}}}
\ddefloop ABCDEFGHIJKLMNOPQRSTUVWXYZ\ddefloop

\def\ddefloop#1{\ifx\ddefloop#1\else\ddef{#1}\expandafter\ddefloop\fi}
\def\ddef#1{\expandafter\def\csname eul#1\endcsname{\ensuremath{\EuScript{#1}}}}
\ddefloop ABCDEFGHIJKLMNOPQRSTUVWXYZ\ddefloop

\def\ddefloop#1{\ifx\ddefloop#1\else\ddef{#1}\expandafter\ddefloop\fi}
\def\ddef#1{\expandafter\def\csname scr#1\endcsname{\ensuremath{\mathscr{#1}}}}
\ddefloop ABCDEFGHIJKLMNOPQRSTUVWXYZ\ddefloop

\def\ddefloop#1{\ifx\ddefloop#1\else\ddef{#1}\expandafter\ddefloop\fi}
\def\ddef#1{\expandafter\def\csname b#1\endcsname{\ensuremath{\mathbf{#1}}}}
\ddefloop ABCDEFGHIJKLMNOPQRSTUVWXYZ\ddefloop
\def\ddef#1{\expandafter\def\csname c#1\endcsname{\ensuremath{\mathcal{#1}}}}
\ddefloop ABCDEFGHIJKLMNOPQRSTUVWXYZ\ddefloop
\def\ddef#1{\expandafter\def\csname h#1\endcsname{\ensuremath{\widehat{#1}}}}
\ddefloop ABCDEFGHIJKLMNOPQRSTUVWXYZ\ddefloop
\def\ddef#1{\expandafter\def\csname hc#1\endcsname{\ensuremath{\widehat{\mathcal{#1}}}}}
\ddefloop ABCDEFGHIJKLMNOPQRSTUVWXYZ\ddefloop
\def\ddef#1{\expandafter\def\csname t#1\endcsname{\ensuremath{\widetilde{#1}}}}
\ddefloop ABCDEFGHIJKLMNOPQRSTUVWXYZ\ddefloop
\def\ddef#1{\expandafter\def\csname tc#1\endcsname{\ensuremath{\widetilde{\mathcal{#1}}}}}
\ddefloop ABCDEFGHIJKLMNOPQRSTUVWXYZ\ddefloop



%% file: macros.tex
\usepackage{booktabs}
\usepackage{relsize}
\usepackage{xspace}
\usepackage{subfigure}
\usepackage{listings}
\definecolor{tableheadcolor}{rgb}{0.8,0.8,1.0}
\definecolor{tablealtcolor}{rgb}{0.9,0.9,0.95}

\definecolor{todocolor}{rgb}{0.8,0.8,1.0}
\definecolor{fixcolor}{rgb}{1,0.8,0.8}
\definecolor{commentcolor}{rgb}{0.8,1.0,0.8}

\newcommand{\textjava}[1]{{\lstset{basicstyle=\ttfamily}\lstinline@#1@}}
\newcommand{\textjavafn}[1]{{\lstset{basicstyle=\footnotesize\ttfamily}\lstinline@#1@}}
\usepackage{setspace}
\usepackage{ifthen}

\long\def\sfootnote[#1]#2{\begingroup%
\def\thefootnote{\fnsymbol{footnote}}\footnote[#1]{#2}\endgroup}
\newcommand{\doi}[1]{\href{http://dx.doi.org/#1}{\nolinkurl{doi:#1}}}
\newcommand{\ignore}[1]{}

\newcommand{\algcomment}[1]{\textcolor{blue!70!black}{\footnotesize{\texttt{\textbf{//
					#1}}}}}

\makeatletter
\newcommand{\neutralize}[1]{\expandafter\let\csname c@#1\endcsname\count@}
\makeatother

\newcommand{\wtilde}[1]{\widetilde{#1}}

\newcommand{\reals}{\mathbb{R}}

\newcommand{\E}{\mathbb{E}}

\newcommand{\inner}[2]{\langle#1,#2\rangle}
\newcommand{\commentout}[1]{}

\renewcommand{\P}{\mathbb{P}}

\newcount\Comments  
\newcommand{\nn}{\nonumber}

\newcommand{\bxi}{\bm{\xi}}

\newcommand{\op}{\mathrm{op}}

\newcommand{\bnu}{\bm{\upnu}}

\newcommand{\dlyap}{\dlyap}